\documentclass[11 pt, reqno]{amsart}
\usepackage[margin=1in]{geometry}
\usepackage{amssymb, amsmath, amsthm}
\usepackage{hyperref}
\usepackage{wasysym}
\usepackage{setspace}
\usepackage{cancel}
\usepackage[all]{xy}
\usepackage{empheq}
\usepackage{extpfeil}

\makeatletter
\newtheorem*{rep@theorem}{\rep@title}
\newcommand{\newreptheorem}[2]{%
\newenvironment{rep#1}[1]{%
 \def\rep@title{#2 \ref{##1}}%
 \begin{rep@theorem}}%
 {\end{rep@theorem}}}
\makeatother

\newtheorem{theorem}{Theorem}[section]
\newreptheorem{theorem}{Theorem}
\newtheorem{lemma}[theorem]{Lemma}
\newreptheorem{lemma}{Lemma}
\newtheorem{proposition}[theorem]{Proposition}
\newtheorem{corollary}[theorem]{Corollary}

\newtheorem{question}[theorem]{Question}

\theoremstyle{definition}
\newtheorem{remark}[theorem]{Remark}
\newtheorem{definition}[theorem]{Definition}
\newtheorem{notation}[theorem]{Notation}
\newtheorem{definitions}[theorem]{Definitions}
\newtheorem{example}[theorem]{Example}
\newtheorem{remarks}[theorem]{Remarks}

\def\beq{\begin{eqnarray*}}
\def\eeq{\end{eqnarray*}}
\def\Q{\mathbb{Q}}

\def\R{\mathbb{R}}
\def\Z{\mathbb{Z}}

\def\incl{\hookrightarrow}
\def\to{\rightarrow}

\def\tF{\widetilde{F}}
\def\tX{\widetilde{X}}
\def\tM{\widetilde{M}}

\def\eps{\varepsilon}

\def\dim{\mathrm{dim}\>}

\def\x{\times}
\def\d{\partial}

\def\phi{\varphi}

\def\Emb{\mathrm{Emb}}

\def\Embbar{\overline{\mathrm{Emb}}}

\def\K{\mathcal{K}}
\def\L{\mathcal{L}}

\def\Map{\mathrm{Map}}

\def\SS{\mathfrak{S}}

\def\LD{\mathcal{LD}}
\def\KD{\mathcal{KD}}

\newcommand{\codim}{\operatorname{codim}}

    \title{Bott--Taubes--Vassiliev cohomology classes by cut-and-paste topology}
    \author{Robin Koytcheff}
    \email{koytcheff@louisiana.edu}
    \thanks{Parts of this work were supported by NSF grant DMS-1004610 and by a PIMS Postdoctoral Fellowship.}
    \address{Department of Mathematics, University of Louisiana, Lafayette, LA, USA 70504}

\begin{document}

\begin{abstract}
Bott and Taubes used integrals over configuration spaces to produce finite-type a.k.a. Vassiliev knot invariants.
Cattaneo, Cotta-Ramusino and Longoni then used these methods together with graph cohomology to construct ``Vassiliev classes'' in the real cohomology of spaces of knots in higher-dimensional Euclidean spaces, as first promised by Kontsevich.   Here we construct integer-valued cohomology classes in spaces of knots and links in $\R^d$ for $d>3$.  We construct such a class for any integer-valued graph cocycle, by the method of gluing compactified configuration spaces.  Our classes form the integer lattice among the previously discovered real cohomology classes.  Thus we obtain nontrivial classes from trivalent graph cocycles.  Our methods generalize to yield mod-$p$ classes out of mod-$p$ graph cocycles, which need not be reductions of classes over the integers.
\end{abstract}

\maketitle

\vspace{2pc}

\section{Introduction}
This work concerns spaces of embeddings of 1-manifolds into Euclidean space.  
The quintessential example of such an embedding space is the space of knots in $\R^3$.  Its path components are isotopy classes of knots, and locally constant functions on this space are precisely knot invariants.  Our main result is a construction of integer-valued cohomology classes in these embedding spaces, which are thus generalizations of invariants of knots and links.  

Our key methods are further developments of the configuration space integrals pioneered by Bott and Taubes \cite{BottTaubes}.  These integrals provided a topological interpretation of knot invariants coming from perturbative Chern--Simons field theory in quantum physics and also generalized the Gauss linking integral from links to knots.  
Parametrizing these integrals by the space of knots, one can view them as integrals along the fiber of a bundle over that space.
Configuration space integrals produce all finite-type knot invariants \cite{DThurstonABThesis, VolicBT}.  
They similarly yield real cohomology classes in spaces of knots, as shown by Cattaneo, Cotta-Ramusino, and Longoni \cite{CCRL-AGT}, as well as in spaces of links, as shown in joint work with Munson and Voli\'c \cite{KMV}.
We call these cohomology classes as \emph{Bott--Taubes--Vassiliev classes} or \emph{configuration space integral classes}.
Similar methods yield invariants of homology 3-spheres \cite{KuperbergThurston, LescopOnKTT} and characteristic classes of homology sphere bundles \cite{Watanabe1, Watanabe2}.  
All of these ideas were outlined in the visionary work of Kontsevich \cite{KontsevichFeynman, KontsevichVassiliev}.  A key ingredient for indexing the cohomology classes in all of these constructions is a cochain complex of graphs, which include but are not limited to trivalent graphs.

In our relatively recent work, we considered the setting of cohomology of spaces of knots and links and carried out ``homotopy-theoretic integration" by replacing integration of differential forms by a Pontrjagin--Thom construction \cite{Rbo, HoBTTower}.  This produced cohomology classes with \emph{arbitrary} coefficients rather than just real coefficients.  Subsequently, we refined this construction by gluing configuration spaces and thus recovered the Milnor triple linking number for long links \cite{HoMilnor3ple, HoBTTower}.  
Our main result here is to recover integer multiples of the Bott--Taubes--Vassiliev classes over $\Z$.

\begin{reptheorem}{MainTheorem}
Let $d >3$, and let $\L^d_m$ denote the space of long $m$-component links in $\R^d$.  For each integer-valued graph cocycle $\gamma$, there is a fiber bundle $F_\gamma \to X_\gamma \to \L^d_m$ over the space of $m$-component long links in $\R^d$  and a map
\[
H^{*}(X_\gamma, \d X_\gamma;\, \Z) 
\longrightarrow H^{* - \dim F_\gamma}(\L^d_m;\, \Z)
\]
such that the images over all cocycles $\gamma$ span an integer lattice among the configuration space integral cohomology classes.
\end{reptheorem}

By ``integer lattice'' above, we mean a free $\Z$-module in the $\R$-vector space of Bott--Taubes--Vassiliev classes.  
Many of the configuration space integral cohomology classes are already known to be nontrivial, and thus the same is true of our $\Z$-valued classes above.  
It is conjectured that the configuration space integral classes span all of the real cohomology of the space of knots.  In fact, the graph complexes are quasi-isomorphic to the first term of the Vassiliev--Sinha spectral sequence \cite[p.~5]{CCRL-AGT}, which is known to compute the rational homology of the space of knots \cite{LTV}.

Furthermore, our construction can in principle produce torsion classes with $\Z$ coefficients and generalize to mod-$p$ cohomology classes which need not be reductions of the integer-valued classes.  
See Section \ref{NontrivialTorsion}.

The main idea of our method is simple.  
Calculating an invariant of knots, links, or 3-manifolds as a sum of configuration space integrals is morally equivalent to calculating the degree of a map, a purely topological feature.  
This equivalence can be made precise by gluing different configuration spaces together.
This approach was pursued in the work of Kuperberg and Thurston on 3-manifold invariants \cite{KuperbergThurston}; the work of Polyak and Viro on the type-2 knot invariant \cite{PolyakViro-Casson}; the work of Poirier on all finite-type invariants of knots and links \cite{PoirierAGT}; and our previous work on the triple linking number for long links \cite{HoMilnor3ple}.  The basic idea also appears in the work of Bott and Taubes \cite{BottTaubes}.

To make this idea precise, we topologically reformulate fiberwise integration via the Serre spectral sequence in singular cohomology.  This is in contrast to the Pontrjagin--Thom constructions we used in our previous works on this subject \cite{Rbo}.  
There we did not explicitly recover the Bott--Taubes--Vassiliev classes (however, see Remark \ref{CoarseQuotient}).
So our present construction is at the level of cohomology rather than spaces or spectra.  The benefit is that we avoid the technical obstacle of endowing the glued configuration space with a smooth structure.  The drawback is that we lose potential classes in generalized cohomology theories with respect to which the configuration space bundles involved are orientable.  
On the other hand, singular cohomology with coefficients in $\Z$ or $\Z/p$ are the main cohomology theories which come to mind in this setting, in addition to de Rham cohomology.

Our methods apply just as well to spaces of closed links as to spaces of long links.
We choose spaces of long links for definiteness, and also because these spaces have a monoid structure that could be useful for further investigation.

Because of Stokes's Theorem, a key step for constructing link invariants or cohomology classes is ensuring that the integral along the boundaries of the configuration spaces is ultimately zero.
The above-mentioned gluing of configuration spaces corresponds to the way in which they combine to yield zero.  There are three ways in which this can happen: 
(1) two non-vanishing integrals along a pair of faces cancel; 
(2) the integral along a face vanishes by symmetry;
(3) the integral along a face vanishes by degeneracy of the form to be integrated.
Topologically treating case (2) is most easily done by using a ``smaller'' compactification of configuration space than the Axelrod--Singer a.k.a.~Fulton--Macpherson compactification.  Specifically, we use a compactification which was considered by Kuperberg and Thurston.  A compactification considered by Poirier appears to be similar to this one.  

Our gluing construction is also similar to that of Kuperberg and Thurston, but different in that we construct a glued space for each graph cocycle, rather than one universal glued space for all trivalent graph cocycles.  
In Section \ref{Results}, we give a more precise restatement of our Main Theorem, which includes a factor of $N!\, 2^{N-1}$ (respectively $N!\, 2^{2N - 2}$) for $d$ odd (respectively even) where $N$ is the maximum number of edges in the graph cocycle.  
This factor resembles the denominator in Poirier's rationality result for invariants of links in $\R^3$ \cite[Proposition 1.11]{PoirierAGT}.
However, because we do not use one universal glued space, we are not forced to work over $\Q$.
Thus the framework where each individual class lives in cohomology with $\Z$ coefficients allows a possible extension to $\Z/p$.  The factors $r_{\mathrm{odd}}$ and $r_{\mathrm{even}}$ do not preclude the possibility of mod-$p$ classes that do not come from classes over $\Z$, but it suggests the importance of ensuring that the construction is sharp when seeking nontrivial such classes.
In addition, we consider cocycles which are not necessarily trivalent, unlike in other authors' work on the case $d=3$, but as in the work \cite{CCRL-AGT} the case $d \geq 4$.

The fact that our methods do not completely work in the case $d=3$ is related to the integrals over the anomalous faces of configuration space.  Whether the integrals vanish along these faces is to our knowledge an open problem, but was studied in detail in work of Poirier \cite{PoirierAGT} and unpublished work of Yang \cite{SWYangArXiv}.  
In the setting of homology 3-sphere invariants, one uses a framing of the manifold to collapse this face to lower dimension.  Thus, the problem for classical knot invariants can be viewed as the failure of a framing of the knot to respect the spherical maps (defined in Section \ref{SphericalMaps}).
In any case, this problem is what stymies our topological construction for $d=3$.

We suspect that our construction is adaptable to the setting of characteristic classes of bundles of homology spheres \cite{Watanabe1, Watanabe2, KuperbergThurston, LescopOnKTT}, roughly generalizing from $S^d$ to homology spheres.  We leave this generalization for potential future work.

\subsection{Organization of the paper}
In Section \ref{Basics}, we provide background and general preliminaries on spaces of knots and links, graphs, orientations of graphs, graph cohomology, gluing manifolds with corners, fiberwise integration, and the pushforward via the Serre spectral sequence.

In Section \ref{Section3}, we define a compactification of configuration spaces via iterated blowups. 
This compactification, which is ``smaller'' than that of Fulton--Macpherson and Axelrod--Singer, previously appeared in the preprint of Kuperberg and Thurston on 3-manifold invariants \cite{KuperbergThurston}.

In Section \ref{Section4}, we define bundles over spaces of links whose fibers are the compactified configuration spaces from Section \ref{Section3}.  We review the construction of real cohomology classes in spaces of knots and links by fiberwise integration of forms, as in the work of Bott and Taubes \cite{BottTaubes} and Cattaneo et al.~ \cite{CCRL-AGT}.  
This work is reviewed for the sake of relating our construction to it, and thus establishing its nontriviality.

In Section \ref{GluingSection}, we describe the gluings of the compactified configuration space bundles defined in Section \ref{Section3}.  This construction involves gluing, folding, and collapsing various codimension-1 faces, while leaving some faces to the relative locus.  Using general results from Section \ref{Basics} on spaces thus constructed, we establish that the glued space has fundamental class, as well as a map to an appropriate quotient of a product of spheres.

In Section \ref{Results}, we prove our Main Theorem, Theorem \ref{MainTheorem}.  
Using the glued bundle from Section \ref{GluingSection}, the Serre spectral sequence, and a lemma from Section \ref{Basics}, we construct cohomology classes which form an integer lattice among the classes obtained by integration of forms.  We discuss some examples, like the type-2 invariant for knots and the triple linking number for long links.  

The last two Sections contain only minor results and mainly concern plans for future work.  
In Section \ref{NontrivialTorsion}, we consider the possibility of torsion classes that are reductions of integer-valued classes.  These come from Turchin's calculations in graph complexes.  We sketch a method for potentially detecting their nontriviality due to Pelatt and Sinha.  This involves constructing dual homology classes out of bracket expressions and associated resolutions of singular knots.

In Section \ref{ModPSection}, we consider coefficients in $\Z/p$.  Over $\Z/2$, we describe a slight generalization of our method that includes mod-2 reductions of our classes over $\Z$, as well as classes that are potentially not such reductions.

\subsection{Acknowledgments}
We thank the referees for many useful comments and suggestions.

\tableofcontents

\section{Basic definitions, conventions, background, and general preliminaries}
\label{Basics}

Here we provide basic definitions and background material on spaces of links (Section \ref{SpacesOfLinks}), graphs and their orientations (Section \ref{Graphs}), the graph cochain complex (Section \ref{GraphComplex}), gluing manifolds with corners (Section \ref{GluingMfds}) and fiberwise integration via the Serre spectral sequence (Section \ref{GeneralIntegration}).

\subsection{Spaces of knots and links}
\label{SpacesOfLinks}
Let $\L^d_m$ denote the space of long $m$-component links in $\R^d$.  By the space of long links, we mean the space of embeddings of $\coprod_m \R$ into $\R^d$ with fixed linear behavior outside $\coprod_m [-1,1]$, equipped with the Whitney $C^\infty$ topology.  For technical reasons, the fixed linear behavior is required to have distinct directions for distinct strands.  One can view $\L^d_m$ as an infinite-dimensional manifold, for example by defining smooth maps $M \to \L^d_m$ from finite-dimensional manifolds $M$ in a fairly obvious way.  This allows one to consider the de Rham cohomology of $\L^d_m$.
See \cite[Section 2]{KMV} for details on these points.  
We will also write $\K^d$ for the space of long knots $\L^d_1$.
This paper concerns cohomology classes in spaces of knots and links.  
Our work below applies with minimal modification to spaces of closed links $\Emb(\coprod_m S^1, \R^d)$.  We choose long links for the sake of definiteness.

\subsection{Graphs}
\label{Graphs}
Certain graphs are crucial for our constructions.  
Fix an oriented 1-manifold $L$ without boundary and with finitely many components.  We mainly consider the case where $L$ is a disjoint union of finitely many copies of the real line.\footnote{Ultimately, in Section \ref{SuspendingGraphs}, we will modify a graph on $L$ as defined below to get a graph on a space $L'$ that is a 1-manifold except at one singular point $\infty$.}
We allow the case $L=\varnothing$, even though the empty set may not qualify as a 1-manifold.
The main idea of the relevant graphs, as well as the differential defined in Section \ref{GraphComplex}, goes back to Kontsevich \cite[p.~11]{KontsevichFeynman}.  The same type of graphs have been used by other authors in various contexts, e.g.~\cite{BarNatanTopology, CCRL-AGT, Watanabe2, KMV}.

\begin{definitions}
\label{GraphsDefinition}
An \emph{unoriented (link) graph} or \emph{unoriented (link) diagram $\Gamma$ on $L$} consists of 
\begin{itemize}
\item
$V(\Gamma)$, a (finite) set of \emph{vertices}, which is partitioned as $V(\Gamma) = V_{seg}(\Gamma) \sqcup V_{free}(\Gamma)$ into \emph{segment vertices} which lie on $L$ and the remaining \emph{free vertices}; and
\item
$E(\Gamma)$, a (finite) set of \emph{edges}, which abstractly are unordered pairs of vertices; we call the pair of vertices the \emph{endpoints} of an edge, and we call a pair consisting of an edge and one of its endpoints an \emph{edge-end}.
\item
$A(\Gamma)$, a (finite) set of \emph{arcs}, where each arc is part of $L$ between two segment vertices.  An arc is not considered an edge, but like an edge, it has \emph{endpoints} and \emph{ends}.
\end{itemize}
We require the valence of each vertex to be at least 3, where ends of edges and arcs alike count towards valence; 
an edge may not join a free vertex to itself; and if $L \neq \varnothing$, we require that each component of $\Gamma$ be connected to $L$.  
A self-loop on a vertex is allowed as an edge, and like other edges it has two edge-ends.
Multiple edges (or edge(s) and an arc) may join the same pair of vertices.\footnote{In Section \ref{GraphComplex} we will quotient by graphs with self-loops on free vertices or multiple edges.}  
\begin{itemize}
\item A \emph{chord diagram} is a graph with no free vertices.
\item A \emph{subgraph} $\Gamma'$ of a graph $\Gamma$ will always be determined by a subset $V' \subset V(\Gamma)$, where the edges and arcs in $\Gamma'$ are those in $\Gamma$ between vertices in $V'$.
A subgraph is not required to satisfy the conditions on valence and being connected to $L$.
\item An \emph{isomorphism} of graphs is a homeomorphism of the associated topological spaces which preserves the components of $L$ and their orientations. \qed
\end{itemize}
\end{definitions}

We consider two ways in which graphs on $L$ give rise to ordinary finite graphs, i.e.~1-dimensional CW-complexes.
Let $U(\Gamma)$ denote the graph obtained by forgetting $L$.  Then $U(\Gamma)$ is ``at least unitrivalent'' in the sense that each free vertex has valence $\geq 3$ and each segment vertex has valence $\geq 1$.
Let $T(\Gamma)$ denote the graph obtained by regarding arcs as edges and forgetting the complement in $L$ of all segment vertices and arcs between them.  
(The pieces we forget here are (a) components in $L$ with no segment vertices and (b) half-open rays in $L$ from a segment vertex to $\pm \infty$.)
Thus if $L$ is compact, then $T(\Gamma)$ is an ``at least trivalent'' graph. 
In Section \ref{SuspendingGraphs}, the suspension of graphs from an additional $\infty$ will essentially reduce our considerations to the setting where $L$ is compact.
Finally, if $L=\varnothing$, then $U(\Gamma)=T(\Gamma)=\Gamma$.

For now, we will use ``graph'' to mean ``unoriented graph.''  Below, we will orient our graphs by adding certain decorations, and then we will only omit the adjectives ``unoriented'' and ``oriented'' when the meaning is clear from the context.  

\begin{definitions}
\label{BiconnectedBlocks}
Call a graph $\Gamma$ \emph{connected} if $T(\Gamma)$ is connected, i.e., for all vertices $u,v$ there is a path of edges and arcs joining $u$ to $v$; otherwise call $\Gamma$ \emph{disconnected}.
A vertex of a graph $\Gamma$ is called a \emph{cut vertex} if removing it and all incident edges and arcs in $T(\Gamma)$ produces a disconnected nonempty graph.
A graph $\Gamma$ is \emph{biconnected} if it is connected and has no cut vertices.  
Call a maximal biconnected subgraph of $\Gamma$ a \emph{block} of $\Gamma$.  
\end{definitions}

Thus a single edge or arc is biconnected, but a single vertex is not.  
Some authors use either ``vertex-2-connected'' or ``2-connected'' instead of ``biconnected.''  
We prefer ``block'' to the lengthier term ``biconnected component.''  Kuperberg and Thurston use the term ``lobe'' instead of ``block.''  To our knowledge, they are the only authors who have previously considered the biconnectivity of these types of graphs.

\begin{lemma}
\label{IntersectionOfBlocks}
The intersection of any two distinct blocks $\Gamma_1, \Gamma_2$ of $\Gamma$ is at most one vertex, which is necessarily a cut vertex of $\Gamma$.
\end{lemma}
\begin{proof}
If $\Gamma_1 \neq \Gamma_2$ intersect in more than one vertex, then their union is biconnected, contradicting maximality.  If $\Gamma_1, \Gamma_2$ intersect in a vertex $u$ that is not a cut vertex, then the removal of $u$ leaves $\Gamma$ connected.  Thus we can find a path from a vertex $v$ in $\Gamma_1\setminus \Gamma_2$ to a vertex $w$ in $\Gamma_2 \setminus \Gamma_1$ which does not pass through $u$.  Adjoining this path to $\Gamma_1 \cup \Gamma_2$ creates a biconnected graph, again contradicting maximality.
\end{proof}

\begin{definition}
Define $\Upsilon(\Gamma)$, the \emph{block-cut forest of $\Gamma$}, to be the graph with a vertex for every block $b$ of $\Gamma$ and every cut vertex $c$ of $\Gamma$, and an edge between $b$ and $c$ when the cut vertex $c$ is contained in $b$:
\[
V(\Upsilon(\Gamma)) = \{\text{blocks $b$ of } \Gamma\} \cup \{\text{cut-vertices $c$ of } \Gamma\}  \qquad
E(\Upsilon(\Gamma)) \subset \{ \text{(block $b$, cut-vertex $c$)}  \}
\]
\end{definition}

\begin{lemma}
\label{BlockCutTreeBasics} \ 
\begin{itemize}
\item[(a)] 
The graph $\Upsilon(\Gamma)$ is a forest where the leaves are labeled only by blocks, not cut-vertices.  If $\Gamma$ is connected, then $\Upsilon(\Gamma)$ is a tree.
\item[(b)]  
If $\Gamma_1,\dots,\Gamma_j$ are the blocks of $\Gamma$, and $CV(\Gamma)$ is the set of cut-vertices of $\Gamma$, then 
\[
\sum_{i=1}^j |V(\Gamma_i)| = |V(\Gamma)| + |E(\Upsilon(\Gamma))| - |CV(\Gamma)|.
\]
\end{itemize}
\end{lemma}
\begin{proof}
For (a), $\Upsilon(\Gamma)$ has no cycles, for otherwise we contradict the maximality of the blocks.  
Next, notice that for any pair of vertices $v,w$ in $\Gamma$, a path in $\Upsilon(\Gamma)$ between blocks  $b_v$ and $b_w$ containing $v$ and $w$ can be represented by a path in $\Gamma$ from $v$ to $w$.  Indeed, for each edge in the path in $\Upsilon(\Gamma)$ joining a block $b$ and a cut vertex $c$, one can choose a path within $b$ between $c$ and any vertex in $b$.  Then if a cut vertex were a leaf in $\Upsilon(\Gamma)$, its removal would leave $\Upsilon(\Gamma)$ connected, contradicting that its removal disconnects $\Gamma$.  
A path in $\Gamma$ from $v$ to $w$ gives rise to a path in $\Upsilon(\Gamma)$ between their respective blocks $b_v$ and $b_w$, so $\Upsilon(\Gamma)$ is connected if $\Gamma$ is.

For (b), the left-hand side counts vertices of $\Gamma$, except that each cut-vertex is necessarily counted multiple times, by part (a).  Specifically, each cut-vertex $v$ is counted as many times as $v$ appears in an edge of $\Upsilon(\Gamma)$.  (So if a cut-vertex $v$ appears in $n$ edges, it is counted $n-1$ extra times.)
\end{proof}

The following easily verified lemma will be useful in Section \ref{HiddenFaceInvolution}:
\begin{lemma}
\label{HiddenCornersLemma}
Suppose $v$ is a vertex in a subgraph $\Gamma' \subset \Gamma$ which is bivalent in $\Gamma'$, 
joined to vertices $u$ and $w$ in $\Gamma'$.  If $\Gamma'' \subset \Gamma'$ is biconnected and contains $v$, then either $\Gamma''$ is one of the two edges incident to $v$ or $\Gamma''$ contains both $u$ and $w$.
\qed 
\end{lemma}

\begin{definition}
\label{LabeledGraphs}
A \emph{labeled graph} or \emph{labeled (link) diagram} is a graph $\Gamma$ together with a \emph{labeling}, defined as follows.  If $d$ is odd, a labeling consists of an ordering of the vertices, an orientation on each edge, and additionally an ordering of the two edge-ends of each self-loop on a segment vertex.  
If $d$ is even, a labeling consists of an ordering of the segment vertices and an ordering of the edges (including self-loops on segment vertices).
\end{definition}

For example, line (\ref{Type2Graphs}) in Section \ref{VanishingAndGraphCohomology} below is a linear combination of labeled diagrams.
In the next subsection, when we consider linear combinations of diagrams, two diagrams which differ only by their labelings will be set equal up to a sign.

\subsection{The cochain complex of oriented graphs}
\label{GraphComplex}
The graphs defined above can be organized into a cochain complex, as first observed by Kontsevich \cite{KontsevichFeynman}.  There and in \cite{CCRL-AGT}, the complex was defined over $\Q$ and $\R$ respectively, but one can use coefficients in any ring.\footnote{An abelian group is sufficient for our purposes here, though a ring is useful if one wants to consider multiplicative structures such as the shuffle product of graphs.}
Fix a ring $R$.  We will mainly consider the cases where $R$ is $\R, \Q, \Z,$ or $\Z/p$.  

\begin{definition}
\label{GraphComplexDef}
Let $\LD$ denote the free $R$-module on labeled link diagrams, modulo the following \emph{orientation relations}:
\begin{itemize}
\item
For odd $d$, if $\Gamma$ and $\Gamma'$ differ only by a permutation $\sigma$ of the vertex labels, a reversal of the orientations on $i$ edges, and $j$ transpositions of the orderings of the self-loop edge-ends, then $\Gamma \sim {\mathrm{sign}(\sigma)} (-1)^{i + j} \Gamma'$.
\item
For even $d$, if $\Gamma$ and $\Gamma'$ differ only by a permutation $\sigma$ of the vertex labels and a permutation of the edge-labels, then $\Gamma \sim {\mathrm{sign}(\sigma)\mathrm{sign}(\tau)} \Gamma'$.
\end{itemize}
We impose two further relations:
\begin{itemize}
\item
If $\Gamma$ has a pair of vertices joined by multiple (non-arc) edges, then $\Gamma \sim 0$.
\item
If $\Gamma$ has a free vertex with a self-loop, then $\Gamma \sim 0$.
\end{itemize}
\end{definition}

Call the resulting equivalence class of a labeling an \emph{orientation} and the resulting equivalence class of graph an \emph{oriented graph} or \emph{oriented diagram}.  For a labeled or oriented graph $\Gamma$, let $|\Gamma|$ be the underlying unoriented graph.  For two oriented graphs, an isomorphism of their underlying unoriented graphs may be \emph{orientation-preserving} or \emph{orientation-reversing}.

\begin{remark}
There are other equivalent definitions of orientations on graphs.  For example, for odd $d$, one can use a cyclic ordering of the edge-ends at each vertex (e.g.~induced by a planar embedding of the graph) rather than vertex-labelings and edge-orientations.  Bar-Natan takes this approach in \cite{BarNatanTopology}.  The definition above is however more amenable to defining configuration space integrals.  Kuperberg and Thurston \cite[Section 3.1]{KuperbergThurston} provide a thorough discussion of orientations on graphs in this context.
\end{remark}

\begin{definitions} 
\label{DefectOrderDef} \ 
\begin{itemize}
\item
Define the \emph{defect} of a diagram $\Gamma$ to be $2|E(\Gamma)| - |V_{seg}(\Gamma)| - 3|V_{free}(\Gamma)|$.  
\item
Define the \emph{order} of $\Gamma$ to be $|E(\Gamma)| - |V_{free}(\Gamma)|$.
\item
Let $\LD^{k,n}$ denote the submodule of diagrams of defect $k$ and order $n$, so that $\LD = \bigoplus_{k,n} \LD^{k,n}$. 
\end{itemize}
\end{definitions}

Cattaneo et al.~use the term ``degree'' instead of ``defect.''  As in \cite{KMV}, we reserve the term ``degree'' for cohomological degree.
The diagrams with defect 0 are precisely those which are (uni)trivalent, in the sense that each free vertex is trivalent and each segment vertex has one edge and two arcs emanating from it.  It is precisely the defect-0 cocycles which index finite-type invariants of knots and links in $\R^3$.

The coboundary operator $\delta: \LD^{k,n} \to \LD^{k+1,n}$ is defined to encode which configuration space integrals yield closed forms. (See Section \ref{VanishingAndGraphCohomology}.)  
We first need one further construction on graphs.  
If $e$ is an edge or arc in a labeled graph $\Gamma$ joining vertices $i,j$, the \emph{contraction} $\Gamma/e$ of $e$ has the usual quotient of CW-complexes as its underlying unoriented graph. 
The vertex labels on $\Gamma/e$ are given by lowering by 1 those vertex labels greater than $\max(i,j)$ and by assigning $\min(i,j)$ to the vertex that is image of $e$.  For $d$ odd, if the contraction of $e$ introduces a self-loop, then the endpoints of $e$ are necessarily segment vertices, and the order of edge-ends in the self-loop is determined by the orientation of $L$.
For $d$ even, the labels on the edges are given by lowering by 1 those labels greater than that of $e$.

We now define $\delta$ on each graph $\Gamma$ and extend it to $\bigoplus_{k,n} \LD^{k,n}$ by linearity.  On a graph $\Gamma$, 
it is defined as a signed sum of edge contractions
\[
\delta \Gamma := \sum_{e} \eps(e) \Gamma/e
\]
where the sum is taken over all arcs $e$ and all edges $e$ that are not chords or self-loops.  We first define the sign $\eps(e)$ for $d$ odd.  Suppose $e$ is an edge or arc with endpoints $i < j$, and set
\begin{equation}
\label{OddSigns}
\eps(e):= 
\left\{
\begin{array}{ll}
(-1)^{j-1} & \text{ if } e=(i \to j) \\
(-1)^j & \text{ if } e=(i \leftarrow j)
\end{array}
\right.
\end{equation}
If $d$ is even and $e$ is an \emph{arc} with endpoints $i < j$, define $\eps(e)$ as above, where the arc orientation comes from the orientation of $L$.
If $d$ is even and $e$ is an \emph{edge}, set 
\begin{equation}
\label{EvenSigns}
\eps(e) = 
(-1)^{e+|V_{seg}(\Gamma)|}
\end{equation}
where by abuse of notation $e$ also denotes the label on this edge.

The following theorem was proven by Cattaneo et al.~over $\R$, but the proof shows that $\delta^2=0$ on each element of a basis (given by a graph from each unoriented isomorphism class, with some choice of orientation). 
So the statement holds over any ring.  Similar constructions appear in other contexts in topology, all of which were outlined in the work of Kontsevich \cite{KontsevichFeynman}.
\begin{theorem}[\cite{CCRL-AGT}]
The sequence $(\bigoplus_{k,n} \LD^{k,n}, \delta)$ is a complex, i.e., $\delta^2=0$.
\end{theorem}

The combinatorial input for our construction will be a graph cocycle $\gamma$, which is a linear combination of oriented diagrams, but can be represented (in many ways) by a linear combination of labeled diagrams.

\subsection{Gluing manifolds with corners}
\label{GluingMfds}
Our construction will use a bundle where the fiber is a quotient of a manifold with corners.  The quotient will involve a limited class of operations.  We now check that such a quotient has a fundamental class.

Let $\R_+:=[0,\infty)$.  A \emph{$k$-dimensional manifold with corners} $X$ is a space where each point $x$ has a neighborhood diffeomorphic to $\R^{k-\ell} \x \R_+^{\ell}$ for some $\ell$ with $0 \leq \ell \leq k$.  We call $\ell$ the \emph{codimension} $c(x)$ of $x$, and we call the closure of a component where $c(x)=\ell$ a \emph{corner of codimension} $\ell$.  We also use the synonymous term \emph{face of codimension} $\ell$, especially if $\ell=1$.  
A corner of a manifold with corners is again a manifold with corners.

\begin{proposition}
\label{GluedMfdFundClass}
Let $\tM$ be a (possibly disconnected) compact, oriented $k$-dimensional manifold with corners, with boundary $\d \tM$.  Let $M$ be a connected quotient by a map $\psi: \tM \to M$ which can be obtained by a finite sequence of the following types of operations:
\begin{enumerate}
\item[(a)]
gluing a pair of codimension-1 faces by an orientation-reversing diffeomorphism,
\item[(b)]
folding a codimension-1 face $\widetilde{S}$ by a finite-order, orientation-reversing diffeomorphism 
such that 
$H_{k-1}(\psi(\widetilde{S}), \psi(\d \widetilde{S})) \cong \Z$ 
and $H_{i}(\psi(\widetilde{S}), \psi(\d \widetilde{S})) =0$ for all $i \geq k$, and
\item[(c)]
taking the quotient of a corner $\widetilde{S}$ such that
$H_{i}(\psi(\widetilde{S}), \psi(\d \widetilde{S})) =0$ for all $i \geq k-1$.
\end{enumerate}
Let $\d M$ be the union of the images of all codimension-1 faces whose interiors incur no identification.  Then 
$H_i(M, \d M; \,\Z) =0$ for all $i>k$, and $H_k(M, \d M; \,\Z) \cong \Z$.  Moreover, if the chain $a \in C_k(\tM)$ represents the fundamental class $[\tM, \d \tM] \in H_k(\tM, \d \tM; \,\Z)$, then $[\psi_*(a)]$ is a generator of $H_k(M, \d M; \,\Z)$.
\end{proposition}

The universal coefficient theorem then implies $H^i(M, \d M; \,\Z) =0$ for all $i>k$, and that $H^k(M, \d M; \,\Z) \cong \Z$.
The last statement of the Proposition implies that the generator $[\psi_*(a)] \in H_k(M, \d M; \,\Z)$ maps to $[\tM, \d \tM] \in H_k(\tM, \d \tM; \,\Z)$ under the map induced by the quotient $M/ \d M \to \tM/ \d \tM$, which is defined because $\d M \subset \psi(\d \tM)$.  Hence this induced map on homology is injective.  

\begin{definition}
\label{NicelyGluedMfdDef}
Call a pair $(M, \d M)$ consisting of a quotient $M$ and a subspace $\d M \subset M$ satisfying the hypotheses of Proposition \ref{GluedMfdFundClass} a \emph{nicely glued manifold with corners}.  
\end{definition}

Note that part of the definition is that $M$ is connected.  We will use this notion in Section \ref{GeneralIntegration}.

\begin{proof}[Proof of Proposition \ref{GluedMfdFundClass}]
First suppose that $M$ is the quotient of $\tM$ by a single operation of type (a), i.e.~a gluing of a pair of codimension-1 faces.  
Then $\tM$ must have either one or two components.
Let $S$ be the image of the glued pair of faces in $M$, so $\d M$ is the image of the remaining faces in $\tM$.  Part of the long exact sequence of the triple $M \supset (\d M \cup S) \supset \d M$ is 
\begin{equation}
\label{LES3pleGlue1}
H_k(\d M \cup S, \d M) \to H_k(M, \d M) \to H_k (M, \d M \cup S) \to H_{k-1}(\d M \cup S, \d M)
\end{equation}
where we have omitted $\Z$ coefficients from the notation.  
We replace the pair $(\d M \cup S, \d M)$ by $(S, \d S)$ since their quotients are homeomorphic.  Similarly we  replace $(M, \d M \cup S)$ by $(\tM, \d \tM)$.
Since $S$ is a $(k-1)$-manifold with boundary, this sequence becomes
\begin{equation}
\label{LES3pleGlue}
0={H_k(S, \d S)} \to H_k(M, \d M) \to H_k (\tM, \d \tM) \to H_{k-1}(S, \d S) \cong \Z.
\end{equation}
Thus $H_k(M, \d M)$ is the kernel of $H_k (M, \d M \cup S) \to H_{k-1}(\d M \cup S, \d M)$ in \eqref{LES3pleGlue1}.  We analyze the map using the commutative diagram below, by going through the bottom row and then to the rightmost group. 
\begin{equation}
\label{LES3pleSquare}
\xymatrix{
H_k(M, \d M \cup S) \ar[r] & H_{k-1}(\d M \cup S, \d M) & \ar[l]_-\cong H_{k-1}(S, \d S) \cong \Z \\
H_k(\tM, \d \tM) \ar[u]_-\cong^-{\psi_*} \ar[r]^-\d & H_{k-1}(\d \tM) \ar[u]_-{\psi_*}& 
}
\end{equation}
Since $\d [\tM, \d \tM] = [\d \tM]$, the composition in the top row sends the class corresponding to $[\tM, \d \tM]$ to $[S, \d S] - [S, \d S]=0$, because the gluing is by an orientation-reversing map.  
If $\tM$ has one component, $[\tM, \d \tM]$ generates the domain, so the kernel is indeed $\Z$.
If $\tM$ has two components, then the image of the  fundamental class of either component is $\pm [S, \d S]$.  So in this case too, the kernel is generated by $[\tM, \d \tM]$, hence isomorphic to $\Z$.  This proves all the statements about $H_k$ for an operation of type (a).  The homology groups vanish for $i>k$ by the sequence \eqref{LES3pleGlue} with $k$ replaced by $i$. 

Next suppose that $M$ is the quotient of $\tM$ by a single operation of type (b), i.e., the folding of a single face by an orientation-reversing diffeomorphism.  Then $\tM$ must be connected.  
Let $\widetilde{S}$ be the face in $\tM$ that is folded, let $S := \psi(\widetilde{S})$, and let $\d S := \psi(\d \widetilde{S})$.
Then $\d M$ is the image of the remaining faces in $\tM$.  
As with a type (a) operation, the long exact sequence for this triple gives us the sequence \eqref{LES3pleGlue},
though in this case $\tM$ is connected, so $H_k(\tM, \d \tM) \cong \Z$.  So it suffices to show that the image of $[\tM, \d \tM]$ in the right-hand group in \eqref{LES3pleSquare} is zero.  In turn, it suffices to show that $\psi_*[\d \tM]=0$.  We unravel the right-hand half of \eqref{LES3pleSquare} into the following commutative diagram, which shows this, and which we justify below.
\begin{equation}
\label{LES3pleSquare2}
\xymatrix{
H_{k-1}(\d M \cup S) \ar[r]  & H_{k-1}(\d M \cup S, \d M) & \ar[l]_-\cong H_{k-1}(S, \d S)  \\
H_{k-1}(\d \tM) \ar[u]_-{\psi_*} \ar[r]^-\cong &  H_{k-1}(\d \tM, \d \tM - \widetilde{S})\ar[u]_-{\psi_*} & \ar[l]_-\cong H_{k-1}(\widetilde{S}, \d \widetilde{S}) \ar[u]_-0
}
\end{equation}
The lower-left horizontal map is an isomorphism because $\d \tM - \widetilde{S}$ is a $(k-1)$-manifold with nonempty boundary.  
We check that the right-hand vertical map is indeed zero.  The folding diffeomorphism reverses orientation, so its order must be an even number, say $2j$.  Then we can take a non-fixed point $x \in S$ and a small disk neighborhood $U$ whose preimage in $\widetilde{S}$ is a disjoint union $\coprod_{i=1}^{2j} U_i$, with each $U_i \cong U$.  By excision, the image of the right-hand map is the image of the fundamental class of $\widetilde{S}$ under $\bigoplus_{i=1}^{2j} H_{k-1}(U_i, U_i - x_i) \to H_{k-1}(U, U-x)$.  Because the diffeomorphism is orientation-reversing, this image is $j-j=0 \in \Z$.  We thus reach the desired conclusions about $H_k$, and the conclusion about $H_i$, $i>k$ is an easy application of the sequence \eqref{LES3pleGlue}, as in case (a).

Finally, suppose  $M$ is the quotient of $\tM$ by a single operation of type (c), i.e., the collapse of a single corner $\widetilde{S}$ to a manifold of strictly lower dimension.  Then $\tM$ is connected.  Let $S:=\psi(\widetilde{S})$, and let $\d S= \psi(\d \widetilde{S})$.
The exact sequence of the triple $M \supset (\d M \cup S) \supset S$ gives a sequence of four terms similar to \eqref{LES3pleGlue}.  
By hypothesis, both outer terms are 0 and the third term is $\Z$.  This completes the proof in this case for $H_k$, and $H_i$, $i>k$, is even easier.

The proof is complete if $\tM \twoheadrightarrow M$ involves only one operation.  For an arbitrary finite number of operations, we proceed by induction on the number $p$ of operations, restricting to a subset of components at each stage if necessary.  The base case is $M$ itself, and the induction step proceeds as above, with the result of $p -1$ operations playing the role of $M$, and the result of $p$ operations playing the role of $\tM$.  The same arguments apply.
\end{proof}

\subsection{Fiberwise integration and pushforward in singular cohomology}
\label{GeneralIntegration}
Let $\pi: X \to B$ be an orientable fiber bundle, where the base $B$ is a manifold.  Assume the fiber $F$ is a compact, connected $k$-dimensional manifold, possibly with boundary, or even corners.
Assume also that the fiber $F$ is oriented, by which we mean that the bundle is oriented.  At the level of de Rham cochains, one can always define a map $\int_F: C^{p+k}_{dR}(X) \to C^p_{dR}(B)$ called \emph{integration along the fiber} or \emph{pushforward}.  A notable property is that in general, this pushforward might not be a chain map: Stokes's Theorem implies that 
\[
d \int_F \alpha = \int_F {d \alpha}  \pm \int_{\d F} \alpha |_{\d X}.
\]
Thus it will be a chain map if, for example, $\alpha \in C^{p+k}(X, \d X)$, or if we consider only $\alpha$ in some subcomplex for which $\int_{\d F} \alpha |_{\d X} =0$.  
For more details, see the books of Bott and Tu \cite{BottTu} and Greub, Halperin, and Vanstone \cite{GHV} or the paper of Bott and Taubes \cite[Appendix]{BottTaubes}.
We now describe an approach via singular cohomology, in which context such a map may also be called a \emph{pushforward map}, as well as a \emph{shriek map}, \emph{umkehr map}, \emph{Gysin homomorphism}, or \emph{pre-transfer map}.

Ordinary integration over a manifold $M$ is given by pairing with the fundamental class $[M]$.  Similarly, integration over the fiber of a bundle $F \to X \to B$ can be described using the fundamental class $[F]$ of the fiber, though to proceed in sufficient generality, one needs more machinery than in the case where $B$ is a point.  We will use the Serre spectral sequence, as explained by Morita \cite[Section 4.2.3]{MoritaGeomCharClasses} or Bott and Tu \cite[pp.~177-179]{BottTu}.

Let $\pi: X \to B$ be a fiber bundle as above.  The assumption that the bundle is orientable implies that $\pi_1(B, b)$ acts trivially on $H^k(F, \d F) \cong \Z$ for any choice of basepoint $b$.  
Consider the Serre spectral sequence for this bundle in cohomology with integer coefficients.
For simplicity, first consider the case where $\d F =\varnothing$.
Then $H^k(F; \, \Z) \overset{\cong}{\to} \Z$, where the map can be given by pairing with the fundamental class $[F] \in H_k(F;\,\Z)$ determined by the orientation, and $H^i(F; \, \Z) =0$ for all $i >k$.    This immediately implies that $E_\infty^{p, k}$ is a subspace of  $E_2^{p, k}$ for any $p\geq 0$.  Furthermore, $E^{p,q}_\infty$ in general is isomorphic to the filtration quotient $\mathcal{F}^{p+q}_{p} / \mathcal{F}^{p+q}_{p+1}$, where $\mathcal{F}^{p+q}_*$ is a decreasing filtration of $H^{p+q}(X)$, induced by the increasing filtration of $B$ by skeleta.  
Since the fiber $F$ is $k$-dimensional, the $E_\infty$ page is zero above the $k$-th horizontal line, and $H^{p+k}(X)$ is thus equal to $\mathcal{F}^{p+k}_{p}$.  Let $\pi^!$ be the composition
\begin{equation}
\label{IntegrationBySSS}
H^{p+k}(X;\, \Z) = \mathcal{F}^{p+k}_{p} \twoheadrightarrow
\mathcal{F}^{p+k}_{p} / \mathcal{F}^{p+k}_{p+1} \cong
E_\infty^{p, k} \incl E_2^{p, k}
\cong  H^{p}(B; \, H^k(F;\,\Z)) 
\overset{\cong}{\to} H^{p}(B; \, \Z).
\end{equation}

If $F$ has nonempty boundary $\d F$, let $\d X$ be the corresponding subspace of $X$.
Then we have an analogous map $\pi^!: H^{p+k}(X, \d X; \, \Z) \to H^{p}(B; \, \Z)$ by replacing $[F] \in H_k(F)$ by $[F, \d F] \in H_k(F, \d F)$.
By the universal coefficient theorem, $H^*(F, \d F)$ satisfies the same hypotheses with $\R$ coefficients as with $\Z$ coefficients, so there is also a map  $H^{p+k}(X, \d X; \, \R) \to H^{p}(B; \, \R)$ on singular cohomology with $\R$ coefficients.  In addition, the same is true for $\Z/p$ coefficients.

\begin{lemma}
\label{IntegrationAndSSS} \ 
Suppose $F\to X\to B$ is an orientable fiber bundle with the fiber $F$ a compact, connected, oriented, $k$-dimensional manifold with corners.
Then the pushforward map $\pi^!$ from the Serre spectral sequence agrees with fiberwise integration $\int_F$.  More precisely, if $\int$ denotes de Rham isomorphisms on both $B$ and $X$, then $\pi^! \circ \int = \int \circ \int_F$ as maps
\[
H^{p+k}_{dR}(X,  \d X) \to H^p(B;\, \R).
\]
\end{lemma}

A sketch of one proof is as follows.  Represent a cohomology class by a $(p+k)$-form, and write the form in local coordinates $dx_I \, dt_J$.  Observe that taking the quotient $\mathcal{F}^{p+k}_p \twoheadrightarrow \mathcal{F}^{p+k}_p / \mathcal{F}^{p+k}_{p+1}$ corresponds precisely to forgetting terms where $|J|<k$.  A term with $|J|=k$ represents a class in $H^p(B; \, H^k(F, \d F))$, and the isomorphism to $H^p(B)$ agrees with the one given by integration over $F$.

If the fiber $F$ has multiple components, Lemma \ref{IntegrationAndSSS} holds with almost the same proof.  Instead of an isomorphism $H^k(F, \d F) \cong R$, the fundamental class induces a surjection $H^k(F, \d F) \to R$, where $R = \R, \Z$, etc.~is the coefficient ring.  As in the connected case, the map comes from the orientation of $F$, i.e.~it is given by pairing with the chosen fundamental class $[F, \d F]$.

For the proof of the Main Theorem, we will use Lemma \ref{GluedIntegrationAndSSS} below, which is an analogue of Lemma \ref{IntegrationAndSSS} for the case where the fiber $F$ is a nicely glued manifold with corners.  We will be able to integrate over such $F$, but a de Rham map on simplices will not be readily apparent.  The proof  of this analogue will essentially use singular chains rather than de Rham cochains.  
Since Lemma \ref{IntegrationAndSSS} is just a warmup for Lemma \ref{GluedIntegrationAndSSS} and will not be used directly, we refrain from further details on the proof sketched above. 

\subsubsection{Fiberwise integration for a glued manifold}
\label{GluedMfdIntegration}
Let $\widetilde{F}\to \widetilde{X}\to B$ be an orientable fiber bundle with $\widetilde{F}$ a compact, oriented (possibly disconnected) $k$-dimensional manifold with corners.
Let  $\psi: \widetilde{X} \twoheadrightarrow X$ be a fiberwise quotient, yielding a bundle $F \to X \to B$.
For our purposes, it suffices to consider the situation where there is a subbundle $\d F \to \d X \to B$ such that $(F, \d F)$ is a {nicely glued manifold with corners}, as in Definition \ref{NicelyGluedMfdDef}.

Although $F$ is not obviously a smooth manifold with corners, we can make make sense of smooth functions, differential forms and integration on $F$.  We use an idea appearing for example in works on stratifolds \cite{Ewald-PhD, Kreck-Book}, though we will make no use of stratifolds themselves.  First define smooth functions $F \to \R$ as those continuous functions for which the composition $\widetilde{F} \overset{\psi}{\twoheadrightarrow} F \to \R$ is smooth.  A differential $q$-form on $F$ is then defined as an element that is locally a linear combination of terms $g\ df_1 \dots df_q$, where $g, f_1, \dots, f_q$ are smooth functions on an open subset $U \subset F$.  
One can then define de Rham cohomology of $F$ and thus the total space $X$, as well as de Rham cohomology relative to a subspace of $F$ or $X$.  The integral of such a form is defined by 
\begin{equation}
\label{StratifoldIntegration}
\int_F \omega := \int_{\widetilde{F}} \psi^*\omega.
\end{equation}
This is sufficient to define fiberwise integration at the level of forms.  

Since $(F, \d F)$ is a nicely glued $k$-manifold with corners, Proposition \ref{GluedMfdFundClass} tells us that 
$H^i(F, \d F; \,\Z)$ is $0$ for all $i>k$ and $H^k(F, \d F; \, \Z) \cong \Z$.  
Moreover, it tells us that there is a generator of $H_k(F, \d F; \, \Z) \cong \Z$ that maps to the fundamental class $[\tF, \d\tF] \in H_k(\tF, \d \tF; \, \Z)$ determined by the orientation of $\tF$.  
The orientation of $\tF$ thus ultimately determines the isomorphism $H^k(F, \d F; \, \Z) \cong \Z$.
This suffices to define the pushforward map from the Serre spectral sequence.

Fiberwise integration of forms does not always descend to cohomology for manifolds with boundary, but we now check that it does for forms in $C^*_{dR}(X, \d X)$, using that $(F, \d F)$ is a nicely glued manifold with corners.  We first claim that for all $\eta \in C^*_{dR}(X, \d X)$, we have $\int_{\d \widetilde{F}} \psi^*\eta|_{\d X} = 0$.
Indeed, this is because all the faces in $\d \tF$ are either mapped into $\d F$ (by a diffeomorphism or a collapse to lower dimension) or glued by orientation-reversing diffeomorphisms.  
Then by Stokes's Theorem,
\[
d \int_{F} \eta 
:=d \int_{\tF} \psi^*\eta 
=  \int_{\tF} d (\psi^*\eta) \pm \int_{\d \tF} \psi^* \eta|_{\d X}
=\int_{\tF} \psi^*(d \eta)
=: \int_F d \eta
\]
so fiberwise integration is a chain map and hence descends to cohomology.  
Thus we may compare at the level of cohomology fiberwise integration to pushforward from the Serre spectral sequence, provided the pair of fibers $(F, \d F)$ of the pair of bundles $(F, \d F) \to (X, \d X) \to B$ is a nicely glued manifold with corners.

\begin{lemma}
\label{GluedIntegrationAndSSS}
Let $F \to X \overset{\pi}{\to} B$ be the fiberwise quotient of an orientable bundle $\tF \to \tX \overset{\widetilde{\pi}}{\to} B$  by a map $\psi: \tX \to X$.  Assume that $\tF$ is a compact, oriented, nicely glued $k$-dimensional manifold with corners.
Let $\Phi: (X, \d X) \to (Y, Z)$ be a map to a closed manifold $Y$ such that $H^{p+k}(Y,Z) \cong H^{p+k}(Y)$.
Let $[\omega] \in H^{p+k}(Y)$, where $\omega$ denotes a representative de Rham cochain and $[\omega]$ its image in singular cohomology.
Then the pushforward  map
\[
\pi^!: H^{p+k}(X,  \d X;\, \R) \to H^p(B;\, \R)
\]
defined by the Serre spectral sequence and the class $[F, \d F]$ agrees with integration over the fiber $F$ in de Rham cohomology, in the sense that $\left[ \int_F \Phi^*\omega \right] = \pi_!(\Phi^*[\omega])$
\end{lemma}

\begin{proof}
Unraveling the equality at the end of the Lemma statement, we must verify that a class $\omega$ in the upper-left corner of the following diagram maps to the same element in $H^p(B;\, \R)$.  As in Lemma \ref{IntegrationAndSSS}, $\int$ denotes de Rham isomorphisms, $H^*_{dR}(-)$ denotes de Rham cohomology, and $H^*(-;\, \R)$ denotes singular cohomology with $\R$ coefficients.
\begin{equation}
\label{IntVsPushfwd}
\xymatrix{
H^{p+k}_{dR}(X, \d X) \ar[d]_-{\int_F} & \ar[l]_-{\Phi^*} H^{p+k}_{dR}(Y,Z) \ar[ddd]^-\cong_{\int} \\
H^p_{dR}(B) \ar[d]^-\cong_{\int} & \\
H^p(B; \,\R) & \\
H^{p+k}(X, \d X; \,\R) \ar[u]^-{\pi^!} & \ar[l]_-{\Phi^*} H^{p+k}(Y,Z; \, \R)
}
\end{equation}
We will consider the effect on a cycle $\sigma$ representing an element in $H_p(B)$.  
The sequence \eqref{IntegrationBySSSDual} below is a dual version of \eqref{IntegrationBySSS} in homology, where we now use $\mathcal{F}_q^*$ to denote the increasing filtration of $H_q(X, \d X)$ induced by the $E^\infty$ page.  We denote the following composition by $\pi_!$.
\begin{equation}
\label{IntegrationBySSSDual}
H_p(B;\,\Z) \cong 
H_p(B;\, H_k(F, \d F;\,\Z)) \cong 
E^2_{p,k} \twoheadrightarrow
E^\infty_{p,k}\cong
\mathcal{F}_{p+k}^{p} / \mathcal{F}_{p+k}^{p-1} =
 \mathcal{F}_{p+k}^{p} \incl
H_{p+k}(X, \d X;\,\Z)
\end{equation}
Essentially, we will represent this composition at the chain level.  For this purpose, it is convenient to represent singular homology not by singular simplices but by singular cubes, as in Serre's thesis \cite[Chapter II]{Serre-Thesis}.  Thus a $p$-cycle $\sigma$ in $B$ is a formal linear combination of maps $c:I^p \to B$.  
We will lift $\sigma$ first to a $(p+k)$-chain on $\tX$ via a map $C_p(B) \otimes C_k(\tF) \to C_{p+k}(\tX)$.  It suffices to define this map on an element $c\otimes \widetilde{f}$ for each cube $c$.  In this case we can take the pullback $\tX'$ of $\tX \to B$ along $c:I^p \to B$, and choose a trivialization $I^p \x \tF \overset{\cong}{\longrightarrow} \tX'$.  Fix a chain $\widetilde{f}$ representing the fundamental class $[\tF, \d \tF]$ of the (possibly disconnected) manifold with corners $\tF$, and send $\sigma \in C_p(B)$ to the image of $\sigma \otimes \widetilde{f}$, a $(p+k)$-chain $\widetilde{\pi}^{-1}(\sigma) \in C_{p+k}(\tX)$.  

Let $f:=\psi_*(\widetilde{f})$.  Then by Proposition \ref{GluedMfdFundClass}, $f$ represents $[F, \d F] \in H_k(F/\d F; \,\Z)$.  We will perform the same construction as above with $f$ instead of $\widetilde{f}$ to get a chain on $X$, and we let $\pi^{-1}(\sigma)$ be its image in $C_{p+k}(X, \d X)$.  For each cube $c$ in $B$, we choose the trivialization of the pullback $X'$ to come from the one on $\tX'$ and the map on pullbacks induced by $\psi$, i.e., from the composition $I^p \x \tF \overset{\cong}{\longrightarrow} \tX' \to X'$. This guarantees that 
\begin{equation}
\label{TwoLiftsOfSigma}
\pi^{-1}(\sigma)=\psi_*(\widetilde{\pi}^{-1}(\sigma)).
\end{equation} 
By the same arguments that one uses in the de Rham theorem to pass from singular simplices to smooth singular simplices, we can modify $\widetilde{\pi}^{-1}(\sigma)$ if necessary so that each cube in $\widetilde{\pi}^{-1}(\sigma)$ is smooth on the interior of $\tX$.  We accordingly modify $\pi^{-1}(\sigma)$, using \eqref{TwoLiftsOfSigma}.

Finally, consider the homology Serre spectral sequence for $X\to B$ using relative chains in $C^*(X,\d X)$.  There $\pi^{-1}(\sigma)$ represents an element of $E^0_{p,k}$, which is a quotient of $C_{p+k}(X, \d X)$ but is also isomorphic to $C_p(B) \otimes C_k(F, \d F)$.  It represents an element of $E^1_{p,k} \cong C_p(B) \otimes H_k(F, \d F)$ because $f$ is closed and an element of $E^2_{p,k}\cong H_p(B) \otimes H_k(F, \d F)$ because $\sigma$ is closed.  It represents an element of $E^\infty_{p,k}$ because the spectral sequence is zero above the $k$-th horizontal line.  From its construction out of the fundamental cycle $f$, we see that $\pi^{-1}(\sigma)$ indeed represents $\pi_!([\sigma])$.

Returning to \eqref{IntVsPushfwd}, we describe $\pi^!$ in terms of the effect of $\pi_!$ on a chain.  We thus interpret the value of the clockwise composition on the cycle $\sigma$ as an integral over the image of $\sigma$ in $Y$, namely
\begin{equation}
\label{CWComp}
\int_{\Phi_* ({\pi}^{-1}(\sigma))}  \omega = \int_{\Phi_* \psi_* (\widetilde{\pi}^{-1}(\sigma))}  \omega.
\end{equation}
On the other hand, fiberwise integration followed by integration over $\sigma$ can be formulated as one integral in the total space $X$, so the value of the counterclockwise composition on $\sigma$ is 
\begin{equation}
\label{CCWComp}
\int_{{\pi}^{-1}(\sigma)}  \Phi^* \omega :=
\int_{\widetilde{\pi}^{-1}(\sigma)}  \psi^*\Phi^* \omega
\end{equation}
where we have decorated the equality to indicate that it comes from the defining equation \eqref{StratifoldIntegration}.
Since $\Phi \circ \psi$ is a map of smooth manifolds, the right-hand sides of \eqref{CWComp} and \eqref{CCWComp} are equal.
\end{proof}

\section{Compactified configuration spaces}
\label{Section3}
We will now define and describe a certain compactification of configuration spaces, essentially the one used by Kuperberg and Thurston \cite{KuperbergThurston}.
In Section \ref{Blowups} we cover generalities about blowups, neat submanifolds, and mutual transversality, with intermittent reference to the application in mind.  We then define the compactification in Section \ref{CompactificationDef} and describe its corner structure in Section \ref{CornerDescription}.

In slightly more detail, for any space $X$, let $C_q(X)$ be the space of distinct ordered $q$-tuples in $X$, the (uncompactified, or open) \emph{configuration space} of $q$ points in $X$.  For us $X$ will always be a manifold $M$, often Euclidean space.
We ultimately want to perform an analogue of integration over configuration spaces, so we need to replace them by compact versions.  
We will define a {compact space} $C_\Gamma[M]$ of points in $M$ labeled by the vertices of $\Gamma$, where $\Gamma$ is a $\geq 3$-valent graph as in Section \ref{Basics}.  
This space will have the important property that for $M=\R^d$, it admits maps to $S^{d-1}$ for each edge $(i,j)$ in $\Gamma$.
The graph $\Gamma$ may be oriented as in Definition \ref{LabeledGraphs}, but  $C_\Gamma[M]$ will depend on only the underlying unoriented graph of $\Gamma$.  Moreover it will depend only on the underlying graph $T(\Gamma)$, where edges and arcs are not distinguished from each other.
Throughout, we will mainly use the letters $i,j$, etc.~for labels on vertices of $\Gamma$ and $x_i, x_j$, etc.~for the corresponding points in $M$.  We may occasionally use the same symbol for a vertex and the corresponding configuration point.

\subsection{Iterated blowups, neatness, and mutual transversality}
\label{Blowups}
For the purpose of defining iterated blow-ups, let $X$ be an $n$-dimensional manifold with corners, as discussed at the beginning of Section \ref{GluingMfds}.  We will be able to blow up along \emph{neat submanifolds}; a categorical definition due to Laures \cite{Laures} was used in our previous related work \cite{Rbo}, but the following definition will suffice here.

\begin{definition}
\label{Neat}
Let $X$ and $Y$ be manifolds with corners of dimensions $n$ and $m$.
We say $Y \subset X$ is \emph{a neat(ly embedded) submanifold of $X$} or \emph{neat in $X$} if each $y \in Y$ has a neighborhood $U$ in $X$ such that the pair $(U, \, U \cap Y)$ is diffeomorphic to $(\R^{n-k} \x \R_+^{k}, \ \R^{m-k} \x \{0\} \x \R_+^{k})$ for some $k$ with $0 \leq k \leq \min(m,n)$, where $y\mapsto 0$.   
\end{definition}

So for example, the quarter of the unit 2-sphere in $\R \x \R_+^2$ (with the induced corner structure) is a neat submanifold of $\R \x \R_+^2$, but the ray $x=y$ is not a neat submanifold of $\R_+^2$. 
If $Y$ is neat in $X$, then $Y$ is transverse to each corner of $X$.  The converse fails, for example when points in $\d Y$ are in $\mathrm{int} (X)$.  However, standard transversality techniques can be used to prove a partial converse:

\begin{lemma}
\label{NeatLemma}
Let $X$ be a manifold with corners, and suppose that $Y \subset X$ is locally the intersection of a boundaryless manifold with a Euclidean space with corners.
That is, each $y \in Y$ has a neighborhood $U$ in $X$ such that $(U, U \cap Y)$ is diffeomorphic to $(\R^{n-k} \x \R_+^{k}, (\R^{n-k} \x \R_+^{k}) \cap Z)$ for some submanifold $Z \subset \R^n$.
If $Y$ is transverse to every corner of $X$, then $Y$ is neat in $X$.
\qed
\end{lemma}

We now want to replace a submanifold $Y \subset X$ by the sphere bundle $S(\nu(Y))$ of its normal bundle.  A concrete way to do this is to assume a Riemannian metric on $X$.  If $Y \subset X$ is a neat submanifold with corners, then $Y$ has a well defined normal bundle $\nu(Y)$, and from the metric, a tubular neighborhood $\eta_{\eps}(Y)$ of radius $\eps = \eps(y)$, with $\eta_\eps(Y) \cong \nu(Y)$ \cite{Douady-Nbds-Corners}.  Let $\eta_{\eps/2}(Y)$ be the open sub-neighborhood of radius $\eps/2$.   

\begin{definition} \
\begin{itemize}
\item
Define the \emph{blowup of $X$ along $Y$} as $\mathrm{Bl}(X,Y) := X - \eta_{\eps/2}(Y)$.  
When $X$ is clear, we abbreviate this as the \emph{blowup of $Y$}.  We say that we \emph{blow up $Y$} when we construct $\mathrm{Bl}(X,Y)$ from $X$.
\item 
Define the \emph{blow-down} map $\mathrm{Bl}(X,Y) \to X$ as follows.  On $X - \overline{\eta_{\eps/2}(Y)}$, it is given by a diffeomorphism to $X - Y$, using $\eta_\eps(Y) \cong \nu(Y)$ and a diffeomorphism $(\eps/2, \eps) \to (0, \eps)$.  Elsewhere, it is given by the composition $\mathrm{Bl}(X,Y) \cap \overline{\eta_{\eps/2}(Y)} \cong S(\nu(Y)) \to Y$. 
\end{itemize}
\end{definition}

We would like to blow up the diagonal $x_i=x_j$ in $M^{V(\Gamma)}$ for every pair $i \neq j$ of distinct endpoints of an edge or arc in $\Gamma$.  This is not possible unless we blow up diagonals of lower dimension.  For example, suppose we wanted to blow up $\R^3$ along the planes $P_1:=\{x=y\}$, $P_2:=\{y=z\}$, and $P_3:=\{x=z\}$.  The blowup along all three planes  is ill defined because $P_i$ is not transverse to $P_j \cap P_k$.
Thus after blowing up say $P_j$ and $P_k$, the image of $P_i$ is not a neat submanifold.
However, if we first blow up the line $L:=\{x=y=z\}$, we may then blow up the (images of the) three planes in any order.  (On the other hand, if we wanted to blow up only two of these planes, we could do so without blowing up $L$ first.)

One solution is to blow up all the diagonals in $M^{V(\Gamma)}$ (in an inclusion-preserving order).  
This gives the \emph{canonical} Axelrod--Singer compactification\footnote{This is also called the Fulton--Macpherson compactification, though the latter term is sometimes reserved for projectivized blow-ups over $\mathbb{C}$.}, used by Bott and Taubes.  This compactification has local coordinates that make it a manifold with corners \cite{Axelrod-Singer}.
One can however blow up fewer diagonals and still blow up all the diagonals corresponding to edges of $\Gamma$.
This was the approach of Kuperberg and Thurston, which we follow.  
This smaller compactification is crucial for the hidden face involutions in Section \ref{HiddenFaceInvolution}. Analytically, one can argue the vanishing of the contributions from these faces in the Axelrod--Singer compactification by first integrating over some of the coordinates, but topologically, the involutions of the whole configuration space are not well defined if one blows up every diagonal.  Cf.~\cite[Section 5.1.2]{KuperbergThurston}.
Using the smaller compactification also eliminates the need to argue the vanishing of integrals along certain boundary faces.  Namely, the analogues of Lemmas A.7 and A.8 in \cite{CCRL-AGT} are not needed because we will perform blow-ups for only biconnected subgraphs.  We now return to a general sufficient condition for blowups to be well defined.

\begin{definition}
For a manifold with corners $X$, call a set of submanifolds with corners $Y_1, \dots, Y_k$ \emph{mutually transverse} if any $Y_i$ is transverse (in $X$) to the intersection of any collection of the other $Y_j$ with any corner of $X$.
\end{definition}

By Lemma \ref{NeatLemma}, if the $Y_i \subset X$ are locally intersections of boundaryless manifolds with Euclidean spaces with corners, then mutual transversality implies neatness.
If the $Y_i$ are linear subspaces of $X=\R^d$, then they are mutually transverse if and only if each $Y_i$ is transverse to the intersection of \emph{all} the other $Y_j$.  An equivalent condition is that $\codim \bigcap_{i=1}^k Y_i = \sum_{i=1}^k \codim Y_i$.  
The appearance of the corners of $X$ in the definition means that for example, the rays $L:=\{x=0\}$ and $K:=\{x=y\}$ are \emph{not} mutually transverse in the upper half-plane $\R \x \R_+$, even though each one is neat and $K \pitchfork L$.
We do not guarantee that our definition agrees with others authors' definitions in the case of nonempty boundary.

\begin{lemma}
\label{MutualTransLemma}
Suppose that $Y_1, \dots, Y_k \subset X$ are a collection of mutually transverse neat submanifolds with corners.  Then one can blow up all of the $Y_i$, in any order.
\end{lemma}

\begin{proof}
The blowup of any one $Y_i$ is well defined by neatness.  So assume that some $Y_i$ have already been blown up to form $X'$, and that we want to blow up the submanifold corresponding to $Y$ in $X'$.  
More precisely, let $Y'$ be the preimage of $Y$ under the blow-down map $X' \to X$.
It suffices to check that $Y'$ is neat in $X'$.  
Now $X'$ is obtained by removing neighborhoods from $X$, and $Y$ is neat in $X$, so $Y'$ is locally the intersection of a manifold with a Euclidean space with corners.  By Lemma \ref{NeatLemma}, it suffices to check that $Y'$ intersects each corner in $X'$ transversely.  
A corner in $X'$ corresponds to the intersection of a corner $K$ in $X$ with some  $Y_{i_1}, \dots, Y_{i_k}$ (or more precisely, their preimages in $X'$).  
Mutual transversality ensures that $Y \pitchfork (K \cap Y_{i_1}\cap \dots\cap Y_{i_k})$, as desired.  Mutual transversality, and hence our ability to blow up the $Y_i$, does not depend on the order of $Y_i$.
\end{proof}

In our application, we will blow up a collection of submanifolds that are not all mutually transverse, but in which mutual transversality is satisfied by every subcollection that is minimal with respect to inclusion.  In effect, we will repeatedly use Lemma \ref{MutualTransLemma}.

\subsection{A minimal compactification}
\label{CompactificationDef}
The goal of this subsection is to define compact configuration spaces that depend on a graph $\Gamma$, namely $C_\Gamma[M]$ for compact (and boundaryless) $M$ in Definition \ref{CptfnDefn}(a), and $C_\Gamma[\R^d]$ in Definition \ref{CptfnDefnForRd}.  The main point is Proposition \ref{DiagsTransverse}, which says that the construction is well defined.

For a subgraph $\Gamma' \subset \Gamma$, let $\Delta_{\Gamma'}$ denote the subset of $M^{V(\Gamma)}$ where all the vertices of $\Gamma'$ have collided, i.e., 
\[
\Delta_{\Gamma'}:=\{ f \in M^{V(\Gamma)} : f(x_i) = f(x_j) \ \forall i, j \in V(\Gamma')\}.
\]
We will blow up $X = M^{V(\Gamma)}$ along the diagonal $\Delta_{\Gamma'}$ for every biconnected $\Gamma' \subset \Gamma$.  
Recall that we consider both edges and arcs in our notion of biconnectedness (Definition \ref{BiconnectedBlocks}). 
We will perform these blowups in the following order:
\begin{itemize}
\item First, we blow up  (in any order) all diagonals which are minimal with respect to inclusion among  $\Delta_{\Gamma'}$ with $\Gamma'$ biconnected.
Such a diagonal $\Delta_{\Gamma'}$ is minimal precisely if the subgraph $\Gamma'$ is maximal, i.e.~a block of $\Gamma$.   
\item 
Then we blow up (in any order) the remaining minimal diagonals $\Delta_{\Gamma''}$ with $\Gamma''$ biconnected.
Such a diagonal $\Delta_{\Gamma''}$ corresponds to a maximal subgraph $\Gamma''$ among those properly contained in the subgraphs $\Gamma'$ for which $\Delta_{\Gamma'}$ has already been blown up.
\item 
We continue this process until all the 2-fold diagonals corresponding to edges and arcs of $\Gamma$ are blown up.  
\end{itemize}

\begin{proposition}
\label{DiagsTransverse}
The blowups of diagonals at each stage are well defined.
\end{proposition}

This result will follow quickly from  Lemma \ref{DiagsOfBlocksTransverse} for the minimal diagonals and Lemma \ref{RemainingMinDiags} for the remaining minimal diagonals at each stage.

\begin{lemma}
\label{DiagsOfBlocksTransverse}
The diagonals corresponding to the blocks of $\Gamma$ are mutually transverse.  
\end{lemma}
\begin{proof}
If $M$ were $\R^d$, the diagonals would be linear subspaces, and it would suffice check that the codimension of the intersection is the sum of the codimensions.
In fact, we are considering an arbitrary compact $M$, but this count also suffices because a neighborhood of a diagonal in $M^{V(\Gamma)}$ is modeled by a neighborhood of a diagonal in Euclidean space.
Let $\Gamma_1,\dots,\Gamma_k$ denote the blocks of $\Gamma$, let $v_i = |V(\Gamma_i)|$, $v=|V(\Gamma)|$,  and $c = |CV(\Gamma)|$.  
The intersection of all the $\Delta_{\Gamma_i}$ is $\Delta_\Gamma$, which has codimension $(v-1)d$.  Each $\Delta_{\Gamma_i}$ has codimension $(v_i -1)d$.  So it suffices to check that $\sum_{i=1}^k (v_i -1) = v -1$.  Since $\Upsilon(\Gamma)$ is a tree, $|V(\Upsilon(\Gamma))| - |E(\Upsilon(\Gamma))|=1$.  
Letting $e:=|E(\Upsilon(\Gamma))|$, we rewrite this as $k+c-e=1$, and thus
$$
\sum_{i=1}^k (v_i -1) = 
\sum_{i=1}^k v_i -k =
(v + e - c) - k = 
v - (k+c-e) =
v-1
$$
where Lemma \ref{BlockCutTreeBasics} justifies the second equality.
\end{proof}

\begin{lemma}
\label{RemainingMinDiags}
Suppose some of the $\Delta_{\Gamma'}$ in $M^{V(\Gamma)}$ for biconnected $\Gamma'$ have been blown up in increasing order of inclusion.  That is, suppose that if $\Delta_{\Gamma_i} \subset \Delta_{\Gamma_j}$ and $\Delta_{\Gamma_j}$ has been blown up, then $\Delta_{\Gamma_i}$ has also been blown up.
Then the remaining minimal diagonals are mutually transverse.  
\end{lemma}
\begin{proof}
Let $X=M^{V(\Gamma)}$, and let $X'$ be the blown up manifold.  
For a diagonal $\Delta_{\Gamma_i}$ that has already been blown up, use the same symbol to denote its preimage in $X'$ under the blow-down map $X' \to X$.
Then we must check that 
\begin{equation}
\label{MutualTransEqn}
\Delta_{\Gamma_1} \pitchfork (\Delta_{\Gamma_2} \cap \dots \cap \Delta_{\Gamma_k})
\end{equation}
where $\Delta_{\Gamma_1}$ is a remaining minimal diagonal, and where for each $j=2,\dots,k$, $\Delta_{\Gamma_j}$ is either another remaining minimal diagonal or a diagonal that has already been blown up.  
As our first step, we consider the case where each $\Delta_{\Gamma_j}$ is a remaining minimal diagonal.  
We rule out some easy special cases:
\begin{itemize}
\item
If for some $i$ and $j$, $|V(\Gamma_i \cap \Gamma_j)| \geq 2$, then $\Gamma' :=\Gamma_i \cup \Gamma_j$ is biconnected, so $\Delta_{\Gamma_i} \cap \Delta_{\Gamma_j} = \varnothing$ because $\Delta_{\Gamma'}$ has already been blown up.  
\item
If there is a cycle in $\Gamma_1 \cup \dots \cup \Gamma_k$ which has no repeated vertices and which passes through multiple $\Gamma_{i_1}, \dots, \Gamma_{i_j}$, $j>1$, then $\Gamma':=\Gamma_{i_1} \cup \dots \cup \Gamma_{i_j}$ is biconnected, so $\Delta_{\Gamma'}$ has been blown up, and $\Delta_{\Gamma_{i_1}} \cap \dots \cap \Delta_{\Gamma_{i_j}} = \varnothing$.
\end{itemize}
Thus we may suppose $|V(\Gamma_i \cap \Gamma_j)| \leq 1$ for all $i,j$.  Also, if we form a graph $\mathcal{F}$ analogous to the block-cut tree $\Upsilon(\Gamma)$, in which a vertex is labeled either by some $\Gamma_i$ or a vertex shared by multiple $\Gamma_i$, then we may assume $\mathcal{F}$ is a forest.
As in Lemma  \ref{DiagsOfBlocksTransverse}, it suffices to check that the codimension of the intersection is the sum of the codimensions, since transverse submanifolds remain transverse after a blow-up of another submanifold.  The desired equality is $\sum_{i=1}^k (v_i - 1) = v - f$, where $f$ is the number of components of $\mathcal{F}$.  Lemma \ref{BlockCutTreeBasics} applies to $\mathcal{F}$, thus yielding this equality and establishing \eqref{MutualTransEqn} for the case where all the $\Delta_{\Gamma_j}$ are remaining minimal diagonals.  

As our second step, we consider the case where $\Delta_{\Gamma_j} \subset \Delta_{\Gamma_1}$ for all $j>1$, so these $\Delta_{\Gamma_j}$ have all been blown up.  
Let $p \in \bigcap_{j=1}^k \Delta_{\Gamma_j}$.  
Then any direction in $T_pX$ normal to $T_p \Delta_{\Gamma_1}$ is a fortiori normal to $T_p \Delta_{\Gamma_j}$, $j>1$. 
Hence such a direction is tangent to the corner indexed by $\Delta_{\Gamma_2}, \dots, \Delta_{\Gamma_k}$ (since there is a direction complementary to $T_p(\d X')$ contained in $T_p \Delta_{\Gamma_1}$).  Hence we obtain the desired transversality \eqref{MutualTransEqn} in this case.

Finally, for the general case, let $\Delta'$ be the intersection of those $\Delta_{\Gamma_j}$ contained in $\Delta_{\Gamma_1}$, and let $\Delta''$ be the intersection of the remaining $\Delta_{\Gamma_j}$.  
Let $p \in \bigcap_{j=1}^k \Delta_{\Gamma_j}$.
We can obtain $\Delta_{\Gamma_1} \pitchfork \Delta''$ by reducing to the case of diagonals minimal among those yielding $\Delta''$ and applying the first step.  Thus any direction complementary to $T_p\Delta_{\Gamma_1}$ is in $T_p\Delta''$.  But as in the second step, such a direction also lies in $T_p \Delta'$ (where $\Delta'$ corresponds to diagonals that have already been blown up).  
\end{proof}

\begin{proof}[Proof of Proposition \ref{DiagsTransverse}]
By Lemma \ref{MutualTransLemma}, we just need mutual transversality and neatness of the diagonals considered at each stage (or more accurately, their preimages under the appropriate blow-down maps).  We have mutual transversality from Lemmas \ref{DiagsOfBlocksTransverse} and \ref{RemainingMinDiags}. Each diagonal $\Delta_{\Gamma'}$ is a boundaryless submanifold of $X$, so in each successive blowup, it is locally the intersection of a boundaryless manifold with a Euclidean space with corners.  Thus by Lemma \ref{NeatLemma}, the mutual transversality imply neatness.  
\end{proof}

We may now make the following definition:
\begin{definitions} 
\label{CptfnDefn}
Let $M$ be a compact manifold.
\begin{enumerate}
\item[(a)]
Define $C_\Gamma[M]$ as the result of blowing up $M^{V(\Gamma)}$ along $\Delta_{\Gamma'}$ for every biconnected $\Gamma' \subseteq \Gamma$, in increasing order of inclusion (i.e., starting with the lowest-dimensional diagonals, or the diagonals with the collisions of the most points).  Let $C_\Gamma(M)$ denote the interior of $C_\Gamma[M]$.
\item[(b)]
Define $C_n[M]$ as $C_{K_n}[M]$, where $K_n$ is the complete graph on $n$ vertices (with $L=\varnothing$).  
\end{enumerate}
\end{definitions}

\subsubsection{The case where $M$ is Euclidean space.}
\label{SuspendingGraphs}
For the case $M=\R^d$, which we are primarily interested in, we will also record directions of approach of any point in $\Gamma$ to infinity.
\begin{definition}
\label{GraphSuspension}
Let $\Gamma$ be a graph on $L$.  
Then $\Sigma \Gamma$, the \emph{suspension} of $\Gamma$, is a graph on a singular 1-manifold $L'$ defined as follows:  
\begin{itemize}
\item
The space $L'$ is the one-point compactification of the union of non-compact components of $L$, if $L$ has such components; otherwise $L'=L$.  
\item
The vertices of $\Sigma\Gamma$ are the vertices of $\Gamma$, plus one extra vertex $\infty$.  If $L$ has non-compact components, this vertex coincides with the extra point in the one-point compactification, and it is considered a segment vertex.  Otherwise, $\infty$ is considered a free vertex.  
\item
The edges of $\Sigma\Gamma$ are those of $\Gamma$, plus an edge between $\infty$ and every free vertex of $\Gamma$.  
\item
The arcs of $\Sigma\Gamma$ are those of $\Gamma$, plus an arc incident to $\infty$ for every $[v, \infty) \in \R$ and every $(-\infty, v] \in \R$, where $v$ is a segment vertex in $\Gamma$ and the half-open ray contains no other segment vertices.
\end{itemize}
Note that the vertex $\infty$ may not satisfy the valence conditions in Definition \ref{GraphsDefinition}.
\end{definition}

\begin{definition}
\label{CptfnDefnForRd}
View $S^d$ as the one-point compactification of $\R^d$.  
Define $C_\Gamma[\R^d]$ as the subset of $C_{\Sigma \Gamma}[S^d]$ where the extra vertex is fixed at $\infty \in S^d$.  
\end{definition}
So $C_\Gamma[\R^d]$ consists of those points in $C_{\Sigma \Gamma}[S^d]$ which blow down to $(x_1,\dots,x_{|V(\Gamma)|}, \infty) \in (S^d)^{V(\Sigma \Gamma)}$ for some $x_i \in S^d$.  
A biconnected subgraph of $\Sigma \Gamma$ is either a biconnected subgraph of $\Gamma$ or the suspension of a connected subgraph of $\Gamma$ by the extra vertex $\infty$.  (However, the converse is not quite true: for connected $\Gamma' \subset \Gamma$ with segment vertices, the corresponding subgraph of $\Sigma \Gamma$ need not be biconnected.)  
Thus we have defined $C_\Gamma[M]$ when $M$ is either compact or Euclidean space.

\bigskip

Now for $M$ either a compact manifold or $\R^d$, the space $C_\Gamma[M]$ is a manifold with corners, which one can thus integrate over.  There is a blow-down map
\[
C_\Gamma[M] \to M^{V(\Gamma)}.
\]
It sends an interior point in $C_\Gamma(M)$ into the complement of the diagonals.  
We will see in Section \ref{CornerDescription} that a corner is indexed by a list of subgraphs $\{\Gamma_1, \dots, \Gamma_k\}$.
Roughly the blow-down map collapses $\Gamma_1 \cup \dots \cup \Gamma_k$ to its components.  More precisely, a point in such a corner maps into the diagonal where $x_i = x_j$ for every pair of vertices $(i,j)$ which lie in the same connected component of $\Gamma_1 \cup \dots \cup \Gamma_k$.

\begin{remark}[Slight abuse of terminology]
Since we didn't blow up every diagonal in $M^{V(\Gamma)}$ to get $C_\Gamma[M]$, a point in the interior of $C_\Gamma[M]$ is a tuple of points in $M$ that need not be pairwise distinct.  Nonetheless, we will sometimes refer to such a tuple as a \emph{configuration}.
\end{remark}

\begin{remarks}[Relationships to other compactifications] \
\begin{enumerate}
\item 
The space $C_n[M]$ is precisely the Axelrod--Singer compactification, used by Bott and Taubes in \cite{BottTaubes} and Cattaneo et al in \cite{CCRL-AGT}.  In $C_n[M]$, a collision of any two points is accompanied by the datum of a direction of collision, and a collision of any three points is further accompanied by the datum of a relative rate of approach $(|x_j - x_i| / |x_k - x_j|)$.  See also the work of Sinha \cite{SinhaCptn}.
\item  In our work with Munson and Voli\'c, we also considered spaces obtained by blowing up only some of the diagonals in the cartesian product \cite[Section 4.2.4]{KMV}.  That construction is however somewhat different from the present one.  The construction in \cite{KMV} involves altering $\Gamma$ into a \emph{hybrid} with \emph{graft components} for the purpose of working with homotopy links (i.e. link maps) rather than links (i.e. embeddings).  Here we work only with embeddings, so the hybrid is not needed.  Also, in \cite{KMV}, every diagonal in each ``graft component'' is blown up, whereas we blow up even fewer diagonals here.  
\item 
The compactification used by Poirier \cite{PoirierAGT} appears to be closely related to the Kuperberg--Thurston compactification that we will use below.  It is defined differently from the latter compactification but seems to share the feature of blowing up only those 2-fold diagonals corresponding to edges in $\Gamma$.  
\item
The canonical compactification $C_{|V(\Gamma)|}[\R^d]$ can be obtained from $C_\Gamma[\R^d]$ by merely blowing up (in increasing order of dimension) the images of all the diagonals which have not already been blown up.  There is then a {blow-down} map $C_{|V(\Gamma)|}[\R^d] \to C_\Gamma[\R^d]$.
\end{enumerate}
\end{remarks}

\subsection{The corners of the compactified configuration spaces}
\label{CornerDescription}
We review and elaborate on some details regarding the corner structure from Section 4.3 of \cite{KuperbergThurston}.  These details will help us verify that various gluings and foldings of codimension-1 faces extend to their corners.  They will also be useful in the degeneracy arguments, where we consider the images of various faces under the spherical maps, in Lemma \ref{PositiveCodimension}. 

Let $M$ be a $d$-dimensional manifold without boundary.  The corners of codimension $k$ in $C_\Gamma[M]$ are indexed by sets of subgraphs $S=\{\Gamma_1,\dots,\Gamma_k\}$ of $\Gamma$, where each $\Gamma_i$ is a biconnected graph on at least two vertices, and where $S$ satisfies certain conditions, described below.  
Such a set $S$ indexes a corner which is the intersection of the closures of all of the codimension-1 faces indexed by $\{\Gamma_1\},\dots,\{\Gamma_k\}$.
Different $\Gamma_i$ may correspond to different scales, and the conditions on $S$ can be checked at each scale, much like we repeatedly used Lemma \ref{RemainingMinDiags} at various scales to ensure that the blowups are well defined.  Let  $\Gamma':=\Gamma_1 \cup \dots \cup \Gamma_k$.
\begin{enumerate}
\item
At the most macroscopic scale are the $\Gamma_i \in S$ which are blocks of $\Gamma'$.  These subgraphs thus form the block-cut forest of $\Gamma'$.  The datum at this scale is a configuration in $M$.  
The points are indexed by the vertices in the quotient of $\Gamma$ by collapsing every tree in the block-cut forest of $\Gamma'$ to a point.
\item
At the next level are certain subgraphs in $S$ contained in a biconnected component of $\Gamma'$, say $\Gamma_i$.  The union of these subgraphs $\Gamma_i':=\Gamma_{i_1} \cup  \dots \cup \Gamma_{i_\ell}$ must be a proper subgraph of $\Gamma_i$, and each $\Gamma_{i_j}$ must be a block of $\Gamma_i'$.  
The datum at this stage is a configuration of points in $\R^d (=T_pM)$, modulo translation and positive scaling, which may be thought of as an infinitesimal configuration.
The configuration is indexed by vertices in the quotient of $\Gamma_i$ obtained by collapsing every tree in the block-cut forest of $\Gamma_i'$ to a point.
We call this quotient of configuration space the \emph{screen space for $\Gamma_i$} and a point in it a \emph{screen}.  We allow the possibility that $\Gamma_i' = \varnothing$.
\item
A subgraph $\Gamma_{i_j}$ may then contain further subgraphs, and in general step (2) may be repeated an arbitrary finite number of times.  
\end{enumerate}
One can check whether a set $S$ indexes a (nonempty) corner by discarding the $\Gamma_i$ maximal in $\bigcup \Gamma_i$, then checking the conditions in step (2) for the remaining maximal subgraphs, then discarding those subgraphs and checking (2) for the remaining maximal subgraphs, and so on.  
The next statement gives a necessary condition for $S$ to index a corner.
It is not quite sufficient because of the maximality condition described above: for example, the three edges of a triangle satisfy the hypotheses below, but they do not index a corner.

\begin{proposition}
\label{NecessaryConditionCorner}
If $S = \{\Gamma_1, \dots, \Gamma_k\}$ indexes a corner, then the $\Gamma_i$ are biconnected subgraphs such that each pair $\Gamma_i, \Gamma_j$ is either disjoint, nested, or intersects in a single vertex.
\end{proposition}
\begin{proof}
Any $\Gamma_i$ is a block of a graph contained in some other $\Gamma_j$, so this follows from Lemma \ref{IntersectionOfBlocks}.
\end{proof}

Ultimately, a point in the corner indexed by $S$ is given by a configuration of points in $M$ plus one screen for each $\Gamma_i$ in $S$.  
Faces indexed by $S$ and $S'$ intersect precisely when the union $S \cup S'$ satisfies the above conditions.  In that case, $S\cup S'$ is the indexing set of the intersection.  

\begin{notation}
\label{StratumNotation}
A stratum (i.e.~corner) of $C_\Gamma[M]$ indexed by a set $S$ of subgraphs in $\Gamma$ will be denoted $\SS(\Gamma, S)$.  If $S$ consists of a single subgraph $\Gamma'$ or a single edge $e$, we write $\SS(\Gamma, \Gamma')$ and $\SS(\Gamma, e)$ respectively.  For $M=\R^d$, each corner is indexed by a set $S$ of subgraphs of $\Sigma \Gamma$, but we will use the same notation $\SS(\Gamma, S)$ for such a corner.
\end{notation}

\begin{example}
The 6 circled subgraphs shown in Figure \ref{CornerExFig1} correspond to a corner of codimension 6.  
At the largest scale, the vertices shown below correspond to 2 configuration points in $M$, one for the edge labeled $h$ and one for the remaining vertices.  At the next scale, we have a screen with 2 points for edge $h$, a screen with 3 points for the triangle, and a screen with 3 points from collapsing all the circled subgraphs in the double-square.  At the smallest scale, we have three screens with 2 points, one for each circled edge in the double-square.  (The screens for the edges $e$ and $f$ are independent, but this poses no problems since we do not blow up every diagonal.)  
The reader may verify that the resulting dimension of this corner is such that its codimension is indeed 6.  

Figure \ref{CornerExFig2}, where an extra subgraph is added, represents a corner of codimension 7.  Compared to the corner represented by Figure \ref{CornerExFig1}, a screen with 3 points at the intermediate scale is replaced by a screen with 2 points (one for the square and one for edge $g$) and, at a smaller scale, a screen with 2 points (one for edges $e$ and $f$ and one for the remaining vertex in the square).  \qed
\end{example}

\begin{figure}[h!]
\qquad  \raisebox{-8pc}{\includegraphics[scale=0.3]{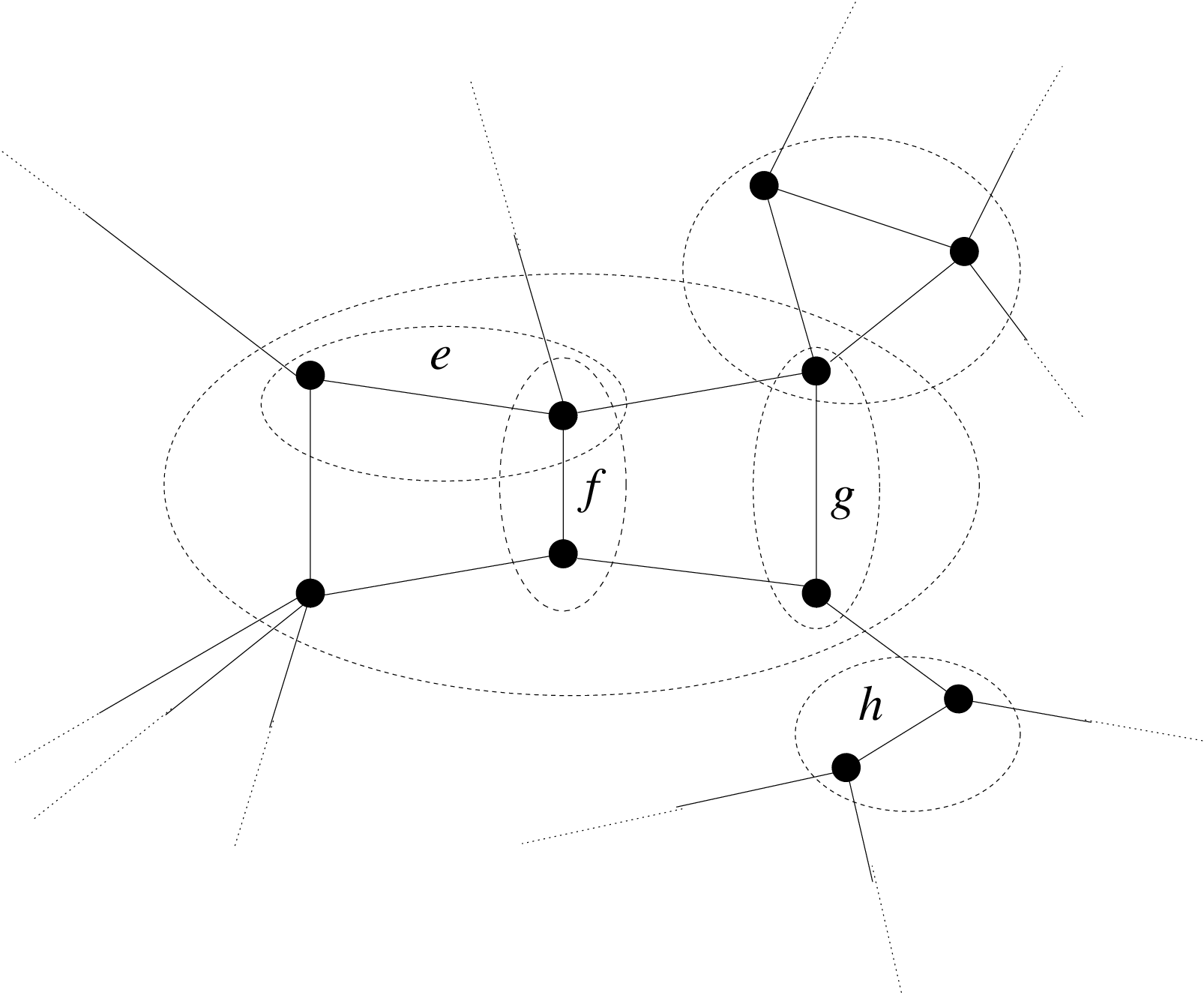}} 
\caption{An example of a set of subgraphs which indexes a corner of codimension 6.}
\label{CornerExFig1}
\end{figure}
\begin{figure}[h!]
\qquad \raisebox{-8pc}{\includegraphics[scale=0.3]{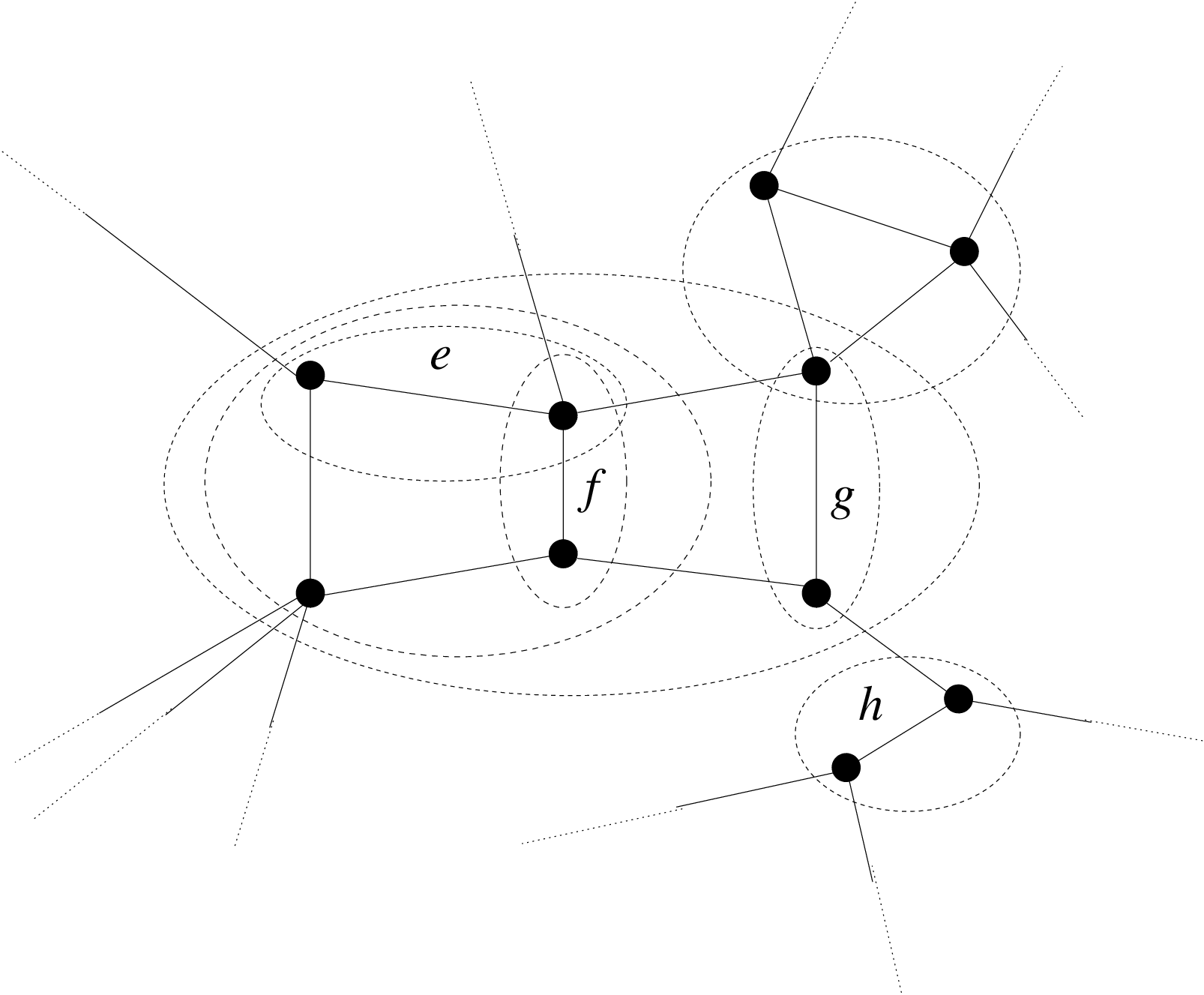}}
\caption{An example of a set of subgraphs which indexes a corner of codimension 7.}
\label{CornerExFig2}
\end{figure}

\newpage
We partition the codimension-1 faces of $C_\Gamma[M]$ into three types, using standard terminology.  Each type of codimension-1 face corresponds to a certain type of subgraph of $\Gamma$.
\begin{definition}
\label{TypesOfFaces} \
\begin{itemize}
\item
A \emph{principal face} is a face $\SS(\Gamma, e)$ obtained by blowing up the 2-fold diagonal $\Delta_e$ in $M^{V(\Gamma)}$.  In the case of $M=\R^d$, where $C_\Gamma[\R^d] \subset C_{\Sigma \Gamma}[S^d]$, the principal faces are those obtained by blowing up any diagonal involving 2 of the first $|V(\Gamma)|$ points.  Such a face corresponds to an edge $e$ of $\Gamma$.
\item
For $M=\R^d$, a \emph{face at infinity} is a face $\SS(\Gamma, \Gamma')$ where one or more points have collided with $\infty \in S^d$.  Such a face corresponds to a biconnected subgraph $\Gamma' \subset \Sigma \Gamma$ containing the vertex $\infty$ (whose removal gives a connected subgraph of $\Gamma$).
\item
The remaining codimension-1 faces are called \emph{hidden faces}.  These are faces $\SS(\Gamma, \Gamma')$ involving a collision of more than two points, none of which is $\infty$, if $M=\R^d$.  Such a face corresponds to a biconnected subgraph $V(\Gamma')\geq 3$ and $\Gamma' \subset \Gamma$.
\end{itemize}
\end{definition}

Sometimes the term ``hidden face'' is used to describe any face that is not a principal face, in which case faces at infinity are considered to be a special type of hidden face.

\section{Bundles over spaces of links and resulting real cohomology classes}
\label{Section4}
We now use the compactification from Section \ref{Section3} to construct bundles over spaces of links (Section \ref{BundleOverLinkSpace}) and pull back spherical cohomology classes (Section \ref{SphericalMaps}).  
We then fiberwise integrate differential forms (Section \ref{FiberwiseIntegrals}) and recall why this produces a cohomology class out of any graph cocycle (Section \ref{VanishingAndGraphCohomology}).  
The purpose of this detour into de Rham cohomology (which is not used in our main construction) is to connect our main construction to previous results.  
In particular, we need the nontriviality of the resulting cohomology classes (Theorem \ref{FTFTI-Analogue}).
This Section will be mostly familiar material to readers acquainted with for example the work of Bott and Taubes \cite{BottTaubes}or Cattaneo et al.~\cite{CCRL-AGT} or Munson, Volic, and the author \cite{KMV}. 
Our choice of the Kuperberg--Thurston compactification from Section \ref{Section3} (which is ``smaller'' than the canonical Axelrod--Singer compactification) is the only new twist, but 
Lemma \ref{IntegralsAreTheSame} ensures that this causes no difficulties.

\subsection{A fiber bundle over the space of links}
\label{BundleOverLinkSpace}
We now build a fiber bundle over the space of links whose fibers are closely related to the compactified configuration spaces $C_\Gamma[\R^d]$ described above.  
The bundle will depend on the full data in the graph $\Gamma$ on $L$, 
unlike the spaces $C_\Gamma[\R^d]$ which depend only on $T(\Gamma)$.  The orientation on $\Gamma$ will be used in defining integrals over these bundles.
We first need a suitable compactified configuration space of points in the 1-manifold $L$.

\begin{definitions} 
\label{SegmentConfigsDefn}
Suppose $L \neq \varnothing$, and let $\Gamma$ be a graph on the 1-manifold $L$.
\begin{itemize}
\item
Define $s(\Gamma)$ as the {s}ubdiagram of $\Gamma$ consisting of all the {s}egment vertices and all edges (chords) and arcs between them.  
\end{itemize}
Let $\Map_{s(\Gamma)}[L]$ be the space of maps $V(s(\Gamma)) \to L$ which respect the components that the vertices lie on and the order of the vertices on these components.  (The order is determined by the orientation on $L$.)  Let $C^o_{s(\Gamma)}[L]$ be the result of blowing up in $\Map_{s(\Gamma)}[L]$ the diagonal for every biconnected subgraph of $s(\Gamma)$, in increasing order of inclusion as in Definition \ref{CptfnDefn}.  
This is a configuration space that is compactified near every collision, except for collisions at infinity.  
Finally fix any smooth embedding $e$ of $L$ into some $\R^d$ for $d\geq 3$, such that $e$ is affine-linear outside a compact subset and such that the directions of approach to infinity of the $2m$ rays of $L$ are pairwise distinct.  
\begin{itemize}
\item
Define $C_{s(\Gamma)}[L]$ as the closure of the image of the map $C^o_{s(\Gamma)}[L] \to C_{s(\Gamma)}[\R^d]$ induced by $e$.
\end{itemize}
\qed
\end{definitions}

A key feature of the compact space $C_{s(\Gamma)}[L]$ is that it not only includes limit points at infinity, but also records relative rates of approach of points to infinity.
Any two choices for $e$ will yield diffeomorphic spaces $C_{s(\Gamma)}[L]$.  This space is very similar to one we defined and elaborated on in \cite{KMV}.  If $L$ is closed, then $C_{s(\Gamma)}[L]$ agrees with the result of substituting $L$ for $M$ and $s(\Gamma)$ for $\Gamma$ in Definition \ref{CptfnDefn}.  Putting $L=\R$, the space $C_{s(\Gamma)}[\R]$ agrees with the result of substituting $s(\Gamma)$ for $\Gamma$ in Definition \ref{CptfnDefnForRd} with $d=1$.

Now consider the pullback
\begin{equation}
\label{GammaBTsquare}
\xymatrix{
X[\Gamma] \ar[r] \ar[d] & C_\Gamma[\R^d]  \ar[d] \\
\L^d_m \x  C_{s(\Gamma)}[L] \ar[r] & C_{s(\Gamma)}[\R^d].
}
\end{equation}

We claim that there is a well defined projection map $C_\Gamma[\R^d]  \to C_{s(\Gamma)}[\R^d]$ as shown above.  In fact, a biconnected subgraph of $\Sigma(s(\Gamma))$ is a biconnected subgraph of $\Sigma \Gamma$.
Thus under the projection of the cartesian products $(S^d)^{V(\Sigma\Gamma)} \to (S^d)^{V(\Sigma(s(\Gamma)))}$, the preimage of every diagonal that is blown up to construct $C_{s(\Gamma)}[\R^d]$ is blown up in constructing $C_\Gamma[\R^d]$, justifying the right-hand vertical map.  
The lower horizontal map comes from the fact that an embedding $(f: L \incl \R^d) \in \L^d_m$ induces for any $\Gamma$ a map of compactifications $C_{s(\Gamma)}[L] \incl C_{s(\Gamma)}[\R^d]$.

For each $\Gamma$, we will consider the bundle $p: X[\Gamma] \to \L^d_m$, where $p$ is given by the left-hand vertical map above followed by the projection to $\L^d_m$.  
Let $F[\Gamma]$ denote the fiber of this bundle, say over a link $f:L \to \R^d$.  Its dimension is lower than that of $C_\Gamma[\R^d]$, since some points are constrained to lie in the image of $f$.  Nonetheless its corner structure is the same as that of $C_\Gamma[\R^d]$.  That is, the faces of $F[\Gamma]$ are in one-to-one correspondence with those of $C_\Gamma[\R^d]$, and this correspondence preserves codimension. Thus $F[\Gamma]$ (or $X[\Gamma]$) has principal faces, faces at infinity, and hidden faces corresponding precisely to those of $C_\Gamma[\R^d]$.  

For each graph $\Gamma$, $X[\Gamma]$ is orientable as a bundle.  Indeed, fixing orientations of $\R$ and $\R^d$ and using the ordering of the segment (respectively all) vertices if $d$ is even (respectively odd) determines a compatible orientation on each fiber.  Thus it is possible to fiberwise integrate forms on this bundle.

\begin{remark}
\label{BottTaubesBundle}
This bundle is similar to the one used by Bott and Taubes and Cattaneo et al.  In fact, a pullback square similar (\ref{GammaBTsquare}) but with $C_\Gamma[\R^d]$, $C_{s(\Gamma)}[L]$, and $C_{s(\Gamma)}[\R^d]$ replaced by  $C_{|V(\Gamma)|}[\R^d]$, $C_{|V_{seg}(\Gamma)|}[L]$, and $C_{|V_{seg}(\Gamma)|}[\R^d]$ respectively produces their bundle.  
Theirs depends only on the numbers of segment vertices (say $q_1+\dots+q_m$) and free vertices (say $t$) rather than all the data in the graph $\Gamma$.
Call the total space of this bundle $X[q_1,\dots,q_m;t]$ 
and call its fiber $F[q_1,\dots,q_m;t]$.\footnote{In our previous work, we called such a pullback $E[q_1,\dots,q_m;t]$.  In this article we use $X$ instead of $E$ to avoid overloading notation.  Other authors use various other notations for this total space.}
\end{remark}

\subsection{Spherical maps}
\label{SphericalMaps}
To define classes in $H^*(X[\Gamma])$, we will use the orientation on $\Gamma$.  We do \emph{not} use the suspension $\Sigma \Gamma$ in defining the maps below.

For every edge in $\Gamma$ with endpoints $i \neq j$, we consider the map $\phi_{ij}$ defined by the composition
\[
\phi_{ij}: X[\Gamma] \to C_{\Gamma}[\R^d] \to S^{d-1}
\] 
where the second map is given on the interior by
\[
(x_1,x_2,\dots)\mapsto \frac{x_j - x_i}{|x_j - x_i|}.
\]  
This map is well defined because we have blown up the 2-fold diagonal $x_i=x_j$ for every edge with endpoints $\{i,j\}$ when constructing $C_\Gamma[\R^d]$.  

For a self-loop on a segment vertex $i$, consider the map
\[
\phi_{ii}: X[\Gamma] \to S^{d-1}
\]
given by the unit tangent vector to $L$ at the point $x_i$.  

If $d$ is odd, define 
\begin{equation}
\label{BigPhi}
\Phi=
\Phi_\Gamma :=\prod \phi_{ij} : X[\Gamma] \to \prod S^{d-1}
\end{equation}
where both products are taken over all edges $(i \to j)$ (including self-loops $(i \to i)$) in $\Gamma$, and where the order of factors in the product is chosen arbitrarily.
If $d$ is even, define $\Phi$ similarly, by arbitrarily choosing either $\phi_{ij}$ or $\phi_{ji}$ for each edge between vertices $i$ and $j$, but ordering the factors according to the order of the edges.
We use the term \emph{spherical map} to refer to $\Phi$ or a constituent $\phi_{ij}$.

\subsection{Fiberwise integrals}
\label{FiberwiseIntegrals}
Although our main construction is in singular cohomology, we now consider integration of forms in order to connect our results to the previously found Bott--Taubes--Vassiliev $\R$-valued cohomology classes.

Let $\omega$ denote an antipodally symmetric unit-volume form on $S^{d-1}$.  For each edge of $\Gamma$ with endpoints $i,j$ let $\theta_{ij} = \phi_{ij}^*\omega \in \Omega_{dR}^{d-1} (X[\Gamma])$.
For brevity, we will sometimes write $\alpha_\Gamma := \prod_{\text{edges in $\Gamma$}} \theta_{ij}$.  
As above, the labeling of $\Gamma$ determines precisely the order of the indices $i,j$ if $d$ is odd and precisely the order of the factors in the product if $d$ is even.  In either case, this uniquely determines $\alpha_\Gamma$, regardless of the remaining arbitrary choices.  

We consider the integral $I_\Gamma$ over the fiber of the bundle $F[\Gamma] \to X[\Gamma] \to \L^d_m$, defined by 
\[
I_\Gamma := \int_{F[\Gamma]} \alpha_\Gamma = \int_{F[\Gamma]}\prod \theta_{ij},
\]
which is a differential form on $\L^d_m$.

First, this integral over our bundle $X_\Gamma$ agrees with the corresponding integral over the slightly different original Bott--Taubes bundle $X[q_1,\dots,q_m;t]$.  (See Remark \ref{BottTaubesBundle} for its definition.)
In fact, one can easily define a form on $X[q_1,\dots,q_m;t]$ similar to the form $\alpha_\Gamma$ on $X[\Gamma]$.  We will use the same notation for both forms.
To compare the fiberwise integrals of $\alpha_\Gamma$ over these two bundles,
recall that $C_{q_1+\dots+q_m+t}[\R^d]$ is obtained by merely blowing up subsets of the boundary of $C_{\Gamma}[\R^d]$.  Thus the integrals of $\alpha_\Gamma$ over $F[q_1,\dots,q_m;t]$ and $F[\Gamma]$ agree on complements of arbitrarily small neighborhoods of the boundary:  

\begin{lemma}
\label{IntegralsAreTheSame}
The integrals 
\begin{align*}
\int_{F[q_1,\dots,q_m;t]} \theta_{i_1 j_1} \cdots \theta_{i_k j_k} & & \mbox{and} & & \int_{F[\Gamma]} \theta_{i_1 j_1} \cdots \theta_{i_k j_k}
\end{align*}
along the fibers of the bundles 
\begin{align*}
F[q_1,\dots,q_m;t] \to X[q_1,\dots,q_m;t] \to \L^d_m & & \mbox{and} & & F[\Gamma] \to X[\Gamma] \to \L^d_m
\end{align*}
agree.
\qed
\end{lemma}

Thus we may refer to either integral as $I_\Gamma$.

\subsection{Real cohomology classes from graph cocycles}
\label{VanishingAndGraphCohomology}
We next review the arguments for the construction of nontrivial cohomology classes over $\R$ by integration of forms.  Although we discuss the integrals over $F[\Gamma]$, the arguments used in previous literature for integrals over $F[q_1,\dots,q_m;t]$ apply with minimal modification, as suggested by Lemma \ref{IntegralsAreTheSame}.
We will give a topological reinterpretation of these arguments in Section \ref{GluingSection}, in detail, so our treatment here is brief.  

The forms that are integrated along the fiber are sums of products of the $\theta_{ij}$.
One wants to produce cohomology classes, so one needs to produce closed forms.  The fiber $F=F[\Gamma]$ has nonempty boundary, so by Stokes's Theorem the key is to show that, for certain sums of products of $\theta_{ij}$, the integral along the boundary $\int_{\d F} \alpha_\Gamma |_{\d E}$ is zero.  

Since the corner structure of $F[\Gamma]$ comes from that of  $C_\Gamma[\R^d]$, $F[\Gamma]$ has codimension-1 faces corresponding to those defined in Definition \ref{TypesOfFaces} (though the dimension of some of these faces is lower in $F[\Gamma]$ because some points are constrained to lie on the link).  If we let the link in the base space $\L^d_m$ vary, we can alternately think of these faces as faces of the fiber $F[\Gamma]$ and faces of the total space $X=X[\Gamma]$.  
The arguments for the vanishing along boundary faces differs depending on the type of the face.
\begin{remark}
\label{TypesOfVanishing}
For $d\geq 4$, the vanishing arguments for integrals along these types of faces are respectively as follows:
\begin{itemize}
\item
The integrals over principal faces do not vanish.  Instead, choosing appropriate linear combinations of integrals ensures that these contributions cancel.  We can call this reason for vanishing \emph{cancellation}.
\item  
For certain hidden faces, one argues the vanishing by \emph{symmetry}.  That is, 
the configuration space over which one integrates has an involution which either preserves the form and reverses the orientation, or multiplies the form by $-1$ and preserves the orientation.  Thus the integral is equal to its own negative and must vanish.
\item
The integrals over the remaining hidden faces, as well as faces at infinity, vanish because of \emph{degeneracy}: the image of this face under the map $\Phi$ defined in (\ref{BigPhi}) has positive codimension.  
\end{itemize}
\end{remark}

Thus the problem of constructing closed forms reduces to finding linear combinations of integrals $\sum_i c_i I_{\Gamma_i}$ for which the principal face contributions cancel.  
Via the association $\Gamma \mapsto I_\Gamma$ of configuration space integrals to graphs, one can rephrase this problem in terms of the graph complex $\LD$ defined in Section \ref{GraphComplex}.  
Part (1) of the Theorem below implies that the linear combinations of integrals which are closed forms correspond exactly to those $\sum c_i \Gamma_i \in \LD$ which are cocycles. 

\begin{theorem}[\cite{CCRL-AGT, KMV}]
\label{FTFTI-Analogue} \
\begin{enumerate}
\item
The association $\Gamma \to I_\Gamma$ is a chain map 
\begin{equation}
\label{IntegrationChainMap}
I: \bigoplus_{k,n} \LD^{k,n}
\to \Omega_{dR}^{*}(\L^d_m)
\end{equation}
 from the cochain complex of graphs to the de Rham cochain complex on the space of links $\L^d_m$, provided $d\geq 4$.   It sends a cochain in $\LD^{k,n}$ to a de Rham cochain of degree $n(d-3)+k$.
\item
For defect $k=0$, the integration map $I$ induces an injection in cohomology, producing a nontrivial cohomology class in $H^{n(d-3)}(\L^d_m)$ for each nontrivial cocycle in $\LD^{0,n}$.
\end{enumerate}
\end{theorem}

An analogue of statement (1) was given in \cite{CCRL-AGT} for the case of closed knots $m=1$.  The case of long links was treated in \cite{KMV}.  (The case of closed links can also be covered by all the arguments in \cite{KMV}.)  
Since the coboundary map raises the defect by 1, the space of defect-0 cocycles is the space of defect-0 cohomology classes. 
Part (2) was proven in \cite{CCRL-AGT} for closed knots, and the proof there can be easily generalized to links (long or closed).  

While integration $I$ is not known to be a cochain map for $d=3$, this deficiency can be corrected by adding certain ``anomaly terms,'' which conjecturally vanish.  In that case, the injectivity result holds, and in fact, integration then gives an isomorphism \cite{DThurstonABThesis, VolicBT, KMV} between the space $Z(\LD^{0,n})$ of defect-0, order-$n$ cocycles and the space of all finite-type knot or link invariants of type $n$, studied by Vassiliev \cite{Vassiliev} and others.
The former space is the space of (uni)trivalent graph cocycles, and this isomorphism is sometimes called the Fundamental Theorem of Finite-Type Invariants.  
Theorem \ref{FTFTI-Analogue} is thus an analogue of that theorem on knot and link invariants to cohomology classes of higher degree.
Moreover, work of Bar-Natan \cite{BarNatanTopology} implies that the cohomology of defect-0 graphs $Z(\LD^{0,*})$ has many nontrivial elements.  Thus Theorem \ref{FTFTI-Analogue} gives the existence of many nontrivial cohomology classes in embedding spaces.

\begin{example}
\label{Type2Example}
The simplest nontrivial example for odd $d$, illustrates the above Theorem.  
Consider $L=\R$, and define $\gamma_2^{\mathrm{odd}} =\mathrm{X} -\mathrm{Y}$ by
\begin{equation}
\label{Type2Graphs}
\gamma_2^{\mathrm{odd}} := 
\raisebox{-1pc}{\includegraphics[scale=0.3]{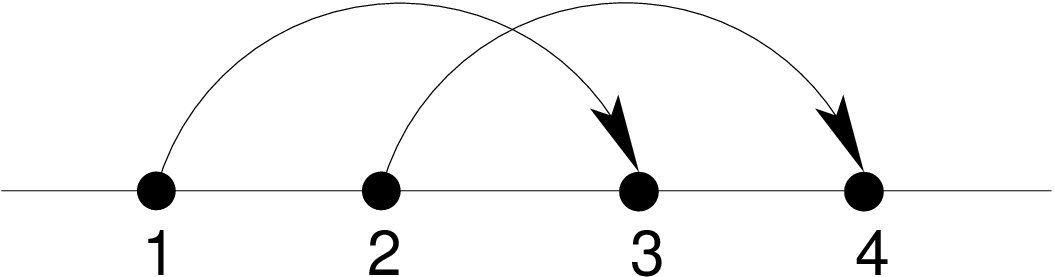}}\qquad - \qquad \raisebox{-1pc}{\includegraphics[scale=0.3]{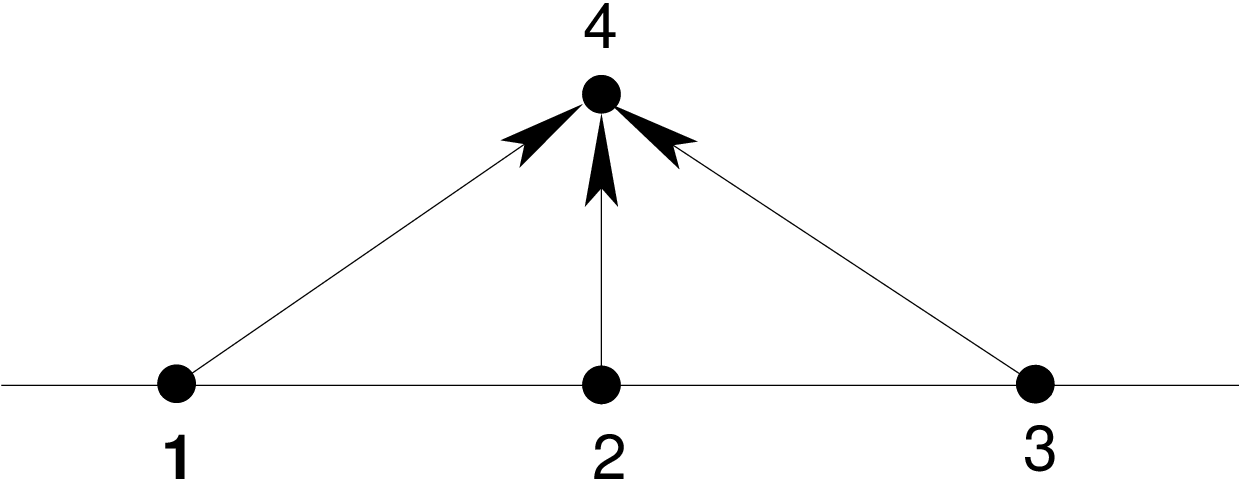}}
\end{equation}
Bott and Taubes considered a closed knot analogue of this cocycle (i.e.~with $S^1$ instead of $\R$), where the coefficients are related to symmetries of the circular versions of these graphs.
One can show that $\gamma_2^{\mathrm{odd}}$ is a cocycle, since the contraction of each of the four arcs of the circle in the first graph produces the same graph as the contraction of each of the three edges in the second graph.
The associated integral
\begin{equation}
\label{Type2Integrals}
I(\gamma_2^{\mathrm{odd}}) = \frac{1}{4}\int_{F[4;0]} \theta_{13} \theta_{24} - \frac{1}{3}\int_{F[3;1]} \theta_{14} \theta_{24}\theta_{34}
\end{equation}
is for $d=3$ an invariant of long knots, and moreover finite-type of order 2.
For either graph, the underlying trivalent graph after compactifying at infinity is a complete graphs, so the compactified configuration space fibers $F[\Gamma_i]$ are in fact the Axelrod--Singer compactifications.  That is, the distinction between their compactification and the one we use (see Remark \ref{BottTaubesBundle}) does not exist here.  There is a similar cocycle of order 2 for even $d$, which by a mild abuse of notation (i.e.~ignoring orientations) we may also write as $\gamma_2^{\mathrm{even}}=\mathrm{X}-\mathrm{Y}$:
\begin{equation}
\label{Type2GraphsEven}
\gamma_2^{\mathrm{even}} := 
\raisebox{-1pc}{\includegraphics[scale=0.3]{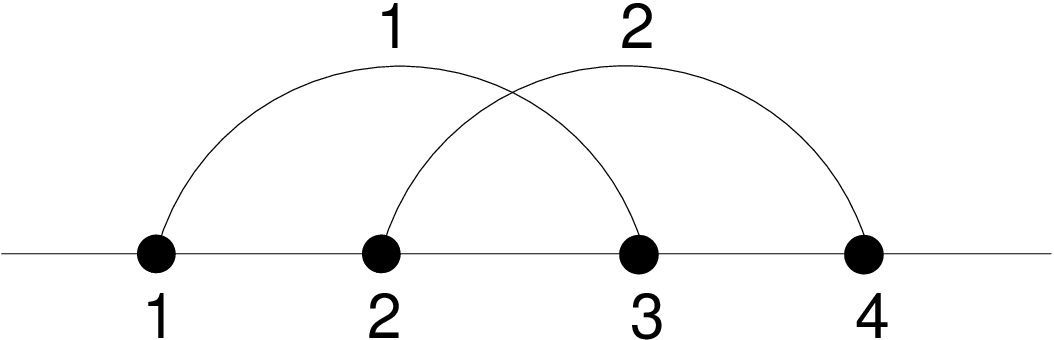}}\qquad - \qquad \raisebox{-1pc}{\includegraphics[scale=0.3]{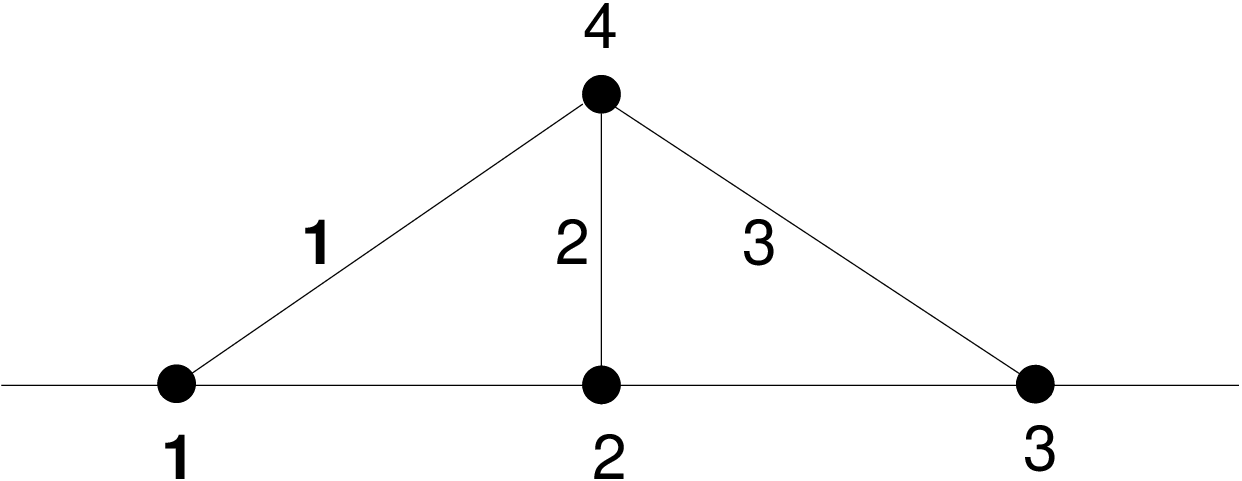}}
\end{equation}
\end{example}

\section{Gluing configuration spaces}
\label{GluingSection}
We now provide a recipe for associating to a graph cocycle $\gamma$ a space obtained by gluing together configuration spaces.  
We will essentially glue together the bundles $X[\Gamma_i]$ for the various graphs $\Gamma_i$ in the cocycle $\gamma$.
This will be a fiberwise construction, meaning that we fix a link $L$ in the base space $\L^d_m$ and then glue together the various fibers $F[\Gamma_i]$ over $L$ to produce a space $F_\gamma$ for each $L$.  
The space $F_\gamma$ will have a fundamental class, relative to $\d F_\gamma$.
Thus allowing the link $L$ to vary yields a bundle $X_\gamma$ over $\L^d_m$ whose fiber has a fundamental class, relative to its boundary.
The gluing will be done so that the space $X_\gamma$ will have a map $\Phi_\gamma$ to an appropriate quotient of a product of $(d-1)$-spheres.  The image under $\Phi_\gamma$ of the remaining boundary $\d X_\gamma$ will have positive codimension, so this boundary can be thought of as the degenerate locus. 
This will allow us to pull back the top-dimensional cohomology class on the product of spheres and use the fundamental class  $[F_\gamma, \d F_\gamma]$ to produce a cohomology class in $\L^d_m$.  The resulting class will correspond to the fiberwise integral of the product of spherical forms.

In Section \ref{IntegralGraphComplex}, we consider integer-valued cocycles $\sum c_i \Gamma_i$, which are the input for our construction.  To such a cocycle we associate a disjoint union of {oriented} configuration spaces $\tX[\Gamma_i]$ closely related to the $X[\Gamma_i]$.  In Section \ref{GluingSubsection} we glue those $\tX[\Gamma_i]$ along their principal faces.  
After gluing, it remains to fold (Section \ref{HiddenFaceInvolution}) and collapse (Section \ref{Collapses}) some faces.  At each of these steps, we observe that the space has a fundamental class relative to the remaining boundary faces.  We will then work relative to these remaining faces, which are listed in Definition \ref{DegenerateLocusDef}.  Finally, in Section \ref{MapToSpheres} we establish the appropriate map to a product of spheres.

\subsection{Preliminaries for the gluing construction}
We address an issue of orientations for comparing  integer-valued cocycles to real-valued cocycles in Section \ref{IntegerCocyclesOrientations}.  We provide a slight refinement of integer-valued cocycles in Section \ref{MinimalCocycles}.  Then in Section \ref{ConfigSpacesToBeGlued}, we describe the associated configuration space bundles, which are slightly modified from bundles of the form \eqref{GammaBTsquare} in Section \ref{BundleOverLinkSpace}.  We orient these bundles (i.e.~their fibers) in Section \ref{OrientationsOnSpacesToBeGlued}.  In Section \ref{ComparisonsToEarlierWork}, we compare the constructions we use throughout the whole Section \ref{GluingSection} to earlier work of other authors.

\label{IntegralGraphComplex}
\subsubsection{Integer-valued cocycles and an issue of orientations}
\label{IntegerCocyclesOrientations}
Recall the definition of the graph complex $\LD$ over a ring $R$ from Section \ref{GraphComplex}.
We first restrict our attention to $R=\Z$.  
Because of the orientation relations in Definition \ref{GraphComplexDef}, the cochain groups need not be free.  In particular, a graph $\Gamma$ with an orientation-reversing automorphism satisfies $\Gamma = - \Gamma$.  Thus the universal coefficient theorem does not apply, and the ranks of the cohomologies over $\Z$ and $\R$ need not be related.  Since we want to recover cocycles originally constructed over $\R$, we amend Definition \ref{GraphComplexDef} by imposing the relation in $\LD$ that 
\begin{equation}
\label{GraphsWithOrRevAutoZero}
\mbox{if $\Gamma$ has an orientation-reversing automorphism, then $\Gamma \sim 0$.}
\end{equation}
By abuse of notation, we will from now until the end of Section \ref{Results} use $\LD$ to denote these free chain groups, where the ring $R$ is still suppressed from the notation.  We view these as bigraded groups $\LD^{*,*}$ as before (Definition \ref{DefectOrderDef}).  The universal coefficient theorem then applies to give $H^*(\LD;\, \R) \cong H^*(\LD;\, \Z) \otimes \R$, where $H^*(\LD;\, R)$ denotes the cohomology of the complex $\LD$  over $R$.  In defect zero, there are no coboundaries, so the space of defect-0 cohomology classes is precisely space of defect-0 cocycles, so we only need to use that $Z^{0,n}(\LD;\, \R) \cong Z^{0,n}(\LD;\, \Z) \otimes \R$, by definition.  

\subsubsection{Minimal cocycles}
\label{MinimalCocycles}
Even after imposing the relation \eqref{GraphsWithOrRevAutoZero}, the orientation relations in $\LD$ pose a second, arguably more minor problem.  An integer-valued cocycle $\gamma \in Z^*(\LD;\,\Z)$ can be written as a finite sum
\[
\gamma = \sum  c_i \Gamma_i
\]
where the  $\Gamma_i$ are oriented diagrams (see Definitions \ref{LabeledGraphs} and \ref{GraphComplexDef}).  
However, this expression is not unique.  The following definitions do not guarantee a unique expression, but they allow us to restrict to relatively efficient ones.

\begin{definitions}
\label{SimplifiedAndMinimalDefs}
For an integral cocycle $\gamma$, call a representative expression $\sum c_i \Gamma_i$ \emph{simplified} if the underlying unoriented graphs $|\Gamma_i|$ are in distinct isomorphism classes, and if none of the $\Gamma_i$ is 0 in $\LD$.  Define $S_\gamma$, the \emph{support} of $\gamma$, as the set of unoriented isomorphism classes which appear in a simplified expression for $\gamma$:
\[
S_\gamma :=\{ |\Gamma_i| : c_i \neq 0 \}.
\]
Call a cocycle $\gamma = \sum c_i \Gamma_i$ \emph{minimal} if both of the conditions below are satisfied:
\begin{enumerate}
\item there is no cocycle $\beta = \sum b_i \Gamma_i$ such that $S_\beta$ is a nonempty proper subset of $S_\gamma$ and such that $b_i = c_i$ for some $i$, and
\item there is no cocycle $\beta$ such that $\gamma = c \beta$ for some $c \in \Z$ with $|c|>1$.
\end{enumerate}
\end{definitions}

Any expression for a cocycle $\gamma$ can be turned into a simplified one, by possibly relabeling the $\Gamma_i$, adding minus signs if necessary, and combining like terms.  Thus $S_\gamma$ is always defined, and it does not depend on the choice of simplified expression for $\gamma$.  Note that in a simplified expression, no $\Gamma_i$ has a multiple edge or an orientation-reversing automorphism.
The next Proposition allows us to restrict our attention to minimal cocycles.  
\begin{proposition}
Any cocycle can be decomposed into a sum of minimal cocycles.
\end{proposition}
\begin{proof}
If an integer-valued cocycle $\gamma$ fails to be minimal because of condition (1) above, then we can write $\gamma = (\gamma - \beta) + \beta$, where the support of each summand on the right-hand side is strictly smaller than that of $\gamma$.  Iterating this step a finite number of times decomposes $\gamma$ into a sum of cocycles satisfying (1) above.  If any summand fails to satisfy (2) as above, we may rewrite it as a sum of $|c|$ terms each of which is minimal.
\end{proof}

\subsubsection{The configuration spaces to be glued}
\label{ConfigSpacesToBeGlued}
Fix a minimal integer-valued cocycle $\gamma$.  Let $N\geq 1$ be any integer at least as large as the maximum number of edges over all the $\Gamma_i$ in $S_\gamma$.  Following terminology of Poirier \cite{PoirierAGT}, we let $a_i:= N - |E(\Gamma_i)|$ be the number of ``absent edges'' in $\Gamma_i$.
Our first step in constructing the glued space $X_\gamma$ is to start with $|c_i|$ disjoint copies of $(S^{d-1})^{a_i} \x X[\Gamma_i]$ for each $i$.  We will abbreviate this total space as $\widetilde{X}[\Gamma_i] := (S^{d-1})^{a_i} \x X[\Gamma_i]$ and its fiber as $\widetilde{F}[\Gamma_i]:=(S^{d-1})^{a_i} \x F[\Gamma_i]$.  By a slight abuse of notation, we will write $\Phi_\Gamma$ for the map from $\widetilde{X}[\Gamma]$ (or $\widetilde{F}[\Gamma]$) to $(S^{d-1})^N$.  We will write $c_i \tX[\Gamma_i]$ for a disjoint union of $|c_i|$ copies of $\tX[\Gamma_i]$ where the sign of $c_i$ affects only the orientation, as described below.  We also write $c_i \Phi_{\Gamma_i}$ for the obvious map $c_i \tX[\Gamma_i] \to (S^{d-1})^N$.

\subsubsection{Orientations on the spaces to be glued}
\label{OrientationsOnSpacesToBeGlued}
Orientations were of course implicitly used in the integration of forms in Section \ref{FiberwiseIntegrals}.  We will now consider them in some detail, so we fix an orientation on $M=\R^d$.  For every $\Gamma_i$, this together with the orientation of $\Gamma_i$ determines an orientation on $M^{V(\Gamma_i)}$ and hence on $C_{\Gamma_i}[\R^d]$.  
By fixing an (arguably canonical) orientation on $\R$, an orientation of $\Gamma_i$ yields an orientation on the fibers $F[\Gamma_i]$ of the bundle $X[\Gamma_i] \to \L^d_m$.  

To address the sphere factors and boundary faces in $\widetilde{F}[\Gamma_i]$, orient all boundaries of manifolds using the convention of outward-pointing normal vector first.  
We then orient $S^{d-1}$ as the boundary of the unit ball in $\R^d$.  This determines an orientation of $\widetilde{F}[\Gamma_i]$.  
If $c_i>0$, take this orientation on all $|c_i|$ copies of $\widetilde{F}[\Gamma_i]$,
while if $c_i<0$, take the opposite of this orientation on all copies of $\widetilde{F}[\Gamma_i]$.
The fiber $F_\gamma$ of the glued space will be a quotient of this disjoint union of oriented manifolds with corners.
There is an obvious bijection between the faces and corners of $\widetilde{F}[\Gamma_i]$ 
and those of $F[\Gamma_i]$.  
We orient a codimension-1 face $\SS(\Gamma_i, \Gamma')$ as part of the boundary of $\widetilde{F}[\Gamma_i]$.  This boundary orientation will be important in the principal face gluings in Section \ref{GluingSubsection}, specifically in Lemma \ref{GluingInteriorsOfPrincipalFaces}.

\subsubsection{Comparisons to earlier work}
\label{ComparisonsToEarlierWork}
For the remainder of Section \ref{GluingSection}, we will recast in a more topological framework all the vanishing arguments alluded to in Remark \ref{TypesOfVanishing}.  
We now compare our steps below to earlier works containing those arguments:
\begin{itemize}
\item 
In Section \ref{GluingSubsection}, we construct a glued space for each graph cocycle by \emph{gluing principal faces} two at a time, whereas Kuperberg and Thurston \cite{KuperbergThurston} construct one universal glued space by identifying more than two principal faces at once.  
\item 
The \emph{hidden face involution} $\iota$ in Section \ref{HiddenFaceInvolution} is due to Kontsevich \cite{KontsevichFeynman}, and also found in the work of Bott and Taubes \cite{BottTaubes}, D.~Thurston \cite{DThurstonABThesis}, Bott and Cattaneo \cite{BottCattaneo}, Kuperberg and Thurston \cite{KuperbergThurston}, and Cattaneo, Cotta-Ramusino, and Longoni \cite{CCRL-AGT}.  
\item
The \emph{collapse for bivalent segment vertices with no incident edge} in Section \ref{Collapses} is an analogue of Lemma A.9 of \cite{CCRL-AGT}.
\item
Proposition \ref{ClassInRelCoh} encodes the \emph{degeneracy arguments} and is analogous to statements in \cite[Proposition 4.5, Theorem A.11]{CCRL-AGT}, \cite[Section 4.2]{DThurstonABThesis}, and \cite[Theorem 4.36]{KMV}.
\item
Lemma \ref{PositiveCodimension} is the main step in establishing the \emph{degeneracy arguments} of Proposition \ref{ClassInRelCoh}, and its proof via several cases is a topological restatement of the arguments given in \cite{CCRL-AGT, KMV}.  We have preserved the numbering of the various cases used in those references.  
\end{itemize}

\subsection{Principal face gluings}
\label{GluingSubsection}
In this section, we glue the configuration space bundles along their principal faces.  We first check in Section \ref{GluingInteriorsOfPrincipalFaces} that the interiors of a pair of principal faces with canceling contributions in $\delta \gamma$ can be identified by an orientation-reversing diffeomorphism.  We then check in Section \ref{GluingClosuresOfPrincipalFaces} that this gluing extends to the closures of these faces.
We summarize this part of the construction in Section \ref{ImplementingPrincipalGluings}.

\subsubsection{Gluing interiors of a pair of principal faces}
\label{GluingInteriorsOfPrincipalFacesSubsubsection}
Consider the fiber $\coprod c_i\widetilde{F}[\Gamma_i]$ of the space associated to an expression for $\gamma$.
Recall from Section \ref{CornerDescription} that an edge or arc $e$ of a $\Gamma_i$ corresponds to a principal face of $F[\Gamma_i]$, denoted $\SS(\Gamma_i, e)$ (see Notation \ref{StratumNotation}). 
We use the same notation $\SS(\Gamma_i, e)$ for the corresponding face in $\widetilde{F}[\Gamma_i]$.  
By this correspondence, we can view $\delta \gamma$ as a disjoint union of principal faces of the $\widetilde{F}[\Gamma_i]$ (or $-\widetilde{F}[\Gamma_i]$ in the case of negative $c_i$), which we will glue in pairs.
Recall the signs $\eps(e)$ from \eqref{OddSigns} and  \eqref{EvenSigns} that were used to define the coboundary operator in the graph complex
.

\begin{lemma}
\label{GluingInteriorsOfPrincipalFaces}
 \
\begin{enumerate}
\item
Given labeled graphs $\Gamma_i$ and $\Gamma_j$, an isomorphism $\Gamma_i /e \to \Gamma_j / f$ gives rise to a diffeomorphism of the interiors of the corresponding principal faces $g: \SS(\Gamma_i, e) \to \SS(\Gamma_j, f)$.  
\item
Suppose the isomorphism $\eps(e) \Gamma_i /e \to \eps(f) \Gamma_j / f$ preserves (respectively reverses) orientation.  
Then we can find a diffeomorphism $g: \SS(\Gamma_i, e) \to \SS(\Gamma_j, f)$ which preserves (respectively reverses) orientation such that the spherical maps $\Phi_{\Gamma_i}$ and $\Phi_{\Gamma_j} \circ g$ agree up to a diffeomorphism of $(S^{d-1})^N$ that is orientation-preserving (in both cases). 
\end{enumerate}
\end{lemma}

\begin{proof}
We will first prove both statements for the case of $d$ odd.
Let $v$ and $w$ be the endpoints of $e$, and suppose $e$ is oriented from $v$ to $w$.  
An interior point of $\SS(\Gamma_i, e)$ is given by a configuration of points in $\Gamma_i/e$ together with a direction of collision between $x_v$ and $x_w$.  We take this to be the direction from $x_v$ to $x_w$.  There are two cases for this direction of collision.  
\begin{itemize}
\item[(i)]
If $v$ and $w$ are both segment vertices, then this is a direction in $\R$, i.e., a point in $S^0=\{+,-\}$, necessarily $+$ since the arc orientation coincides with the orientation of $\R$. 
Then 
\begin{equation}
\label{PrincipalFaceIso1}
\SS(\Gamma_i, e) \cong \widetilde{F}[\Gamma_i/e] = (S^{d-1})^{a_i} \x {F}[\Gamma_i/e].
\end{equation}
\item[(ii)]
Otherwise, it is a direction in $\R^d$, i.e., a vector in $S^{d-1}$.    Then 
\begin{equation}
\label{PrincipalFaceIso2}
\SS(\Gamma_i, e) \cong \widetilde{F}[\Gamma_i/e] = (S^{d-1})^{a_i+1} \x {F}[\Gamma_i/e].
\end{equation}
\end{itemize}
(The space $\widetilde{F}[\Gamma_i/e]$ has an orientation as in Section \ref{OrientationsOnSpacesToBeGlued}.   We address whether it agrees with that of $\SS(\Gamma_i, e)$ in proving part (2) below.)
A similar dichotomy (i)-(ii) applies to $\SS(\Gamma_j, f)$.
The numbers $a_i$ and $a_j$ of absent edges in $\Gamma_i$ and $\Gamma_j$ agree precisely when either both are in case (i) or both are in case (ii).  In these cases, it suffices to identify ${F} [\Gamma_i/e] \to {F}[\Gamma_j/f]$ via the graph isomorphism $\Gamma_i/e \to \Gamma_j/f$.  

The remaining cases are when $a_i$ and $a_j$ differ by 1.  Without loss of generality, suppose $a_j=a_i+1$.  
Then the collision in $\Gamma_j$ is of two segment vertices, so there is no direction of collision to specify.  In $\Gamma_i$ however it is a collision of a free vertex with a segment vertex, parametrized by a vector in $S^{d-1}$.
We map 
$\widetilde{F}[\Gamma_i/e]$ to $\widetilde{F}[\Gamma_j/f]$ 
by sending this vector into the first $S^{d-1}$ factor in 
$\widetilde{F}[\Gamma_j/f]$.  In summary, in each case, we have a composition of diffeomorphisms:
\begin{equation}
\label{PrincipalFaceIso3}
\SS(\Gamma_i, e) \cong \widetilde{F}[\Gamma_i/e] \cong \widetilde{F}[\Gamma_j/f] \cong \SS(\Gamma_j,f)
\end{equation}
This completes the proof of statement (1) for $d$ odd.

For statement (2), one can define a map $\Phi_{\Gamma_i/e}: \tF[\Gamma_i/e] \to (S^{d-1})^N$.  
Because of the labeling on $\Gamma_i/e$, the map $\Phi_{\Gamma_i/e}$ is (like $\Phi_{\Gamma_i}$) canonical up to permutations of the $N$ factors for odd $d$ and up to antipodal maps on the factors for even $d$.  Both of these types of transformations preserve orientation.
We can find maps $(S^{d-1})^N \to  (S^{d-1})^N$ making the following diagram commute, where the horizontal arrows are diffeomorphisms:
\begin{equation}
\label{ProductOfSigns}
\xymatrix{
\SS(\Gamma_i, e) \ar@{<->}[r]\ar[d]_-{\Phi_{\Gamma_i}} & \widetilde{F}[\Gamma_i/e] \ar@{<->}[r] \ar[d]_-{\Phi_{\Gamma_i/e}} & \widetilde{F}[\Gamma_j/f] \ar[d]^-{\Phi_{\Gamma_j/f}} & \ar@{<->}[l] \SS(\Gamma_j, f) \ar[d]^-{\Phi_{\Gamma_j}}  \\
(S^{d-1})^N\ar@{<->}[r] & (S^{d-1})^N \ar@{<->}[r]^-{\mathrm{id}} & (S^{d-1})^N & \ar@{<->}[l] (S^{d-1})^N
}
\end{equation}
Associate to any isomorphism (like those above) is a sign $\pm 1$ according as it preserves or reverses orientation.  
One can verify from our definitions and conventions that the sign $\eps(e)$ is a product of the two signs in the left-hand side of the diagram, while the sign $\eps(f)$ is the product of the two signs in the right-hand side of the diagram.  The sign of the top middle horizontal arrow is the sign $\nu$ of the graph isomorphism $\Gamma_i /e \to \Gamma_j / f$.  If the product of signs in the top row is $\eps(e) \nu \eps(f)$ and the product of signs in the bottom row is $+1$, we are done.  If not, then both of these products are the wrong sign.  In this case, we modify the maps in the middle square by composing by an antipodal map on some factor of $S^{d-1}$ in the product \eqref{PrincipalFaceIso1} or \eqref{PrincipalFaceIso2}.  (Such a factor exists because in \eqref{PrincipalFaceIso1}, we must have $a_i \geq 1$ and in \eqref{PrincipalFaceIso2}, $a_i \geq 0$.)

Finally, suppose $d$ is even.  In this case, we arbitrarily choose a direction between $x_v$ and $x_w$, since it has no effect on orientation.  We get a commutative diagram of diffeomorphisms as in \eqref{ProductOfSigns} for the odd case.  To get the desired effect on orientation, we cannot use the antipodal map, but we can instead use an orientation-reversing diffeomorphism on some factor of $S^{d-1}$, such as a reflection in a hyperplane.
\end{proof}

\subsubsection{Collapses induced by gluing closures of principal faces}
\label{GluingClosuresOfPrincipalFaces}
We now want to extend a diffeomorphism $g: \mathrm{int}(\SS(\Gamma_i, e)) \to \mathrm{int}(\SS(\Gamma_j f))$ of the interiors of principal faces to the closures of these faces.  
Because the compactification $C_\Gamma[\R^d]$ is a blowup of only \emph{some} diagonals, depending on $\Gamma$, it is possible that two corresponding corners of $\SS(\Gamma_i, e)$ and $\SS(\Gamma_j, f)$ are not diffeomorphic and may even have different codimension.

In more detail, for a corner $\SS(\Gamma_i, S)$ of the face $\SS(\Gamma_i,e)$, $S$ is a set of biconnected subgraphs of $\Sigma \Gamma_i$ with $e \in S$ and satisfying the properties given in Section \ref{CornerDescription}.
Since $\Gamma_i/e \cong \Gamma_j/e$, the graphs $\Gamma_i$ and $\Gamma_j$ differ precisely by an edge reconnection as in Figure \ref{CollapsingCornersFig}, that is, a re-partitioning of those edge-ends incident to an endpoint of $e$.  
The same is true of their suspensions.  
After this reconnection, the corresponding subgraph in $\Sigma \Gamma_j$ may or may not be biconnected.  
See Figures \ref{CollapsingCornersFig} and \ref{Cicadas}.

\begin{figure}[h!]
\includegraphics[scale=0.27]{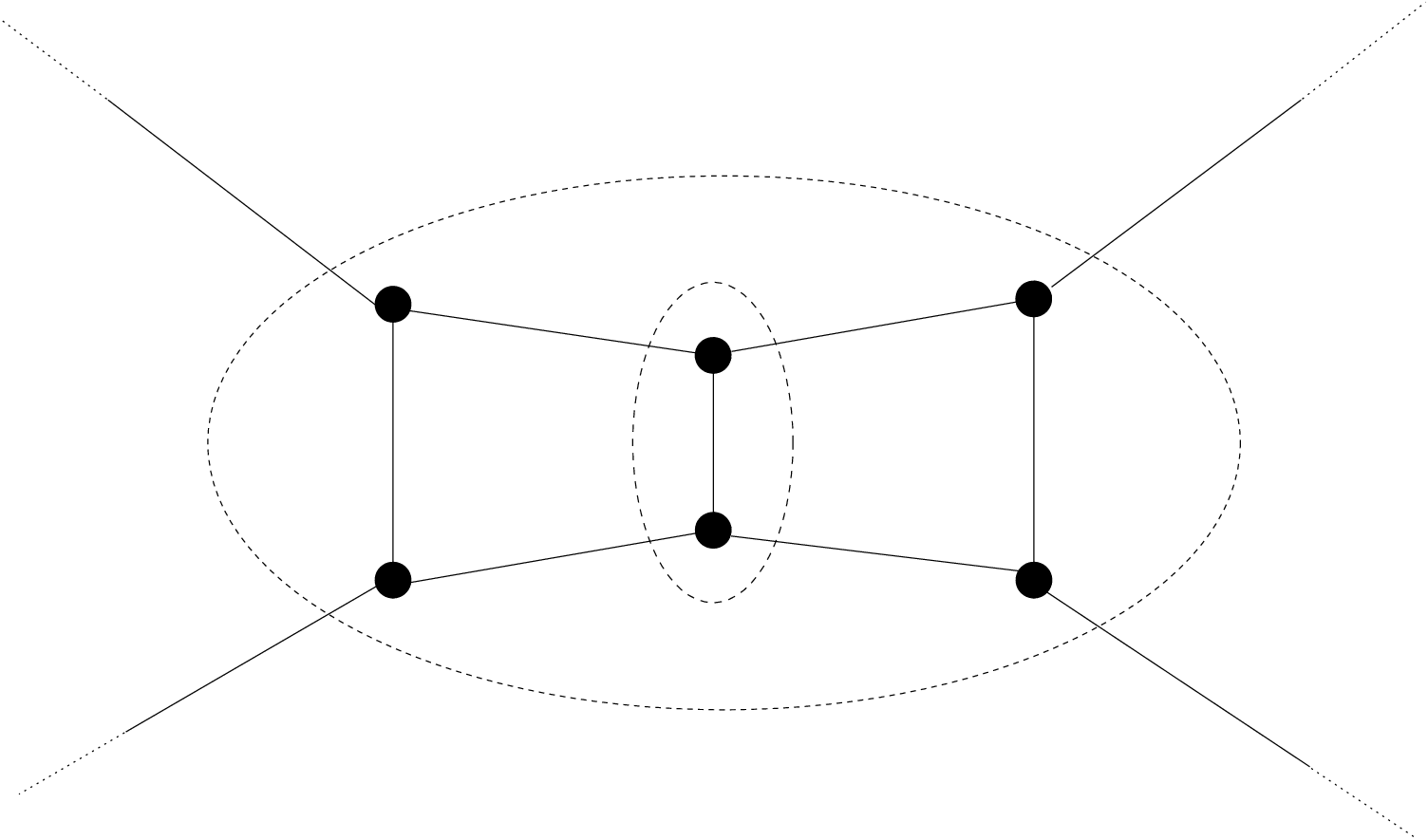}
\qquad \qquad
\includegraphics[scale=0.27]{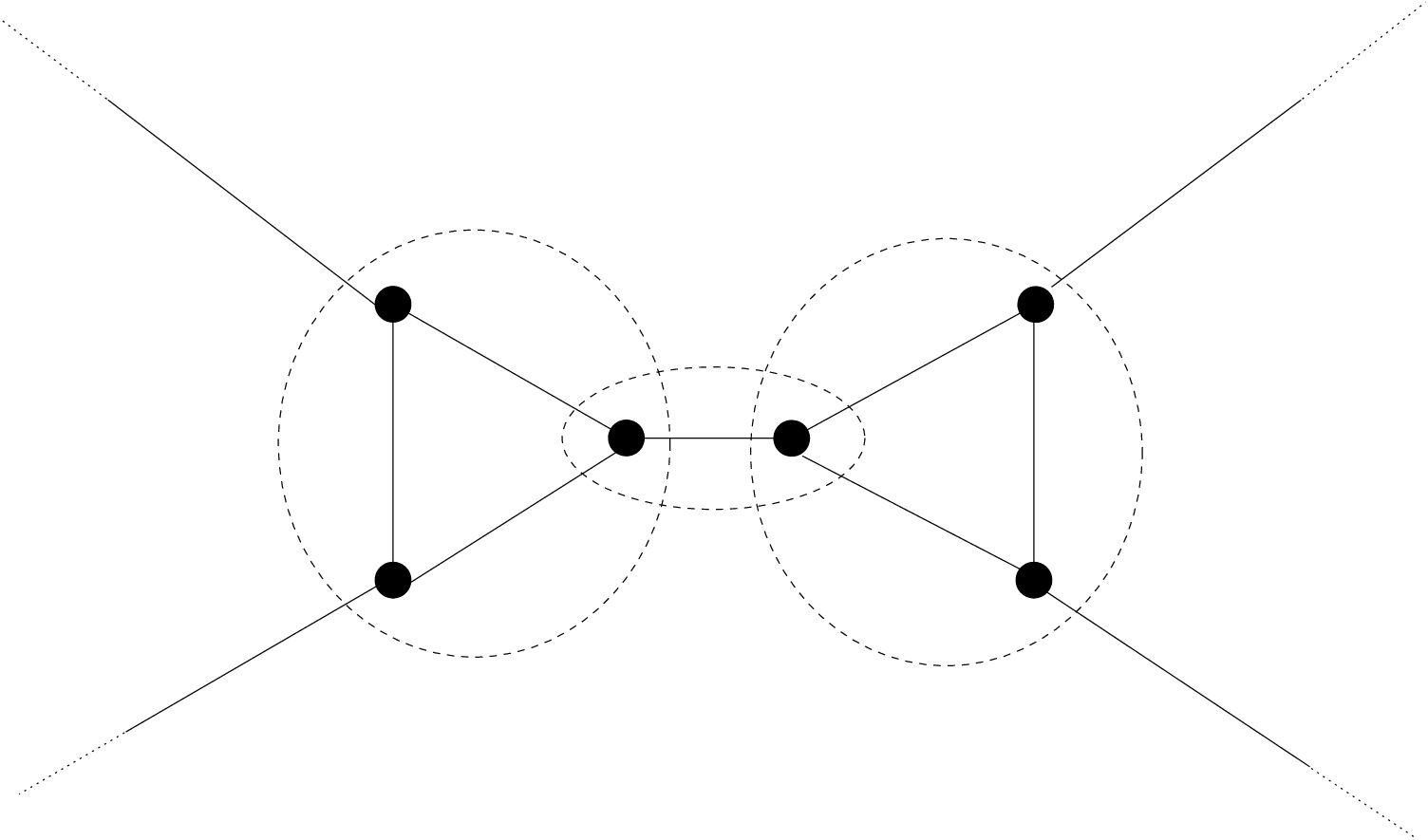}
\caption{This is an example where a principal face gluing identifies corners of different dimension, as in \cite{KuperbergThurston}.  The two graphs (with orientations omitted) agree outside of the pictured portions.  Contracting the middle edge in each figure gives isomorphic graphs, so the interiors of the corresponding principal faces can be glued together.  Consider the codimension-2 face of the left-hand space indexed by the circled subgraphs, i.e., the double-square and the middle edge.  The six vertices in the right-hand graph do not form a biconnected subgraph, so the corresponding corner of the right-hand principal face (where all six points have collided) is indexed by the left triangle, the middle edge, and the right triangle, and is thus a codimension-3 face.  The canonical way to identify these corners is via the blow-down map from the left-hand, codimension-2 corner to the right-hand, codimension-3 corner.}
\label{CollapsingCornersFig}
\end{figure}

\begin{definition}
\label{CornerGluing}
Suppose $g: \mathrm{int}(\SS(\Gamma_i, e)) \to \mathrm{int}(\SS(\Gamma_j f))$ is a gluing of the interiors of two principal faces of $\widetilde{F}[\Gamma_i]$ and $\widetilde{F}[\Gamma_j]$ coming from a graph isomorphism $\Gamma_i/e \cong \Gamma_j/f$.  Let $S$ index a corner of $\SS(\Gamma_i, e)$ (so $S$ is a set of biconnected subgraphs of $\Sigma \Gamma_i$ with $e \in S$).  
We construct $g(S)$ as a set of biconnected subgraphs of $\Sigma\Gamma_j$, starting with $g(S)=\varnothing$, as follows.
The isomorphism $\Gamma_i/e \cong \Gamma_j/f$ induces an isomorphism $\Gamma_i \setminus e \cong \Gamma_j \setminus f$.
Let $\Gamma''$ be a subgraph of $\Sigma\Gamma_j$ corresponding to a subgraph $\Gamma' \in S$ via this isomorphism and a fixed identification of $e$ with $f$.
\begin{itemize}
\item
For each such $\Gamma''$ that is biconnected, add $\Gamma''$ to $g(S)$.    
\item
For each such $\Gamma''$ that is not biconnected and contains $f$, add the maximal biconnected components $\Gamma''_1, \dots, \Gamma''_j$ of $\Gamma''$ to $g(S)$.  
\item
Add $\bigcup \Gamma'' \cup f$ to $g(S)$, where the union is taken over subgraphs $\Gamma''$ which are not biconnected and which intersect $f$ in one vertex.
\end{itemize}
\end{definition}

\begin{lemma}
\label{ExtendingGlueToClosures}
Let $g: \mathrm{int}(\SS(\Gamma_i, e)) \to \mathrm{int}(\SS(\Gamma_j, f))$ be a diffeomorphism coming from a graph automorphism $\Gamma_i/e \cong \Gamma_j/f$.  Let $S$ index a corner of $\SS(\Gamma_i, e)$, and let $g(S)$ be defined as in Definition \ref{CornerGluing}.  Then $g$ extends to a homeomorphism $g: \SS(\Gamma_i, e) \to \SS(\Gamma_j, f)$ of the closures by gluing the images of $\SS(\Gamma_i, S)$ to $\SS(\Gamma_i, g(S))$ after certain canonical collapses, given by forgetting relative rates of approach.  
Locally, the collapses are submersions of Euclidean spaces.
\end{lemma}
\begin{proof}
The reader may verify as in Section \ref{CornerDescription} that $g(S)$ indeed indexes a corner of the face $\SS(\Gamma_j, f)$.  
We will check that the corners can be identified as claimed.
We will proceed by showing that at each of the three steps in constructing $g(S)$, we obtain a diffeomorphism of the relevant screen spaces (or products thereof), after certain collapses.  

Let $\Gamma' \in S$, and let $\Gamma''$ be the corresponding subgraph of $\Gamma_j$.  
First suppose $\Gamma''$ is biconnected.  Then we may simply identify the screen spaces for $\Gamma'$ and $\Gamma''$.  (See Section \ref{CornerDescription} for the definition of the screen spaces.)
Indeed, though the screen space for $\Gamma'$ is indexed by vertices in a \emph{quotient} of $\Gamma'$, the edge reconnection induces the corresponding relation on $\Gamma''$.  In the remaining cases, we will thus make the simplifying assumption that the screen space for $\Gamma'$ is indexed by vertices in $\Gamma'$, etc.

Next consider the case where $\Gamma' \in S$ yields components $\Gamma''_1, \dots, \Gamma''_k \subset \Gamma''$ in $g(S)$.
Here we will replace the screen space for $\Gamma'$  by a product of screen spaces, one for each $\Gamma''_\ell$.  More precisely, we perform a canonical collapse in $\SS(\Gamma_i, e)$ given by the projections induced by the inclusions $V(\Gamma''_\ell) \incl V(\Gamma'') \cong V(\Gamma')$:
\begin{equation}
\label{CanonicalCollapse}
  C_{\Gamma'}(T_p \R^d)/(\R^d \rtimes \R_+)
 \longrightarrow
\prod_{\ell=1}^k C_{\Gamma''_\ell}(T_p \R^d)/(\R^d \rtimes \R_+).
\end{equation}
This dimension of the codomain is $k-1$ less than the dimension of the domain.
A $(k-1)$-tuple of ratios in $[0,\infty]$ specifying the scales of the screens (relative to say the first one) gives a diffeomorphism
\[
[0, \infty]^{k-1} \x \prod_{\ell=1}^k C_{\Gamma''_\ell}(T_p \R^d)/(\R^d \rtimes \R_+)
\overset{\cong}{\longrightarrow} C_{\Gamma'}(T_p \R^d)/(\R^d \rtimes \R_+).
\]
The collapse \eqref{CanonicalCollapse} is thus given by forgetting these ratios, or relative rates of approach, so it is locally a submersion of Euclidean spaces.  After performing that collapse, we can identify the image of the screen space for $\Gamma'$ with the product of the screen spaces for the $\Gamma''_\ell$.

For the third case in Definition \ref{CornerGluing}, first suppose that there is only one subgraph $\Gamma'$ yielding $\Gamma'' \cup f$.  Then the screen space for $\Gamma'$ is canonically diffeomorphic to that for $\Gamma'' \cup f$.  This is because $f \subset \Gamma'' \cup f$, so the two endpoints of $f$ are identified in the latter screen space.

Finally, it remains to consider the case where subgraphs $\Gamma'_1, \dots, \Gamma'_k \in S$ and corresponding subgraphs $\Gamma''_\ell \subset \Sigma\Gamma_j$ yield one subgraph $\Gamma''_1 \cup \dots \cup \Gamma''_k \cup f =: \Gamma'' \in g(S)$.  Here we implement a canonical collapse in $\SS(\Gamma_j, f)$, analogous to \eqref{CanonicalCollapse} and given by 
\begin{equation}
\label{CanonicalCollapse2}
  C_{\Gamma''}(T_p \R^d)/(\R^d \rtimes \R_+)
 \longrightarrow
\prod_{\ell=1}^k C_{\Gamma''_i}(T_p \R^d)/(\R^d \rtimes \R_+).
\end{equation}
Then we can identify the product of screen spaces for the $\Gamma'_i$ with the screen space for $\Gamma''$.
\end{proof}

\begin{figure}
\includegraphics[scale=0.27]{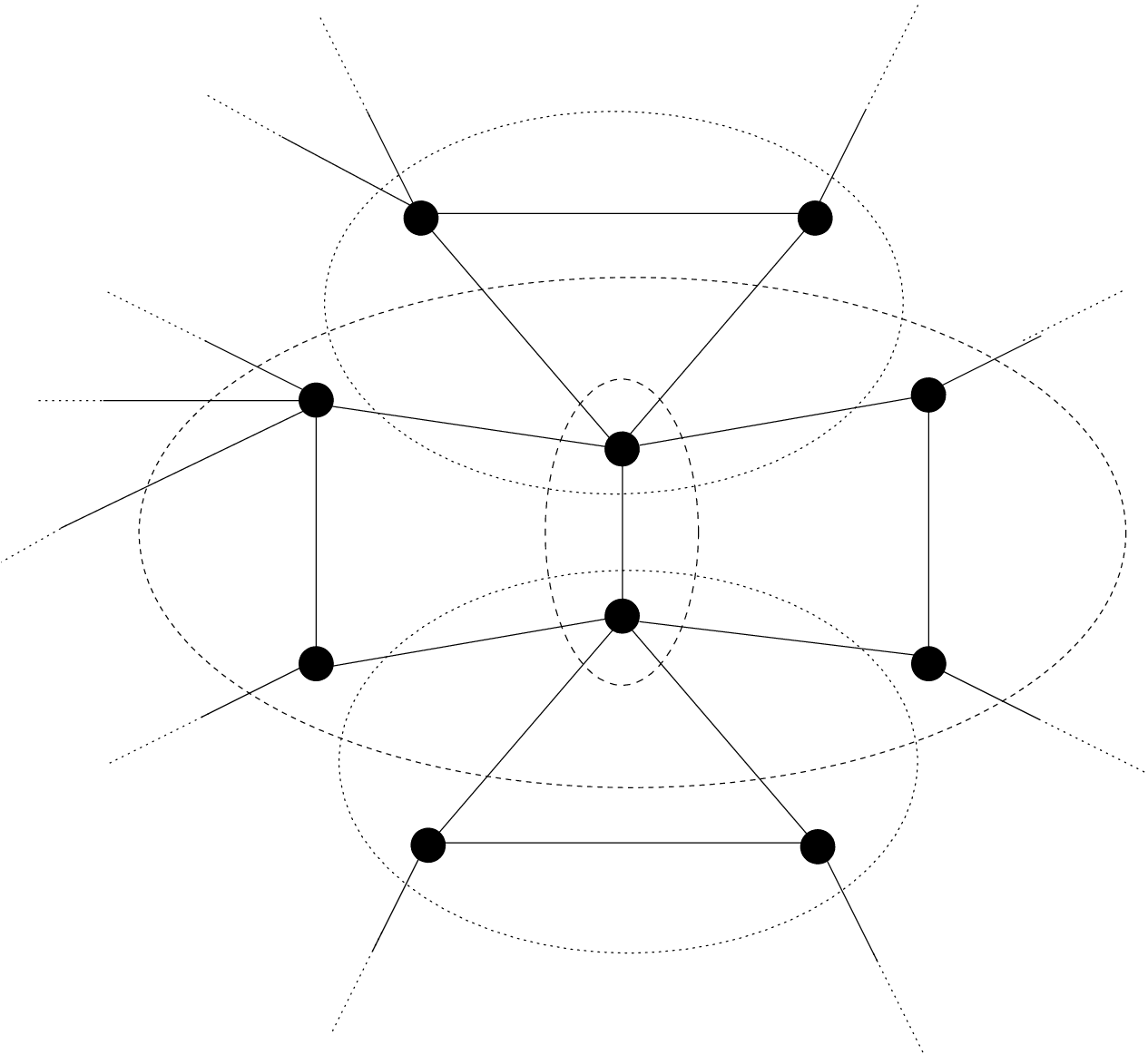}
\qquad \qquad
\rotatebox{90}{\includegraphics[scale=0.27]{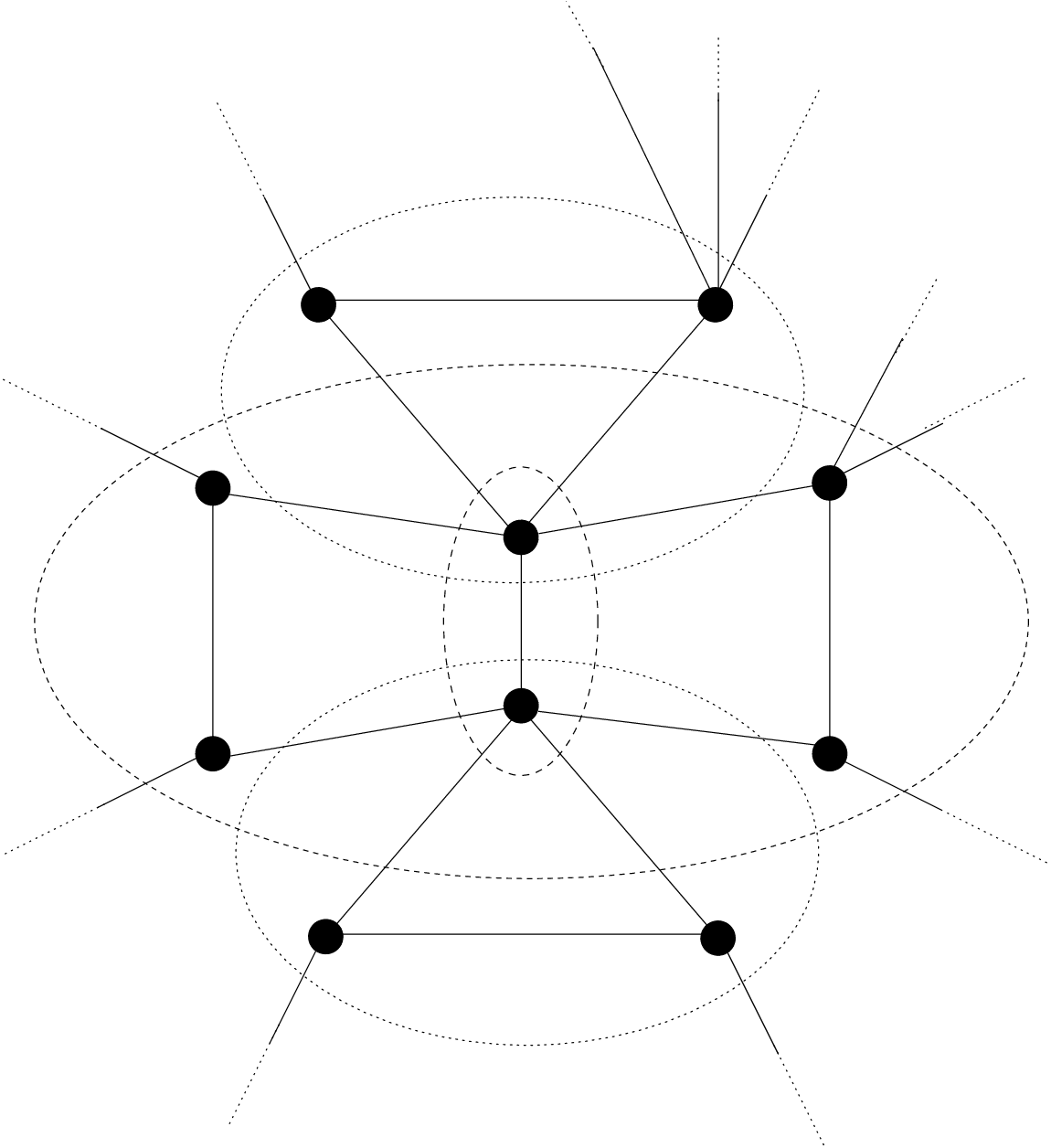}}
\caption{Let $e$ and $f$ be the circled edges above in $\Gamma_i$ (left) and $\Gamma_j$ (right).  
Consider corners $\SS(\Gamma_i, S)$ and $\SS(\Gamma_j, T)$, where $S$ and $T$ each consist of the four circled subgraphs in the left and right pictures respectively.  
This is an example where a principal face gluing induced by $\Gamma_i / e \cong \Gamma_j/f$ reduces the dimension of both identified corners, even though both corners are of codimension 4.  The contraction of the circled edge in each figure yields isomorphic graphs and hence a principal face gluing.  For either face, the gluing results in forgetting a relative rate of approach when passing from a double-square to two triangles joined by an edge.  That is, we forget the relative sizes of the two triangles involved, which was part of the data in the double-square, and the corners become of codimension 5.}
\label{Cicadas}
\end{figure}

\subsubsection{Implementing the gluings}
\label{ImplementingPrincipalGluings}

We now glue the principal faces by partitioning the graphs in an expression for $\delta \gamma$ into unoriented isomorphism classes.  For each class, the terms sum to zero when one accounts for orientations.  Thus within each class, we can group the terms in pairs with opposite orientations.  We arbitrarily choose any such pairings, where each pair determines an orientation-\emph{reversing} isomorphism of graphs.  By Lemma \ref{GluingInteriorsOfPrincipalFaces} we get an orientation-reversing diffeomorphism of their interiors, and by Lemma \ref{ExtendingGlueToClosures}, the induced collapses of some corners yield a homeomorphism of the resulting closures.

\begin{notation}
\label{GluingNotation}
Let $g$ denote the equivalence relation induced by grouping the terms in $\delta \gamma$ into pairs and gluing these pairs via Lemma \ref{GluingInteriorsOfPrincipalFaces} and \ref{ExtendingGlueToClosures}.
\end{notation}

\subsection{Folding principal faces with orientation-reversing automorphisms}
\label{RemainingPrincipalFaces}
The relation \eqref{GraphsWithOrRevAutoZero} means that some graphs in $\delta \gamma$ could vanish by having orientation-reversing automorphisms.

\begin{lemma}
\label{FoldingPrinFacesLemma}
Suppose $\Gamma_i/e$ has an orientation-reversing automorphism.
Then we can take the associated diffeomorphism $\iota$ of $\SS(\Gamma_i, e)$ to be orientation-reversing, and the spherical maps $\Phi_{\Gamma_i}$ and $\Phi_{\Gamma_i} \circ \iota$ to agree up to an orientation-preserving diffeomorphism of the codomain.
\end{lemma}
\begin{proof}
As in Lemma \ref{GluingInteriorsOfPrincipalFaces}, a graph automorphism $\alpha$ of $\Gamma_i/e$ determines a diffeomorphism $\iota$ of $\SS(\Gamma_i, e)$.  
Suppose $\alpha$ is orientation-reversing.  
If $\iota$ reverses orientation, then $\Phi_{\Gamma_i}$ and $\Phi_{\Gamma_i} \circ \iota$ must agree up to an orientation-preserving diffeomorphism, since the sign of $\alpha$ is the product of these two signs.  In this case, we are done.  If $\iota$ preserves orientation, then $\Phi_{\Gamma_i}$ and $\Phi_{\Gamma_i} \circ \iota$ must agree up to an orientation-reversing diffeomorphism.  In this case, we complete the proof by composing $\iota$ with an orientation-reversing map of a factor $S^{d-1}$, as in the proof of Lemma \ref{GluingInteriorsOfPrincipalFaces}.
\end{proof}

The next lemma is proven by similar arguments as those in the proof of Lemma \ref{ExtendingGlueToClosures}:
\begin{lemma}
\label{PrinFoldExtendsToCorners}
Let $\iota$ be a diffeomorphism of $\SS(\Gamma_i, e)$ coming from a graph automorphism $\alpha$ of $\Gamma_i/e$.  Then $\iota$ extends to the closure of $\SS(\Gamma_i, e)$.
\qed
\end{lemma}

Now for either parity of $d$, we fold $\SS(\Gamma_i, e)$, identifying $x \in \SS(\Gamma_i, e)$ with $\iota(x) \in \SS(\Gamma_i, e)$, for every such principal face in $\gamma$.  If a principal face has multiple such automorphisms, we arbitrarily choose one, making no claim that different choices yield identical homeomorphism types.

\begin{notation}
\label{PrincipalFoldingNotation}
Let $r$ denote the equivalence relation given by folding all principal faces with orientation-reversing automorphisms.\\
\end{notation}

At this point, the only principal faces unaccounted for are those corresponding to $\Gamma_i/e$ with a multiple edge.
Such a graph is zero in the graph complex because the associated configuration space integral has a factor $\theta_{ij}^2$ and hence vanishes.  In our construction, we leave the corresponding face $\SS(\Gamma_i, e)$ unglued to any other face, to be relegated to the degenerate locus $\d F_\gamma$ in Definition \ref{DegenerateLocusDef}.

\subsection{Folding hidden faces with bivalent free vertices}
\label{HiddenFaceInvolution}
Having completed principal face gluings, we will modify the result $\widetilde{F}[\Gamma_i]/(c,g,r)$ by folding certain hidden faces.
A hidden face of $\widetilde{F}[\Gamma_i]$ is $\SS(\Gamma_i, \Gamma')$ for some proper, biconnected subgraph $\Gamma' \subset \Gamma_i$.  
Recall that a point in $\SS(\Gamma_i, \Gamma')$ can be described by a configuration in $\R^d$ of the vertices of $\Gamma_i/\Gamma'$, together with a configuration of the vertices of $\Gamma'$ in $T_p\R^d$, modulo translation and oriented scaling, where $p$ is the location of the basepoint in $\Gamma_i/ \Gamma'$.  
(Some of the points in this infinitesimal configuration may be constrained to lie on the tangent line to the link at a point.)
Since $\Gamma'$ is biconnected, it has no 0-valent vertices and no univalent vertices.  

In this Subsection, we consider only hidden faces $\SS(\Gamma_i, \Gamma')$ such that $\Gamma'$ has a bivalent free vertex $v$.  
Let $u,w$ be the vertices in $\Gamma'$ joined by edges to $v$.
We consider an involution of $\SS(\Gamma_i, \Gamma')$ which comes from an involution of the screen space $C_{\Gamma'}(T_p S^d)/ (\R^d \rtimes \R_+)$, given by $x_v \mapsto x_u + x_w - x_v$ and fixing all the other vertices of $\Gamma'$.  
Denote the resulting involution of $\SS(\Gamma_i, \Gamma')$ by $\iota$.
\begin{lemma}
\label{InvolutionsExtendToCorners}
The involution $\iota$ extends from the interior of $\SS(\Gamma_i, \Gamma')$ to its corners.
\end{lemma}
\begin{proof}
As before, such a corner is indexed by a set $S$ of biconnected subgraphs of $\Sigma \Gamma_i$.
We consider all the cases of such subgraphs involved in a corner contained in $\SS(\Gamma_i, \Gamma')$, and their associated screen spaces.  
By Proposition \ref{NecessaryConditionCorner} applied to $C_{\Sigma \Gamma_i}[S^d]$, these subgraphs either are disjoint from $\Gamma'$, contain $\Gamma'$, are contained in $\Gamma'$, or intersect $\Gamma'$ in a single vertex.
The only screens possibly affected are those involving $v$.  The screen spaces involving $v$ but neither $u$ nor $w$ are unaffected.  Those involving both $u$ and $w$ incur an involution just like the screen space for $\Gamma'$.  
This leaves the case of a biconnected subgraph $\Gamma''$ of $\Gamma'$ containing $v$, but only one of $u$ and $w$. By Lemma \ref{HiddenCornersLemma}, $\Gamma''$ is one of the edges $\{u,v\}$ and $\{v,w\}$.  In this case the screen spaces for those edges are interchanged.  
That is, such a corner indexed by $S$ is mapped to the corner indexed by $T$, where $S$ and $T$ differ by exchanging the edges $\{u,v\}$ and $\{v,w\}$.
\end{proof}

For $d$ odd, we ``fold'' $\SS(\Gamma_i, \Gamma')$ via the identification $x \sim \iota(x)$ for every biconnected $\Gamma' \subset \Gamma_i$ with a bivalent free vertex.  
If $\Gamma'$ has multiple bivalent free vertices, we arbitrarily choose one, again with no claim that different choices yield homeomorphic spaces. 
For $d$ even, let $\rho$ be an orientation-reversing diffeomorphism of some $S^{d-1}$ factor of $\widetilde{F}[\Gamma_i]$, such as a reflection in a hyperplane.  We then identify $x \sim \rho \circ \iota(x)$ for every biconnected $\Gamma' \subset \Gamma_i$ with a bivalent free vertex, possibly making arbitrary choices as we did for $d$ odd.

\begin{notation}
\label{HiddenFoldingNotation}
Let $h$ denote the equivalence relation given by folding hidden faces by an involution as above.
\end{notation}

\begin{remark}
Since we do not blow up all the diagonals of $(T_p S^d)^{V(\Gamma')}$ to get $C_{\Gamma'}[T_p S^d]$, the involution is well defined.  In particular, $\Delta_{\{v,u\}}$ and $\Delta_{\{v,w\}}$ are the only two-fold diagonals involving $v$ that are blown up.  If we had blown up all the diagonals, then the involution would not be defined on the whole configuration space.  Instead, it would only be defined on the configuration space of 3 points corresponding to $u,v,w$.  This latter approach can be taken in the setting of integration of forms in \cite{CCRL-AGT}, where integration over the remaining configuration points is performed before applying the involution to this 3-point configuration space.
\end{remark}

\subsection{Collapses for bivalent segment vertices with no incident edge}
\label{Collapses}
Finally, we want to ensure that for every remaining hidden face $\SS(\Gamma_i, \Gamma')$,  
\begin{enumerate}
\item
every free vertex in $\Gamma'$ has valence $\geq 3$, and 
\item
every segment vertex in $\Gamma'$ has an incident edge (i.e., is joined by an edge to another vertex in $\Gamma'$).  
\end{enumerate}
Condition (1) is ensured by the folding in Section \ref{HiddenFaceInvolution}; indeed, vertices cannot have valence 0 or 1, since $\Gamma'$ is biconnected.  Condition (2) is not yet satisfied, but we will collapse those faces $\SS(\Gamma_i, \Gamma')$ which violate it.  
This collapse amounts to forgetting a relative rate of approach.

Suppose $\Gamma'\subset \Gamma_i$ has a segment vertex $v$ with no incident edge.  The neighbors $u,w$ of $v$ on the segment must be in $\Gamma'$ because $\Gamma'$ is biconnected.  Let $L$ be the link over which $F[\Gamma_i]$ is the fiber of $X[\Gamma_i]\to \L^d_m$.
Recall that $\SS(\Gamma_i, \Gamma')$ can be described as a bundle over $F[\Gamma_i / \Gamma']$.  The fiber of this latter bundle, over a configuration where the basepoint of $\Gamma_i/\Gamma'$ is located at $p\in \R^d$,
is the subspace of $C_{\Gamma'}[T_p \R^d]/(\R^d \rtimes \R_+)$ with the segment vertices of $\Gamma'$ constrained to lie in $T_p L$.  This fiber records the relative rate of approach $|v-u|/|v-w|$.  We can forget this relative rate of approach without altering the associated spherical map (cf.~Section \ref{MapToSpheres} below).  In other words, we will forget $v$ from $C_{\Gamma'}[T_p \R^d]/(\R^d \rtimes \R_+)$, thus mapping $\SS(\Gamma_i, \Gamma)$ to a space whose fiber dimension is 1 lower.  Thus in $F[\Gamma_i]$, we have collapsed the codimension-1 face $\SS(\Gamma_i, \Gamma')$ to a codimension-2 face.  We perform such a collapse for every subgraph $\Gamma'$ of some $\Gamma_i$ with a segment vertex with no incident edge. 

\begin{notation}
\label{CollapseNotation}
Let $c$ denote the equivalence relation resulting from the collapses as above.\\
\end{notation}

\subsection{The fundamental class of the glued fiber}
\label{FundClass}
Recall the gluings $g$, the foldings $r$ and $h$, and the collapses $c$ from the previous Subsections.  (See Notations \ref{GluingNotation}, \ref{PrincipalFoldingNotation}, \ref{HiddenFoldingNotation}, and \ref{CollapseNotation}.)  
For any minimal cocycle $\gamma = \sum c_i \Gamma_i \in Z^*(\LD;\, \Z)$, define $\tF_\gamma := \coprod_i  c_i \tF[\Gamma_i]$, and 
define the final glued fiber $F_\gamma$ as
\[
F_\gamma := \coprod_i \left.  c_i \tF[\Gamma_i] \right/ (g, r, h, c)
\] 
by which we mean the transitive closure of these equivalence relations.
Since $\gamma$ is minimal, the glued fiber $F_\gamma$ is connected.  
The diffeomorphism types of the pieces $F[\Gamma_i]$ are independent of the vertex-labelings and edge-orientations in the representative expression for $\gamma$.  However, the homeomorphism type of $F_\gamma$ may depend on the arbitrary choices of gluings.
Letting the link in the base $\L^d_m$ vary, we have a bundle $\tX_\gamma \to \L^d_m$ with fiber $\tF_\gamma$ and a glued bundle $X_\gamma \to \L^d_m$ with fiber $F_\gamma$, which is the fiberwise quotient by a map $\psi: \tX_\gamma \to X_\gamma$. 
We will now define a subspace $\d F_\gamma$ such that $(F_\gamma, \d F_\gamma)$ will be a nicely glued manifold with corners as in Definition \ref{NicelyGluedMfdDef}.

\begin{definition}
\label{DegenerateLocusDef}
Fix a cocycle $\gamma = \sum c_i \Gamma_i \in Z^*(\LD;\, \Z)$.
Let $\d F_\gamma$ be the union of all the codimension-1 faces in the $\widetilde{F}[\Gamma_i]$ which have not been glued or folded.  
Let $\d X_\gamma$ denote the corresponding subspace in $X_\gamma$.  
We call $\d F_\gamma$ (or $\d X_\gamma$) the \emph{degenerate locus} of $F_\gamma$ (respectively $X_\gamma$).
\end{definition}

Every such remaining face is either
\begin{itemize}
\item
a principal face $\SS(\Gamma_i, e)$ of some $\widetilde{F}[\Gamma_i]$ such that $\Gamma_i/e$ has a multiple edge; or
\item
a hidden face $\SS(\Gamma_i, \Gamma')$ of some $\widetilde{F}[\Gamma_i]$ such that in $\Gamma'$, every free vertex has valence $\geq 3$, and every segment vertex has valence $\geq 1$; or
\item
a face at infinity of some $\widetilde{F}[\Gamma_i]$.
\end{itemize}
The converse is not true, since some hidden faces as in the list above were collapsed to higher codimension in Section \ref{Collapses}.  The upshot of the next Proposition is the existence of a fundamental class $[F_\gamma, \d F_\gamma]$, which by Lemma \ref{GluedMfdFundClass} allows us to define the pushforward in singular cohomology; moreover, by Lemma \ref{GluedIntegrationAndSSS}, this pushforward agrees with fiberwise integration.

\begin{proposition}
\label{FundClassProp}
The pair $(F_\gamma, \d F_\gamma)$ is a nicely glued manifold with corners.  
\end{proposition}

\begin{proof}
We just have to check that all the identifications are of the types listed in Proposition \ref{GluedMfdFundClass}.
Write each collapse or folding of a face as $\widetilde{S} \twoheadrightarrow S$ and write $\d S$ for the image of $\d \widetilde{S}$.  Let $k$ be the dimension of $F_\gamma$.  We view the identifications as being made in the following order.  

First apply the collapses $c$, as well any collapses of corners of principal faces that will be induced by the gluings $g$.  
Both types of collapses are given by forgetting relative rates of approach, which are locally standard submersions of Euclidean spaces.  (See Lemma \ref{ExtendingGlueToClosures} and Section \ref{Collapses}.)  Thus $S$ is a a manifold with corners of codimension at least 2.  
This yields the required homological condition in item (3) of Proposition \ref{GluedMfdFundClass}, that $S$ has the homology of an $(k-2)$-dimensional manifold with boundary in degrees $(k-2)$ and higher.

Next perform $h$ and $r$, the foldings of faces by orientation-reversing diffeomorphisms.  
We need to check in each case the hypothesis in item (2) of Proposition \ref{GluedMfdFundClass}, that $S$ has the homology of an $(k-1)$-dimensional manifold with boundary in degrees $(k-1)$ and higher.  Away from a fixed point of a hidden face involution, $h$ is a 2-to-1 local diffeomorphism.  Near a fixed point, $h$ is locally given by identifying points in a Euclidean space with corners with their reflections across a hyperplane.  So $S$ is an $(k-1)$-dimensional manifold with corners.  

The relation $r$ is locally the quotient of a Euclidean space with corners by some subgroup of permutations of blocks of coordinates of size $d-1\geq 2$.  Thus the fixed-point set has codimension at least 2.  Let $T$ be its image in $S$, and let $\eta(T)$ be the image of a tubular neighborhood of the fixed-point set.  The long exact sequence of the pair $(S, \eta(T) \cup \d S)$ and excision of a closed sub-neighborhood then show that $(S, \d S)$ satisfies the needed homological conditions.

Finally, once all the necessary collapses have been performed, the gluings $g$ of principal faces are certainly of type (1) as orientation-reversing diffeomorphisms.  
\end{proof}

\subsection{Mapping to spheres}
\label{MapToSpheres}
Recall that before gluing and folding, each piece $\widetilde{X}[\Gamma_i]:=X[\Gamma_i] \x (S^{d-1})^{a_i}$ of $X_\gamma$ had a map $\Phi_i$ to the $N$-fold cartesian product $(S^{d-1})^N$.  The gluing and folding do not quite respect these maps, but they will if we quotient by certain symmetries of this product.  We will define a subgroup $G$ of symmetries of $(S^{d-1})^N$ which will depend on the parity of $d$.  We will write $G_o$ and $G_e$ for the odd and even cases respectively, but we will simply write $G$ in contexts where the parity is not needed.  

Let $\alpha$ denote the antipodal map on $S^{d-1}$, and for each $j=1,...,N$, let $\alpha_j$ denote the antipodal map on the $j$-th factor of $S^{d-1}$ in $(S^{d-1})^N$.  For even $d$, fix an arbitrary reflection $\rho$ of $S^{d-1}$, and let $\rho_j$ denote this reflection on the $j$-th factor of $S^{d-1}$ in $(S^{d-1})^N$.  We let the symmetric group $\Sigma_N$ act on $(S^{d-1})^N$ by permuting the factors.

\begin{definition} \ 
\label{GDefinitions}
\setlength{\leftmargini}{4.5em}
\begin{itemize}
\item[($d$ odd)] Let $G_o$ be the subgroup of the group generated by $\Sigma_N$ and the $\alpha_j$ which acts by orientation-preserving diffeomorphisms on $(S^{d-1})^N$.
\item[($d$ even)] Let $G_e$ be the subgroup of the group generated by $\Sigma_N$ the $\alpha_j$, and the $\rho_j$ which acts by orientation-preserving diffeomorphisms on $(S^{d-1})^N$.
\end{itemize}
\end{definition}

Thus $G_o$ is generated by all permutations in $\Sigma_N$ and two-fold products $\alpha_j \alpha_k$.
Similarly $G_e$ is generated by all the antipodal maps $\alpha_j$, even permutations, and products of an odd permutation with a reflection $\rho_j$.  (Thus a product $\rho_j\rho_k$ of two reflections is also in $G_e$.)  

Let $\omega$ be a cochain which generates $H^{d-1}(S^{d-1}; \, \Z)$ and which is invariant under the antipodal map $\alpha$ and the reflection $\rho$.  Let $\omega_j$ denote the cochain on the $j$-th factor in $(S^{d-1})^N$.
Let $\omega^{(N)}$ be the class in $H^{N(d-1)}((S^{d-1})^N;\, \Z)$ corresponding to the cohomological cross product $\omega_1 \x \dots \x \omega_N$.

We first list the results to be proven in this subsection before giving their proofs:

\begin{proposition}
\label{MapToSpheresP}
For each cocycle $\gamma = \sum_i c_i \Gamma_i$, the map $\coprod_i c_i\Phi_{\Gamma_i} : \coprod_i c_i \widetilde{X}[\Gamma_i] \to (S^{d-1})^N$ descends to a continuous map of quotients $\Phi_\gamma: X_\gamma \to (S^{d-1})^N / G$ where $X_\gamma := \coprod_i \left. c_i \widetilde{X}[\Gamma_i] \right/ (g,h,r,c)$.  
\end{proposition}

\begin{proposition}
\label{GeneratorOfSpheres}
The quotient $q: (S^{d-1})^N \to (S^{d-1})^N /G$ induces an isomorphism  
\[
q^*: H^{N(d-1)}((S^{d-1})^N /G; \, \R) \overset{\cong}{\longrightarrow} H^{N(d-1)}((S^{d-1})^N; \, \R) \cong \R
\] 
given by including $G$-invariant forms.
It also induces an injection 
\[
q^*: \left.H^{N(d-1)}((S^{d-1})^N /G; \, \Z)\right/ \mathrm{Torsion} \incl H^{N(d-1)}((S^{d-1})^N;\, \Z) \cong \Z
\]
with image $r \Z$, where for $d$ odd, $r = |G_o|=N!\, 2^{N-1}$, and for $d$ even, $r = |G_e| = N!\, 2^{2N-2}$.
\end{proposition}
 
The above injectivity over $\Z$ means that $H^{N(d-1)}((S^{d-1})^N /G; \, \Z) \cong \Z \oplus \mathrm{Torsion}$. 
We do not make any assertions about the torsion subgroup above, e.g.~its nontriviality. 

The following Lemma has the most involved proof of this Subsection.
\begin{lemma}
\label{PositiveCodimension}
Let $\mathcal{R} \subset (S^{d-1})^N/G$ denote the image of $\d X_\gamma$ under $\Phi_\gamma$.  
\begin{enumerate}
\item
If $d \geq 4$, then $\mathcal{R}$ is a union of positive-codimension subspaces of $(S^{d-1})^N/G$.  
\item
If $d\geq 5$, then $\mathcal{R}$ has codimension at least two in $(S^{d-1})^N/G$.  
\end{enumerate}
\end{lemma}
By the long exact sequence of the pair $((S^{d-1})^N/G, \mathcal{R})$, this immediately implies the next result:

\begin{proposition}
\label{ClassInRelCoh}
For $d\geq 4$, the map
\[
H^{N(d-1)} \left((S^{d-1})^N/G, \mathcal{R}; \, \Z\right) \to 
H^{N(d-1)}\left((S^{d-1})^N/G; \, \Z \right)
\]
is surjective.
For $d \geq 5$, this map is an isomorphism.  The same is true with $\R$ coefficients.
\qed
\end{proposition}

The last result of this Subsection merely combines Propositions \ref{GeneratorOfSpheres} and \ref{ClassInRelCoh}.  So we may regard Proposition \ref{MapToSpheresP} and Corollary \ref{GeneratorOfSpheresCor} as the main results of this Subsection.

\begin{corollary}
\label{GeneratorOfSpheresCor}
The composition
\[
H^{N(d-1)} \left((S^{d-1})^N/G, \mathcal{R}; \, \Z\right) \to 
H^{N(d-1)}\left((S^{d-1})^N/G; \, \Z \right)
  \to H^{N(d-1)}\left((S^{d-1})^N; \, \Z \right)
\]
has image $r \Z$ where for $d$ odd, $r = N!\, 2^{N-1}$, and for $d$ even, $r = N!\, 2^{2N-2}$.  For $d\geq 5$, this composition is injective.  Over $\R$, this composition is surjective for $d=4$ and an isomorphism for $d\geq 5$.
\end{corollary}

To indicate how we will use these results, we let $[\omega^{(N)}]$ be the generator of $H^{N(d-1)}\left((S^{d-1})^N; \, \Z \right)$ determined by the orientation on $S^{d-1}$.  Let $[\omega^{(N)}_G]$ be a chosen preimage of $r [\omega^{(N)}]$ in the relative cohomology $H^{N(d-1)}\left((S^{d-1})^N/G, \mathcal{R}; \, \Z \right)$.  For $d \geq 5$, this choice of cohomology class is unique.  We will take the pushforward of  $[\omega^{(N)}_G]$ in singular cohomology with $\Z$ coefficients.  By the universal coefficient theorem, $\omega^{(N)}_G$ also gives rise to a nontrivial class over $\Z/p$. 

\bigskip

We now return to the proofs of the lemmas stated above.  

\begin{proof}[Proof of Proposition \ref{MapToSpheresP}]
We check that the identifications of faces by gluing, folding, and collapsing respect the map $\Phi_{\Gamma_i}$, up to the action of $G$.  

The fact that the gluing of principal faces of $\widetilde{X}[\Gamma_i]$ and $\widetilde{X}[\Gamma_j]$ respects the maps $\Phi_{\Gamma_i}$ and $\Phi_{\Gamma_j}$ in this way was established in part (2) of Lemma \ref{GluingInteriorsOfPrincipalFaces}, where we use our chosen reflection $\rho$ in the case of even $d$ to attain an orientation-reversing diffeomorphism of the principal faces such that the spherical maps agree up to orientation-preserving diffeomorphism.  The same property was established in Lemma \ref{FoldingPrinFacesLemma} for the folding of principal faces with orientation-reversing diffeomorphisms.

For an involution $\iota$ of a hidden face of $\widetilde{X}[\Gamma_i]$, 
$\Phi_{\Gamma_i}(x)$ and 
$\Phi_{\Gamma_i}(\iota(x))$
differ by a transposition two factors of $S^{d-1}$ and an antipodal map on both of these factors.  
For odd $d$, this map is orientation-preserving, hence an element of $G_o$.
For even $d$, it reverses orientation, but in this parity we folded hidden faces by $\rho \circ \iota$ rather than $\iota$, so again $\Phi_{\Gamma_i}(x)$ and 
$\Phi_{\Gamma_i}(\iota(x))$
differ by an orientation-preserving diffeomorphism, i.e.~an element of $G_e$.

Finally, the collapses in Section \ref{Collapses} respect $\Phi_{\Gamma_i}$ on the nose, since the relative rates of approach along the segment are inconsequential to the map $\Phi_{\Gamma_i}$.
\end{proof}

\begin{proof}[Proof of Proposition \ref{GeneratorOfSpheres}]
Consider the quotient $(S^{d-1})^N/G$.  It is known that if a space $Y$ has an action of a finite group $G$,  then for any field $\mathbb{F}$ of characteristic zero, $H^*(Y/G; \, \mathbb{F})$ maps isomorphically onto the $G$-invariant subspace $(H^*(Y; \, \mathbb{F}))^G$ \cite[Chapter III, Corollary 2.3]{Borel-Seminar},  \cite{Grothendieck-Tohoku}.  
Since $G$ preserves orientation, the top cohomology $H^{N(d-1)}((S^{d-1})^N /G; \, \mathbb{F})$ is one-dimensional, and the map to $H^{N(d-1)}((S^{d-1})^N; \, \mathbb{F})$ is an isomorphism.  This proves the first assertion.

By the universal coefficient theorem, $H^{N(d-1)}((S^{d-1})^N /G; \ \Z)$ has rank one.    
The naturality of the universal coefficient short exact sequence gives the claimed injectivity over $\Z$, modulo torsion.

This naturality also allows us to calculate the value of $r$ using de Rham cohomology, where de Rham cochains on the quotient $(S^{d-1})^N/G$ are $G$-invariant forms on $(S^{d-1})^N$.  Although the quotient by $G$ is not a smooth manifold, we can describe the degree of $q$ as for a map of smooth manifolds: the induced map on cohomology is multiplication by the signed count of preimages of a regular value.  
Since $r=|\mathrm{deg}(q)|$ and $\mathrm{deg}(q)=|G|$, it suffices to observe that $|G_o|=N!\, 2^{N-1}$ and $|G_e|=N!\, 2^{2N-2}$.
\end{proof}

\begin{proof}[Proof of Lemma \ref{PositiveCodimension}]
We check the statement for each of the three types of faces listed after Definition \ref{DegenerateLocusDef}.  First, for a principal face $\SS(\Gamma_i, e)$ such that $\Gamma_i/e$ has a multiple edge, the image of $\Phi_i$ in $(S^{d-1})^N/G$ is contained in the image in $(S^{d-1})^N/G$ of some diagonal in $(S^{d-1})^N$.
Since $G$ is finite, the codimension of this image is positive and in fact at least two for $d\geq 3$.

Now we consider the remaining faces, first addressing statement (1) and returning to statement (2) at the end of the argument.  
Since $G$ is finite and each face of $X_\gamma$ is just a face of some $X[\Gamma_i]$, it suffices to show that the images under the original spherical maps $\Phi_i$ have positive codimension in $(S^{d-1})^N$.
Now each face of $X[\Gamma_i]$ corresponds to a biconnected subgraph of $\Sigma\Gamma_i$.  
We consider four cases according to whether the collision involves some segment vertices, and to whether it involves $\infty$: Case I (no segment vertices, no $\infty$), Case II (no segment vertices, $\infty$), Case III (segment vertices, no $\infty$), and Case IV (segment vertices and $\infty$).
The Cases involving $\infty$ (II and IV) correspond to $\Sigma \Gamma'$ for a connected subgraph $\Gamma' \subset \Gamma_i$.  The other Cases (I and III) correspond to a biconnected subgraph $\Gamma' \subset \Gamma_i$.  These arguments use the descriptions of faces in terms of screens given in Section \ref{CornerDescription}, in just the case of codimension 1.

\emph{Case I}:
Let $\SS(\Gamma_i, \Gamma')$ be a hidden face of $X[\Gamma_i]$ where free vertices collide away from $\infty$.  Such a face corresponds to a biconnected $\Gamma'\subset \Gamma_i$ with no segment vertices in $\Gamma'$.  Then $\SS(\Gamma_i, \Gamma')$ is a bundle over $X[\Gamma_i/\Gamma']$.  Its fiber is 
$\overline{C}_{\Gamma'}[\R^d]:=C_{\Gamma'}[\R^d] / (\R^d \rtimes \R_+)$, a space of configurations modulo translation and scaling, which we think of as a space of infinitesimal configurations.
Since there are no segment vertices in $\Gamma'$, the standard framing on $\R^d$ can be used to trivialize this bundle.
Then the map 
\[
\xymatrix
{
\SS(\Gamma_i, \Gamma') \ar[r] \ar@{=}[d] & (S^{d-1})^{E(\Gamma_i)}  \ar@{=}[d]  \\
X[\Gamma_i/\Gamma'] \x \overline{C}_{\Gamma'}[\R^d] \ar[r] &
(S^{d-1})^{E(\Gamma_i/\Gamma')} \x (S^{d-1})^{E(\Gamma')}
}
\]
factors as a product, and it suffices to show that the image of $\overline{C}_{\Gamma'}[\R^d] \to (S^{d-1})^{E(\Gamma')}$ has positive codimension.
If $s$ is the number of (free) vertices in $\Gamma'$, then the dimension of the domain is $ds - d - 1$, and, by the valence conditions in Definition \ref{DegenerateLocusDef},
\[
(d-1)|E(\Gamma')| - (ds-d-1) \geq 
(d-1) \frac{3s}{2} - (ds-d-1)=
\frac{1}{2}(d-3)(s+2) + 4 >0.
\]

\emph{Case II}:
Next consider a face $\SS(\Gamma_i, \Gamma')$ where free vertices collide at $\infty$.  Such a face corresponds to a biconnected $\Sigma \Gamma' \subset \Sigma \Gamma_i$, with no segment vertices in $\Gamma'$.  
Then $\SS(\Gamma_i, \Gamma')$ is also given by a product bundle $X[\Gamma_i \setminus\Gamma'] \x \overline{C}_{\Sigma \Gamma'}[\R^d]$.  Indeed, the base $X[\Gamma_i \setminus\Gamma']$ is the subspace of $X[\Gamma_i / \Gamma']$ where the basepoint of $\Gamma_i/\Gamma'$ is fixed at $\infty \in S^d$; the fiber describes all collisions of points in $\Gamma'$ with $\infty \in S^d$.  Once again, we can factor the map 
\[
\xymatrix
{
\SS(\Gamma_i, \Gamma') \ar[r] \ar@{=}[d] & (S^{d-1})^{E(\Gamma_i)}  \ar@{=}[d] \\
 X[\Gamma_i \setminus \Gamma'] \x \overline{C}_{\Sigma \Gamma'}[\R^d] \ar[r] & 
 (S^{d-1})^{E(\Gamma_i \setminus \Gamma')} \x (S^{d-1})^{E(\Gamma')} 
}
\]
as a product, and it suffices to show that the image of $\overline{C}_{\Sigma \Gamma'}[\R^d] \to (S^{d-1})^{E(\Gamma')}$ has positive codimension.
If $s$ is the number of (free) vertices in $\Gamma'$, then the dimension of the domain of the latter map is $d(s+1) - d - 1=ds-1$, and
\[
(d-1)|E(\Gamma')| - (ds - 1) \geq
(d-1) \frac{3s}{2} - (ds-1) =
\frac{1}{2}(d-3)s +1>0.
\]

\emph{Case III}:
Let $\SS(\Gamma_i, \Gamma')$ be a hidden face not involving infinity, but with $\Gamma'$ having $r>0$ segment vertices, as well as $s$ free vertices.  Then $\SS(\Gamma_i, \Gamma')$ is a bundle over $X[\Gamma_i/\Gamma']$ whose fiber we denote $\overline{F}[\Gamma']$.  
A point in the base $X[\Gamma_i/\Gamma']$ is a link $L$ together with a configuration of the vertices of $\Gamma_i/\Gamma'$ in $\R^d$; let $p \in \R^d$ be the location of the basepoint of $\Gamma_i/\Gamma'$. 
Then $\overline{F}[\Gamma_i]$ is the subspace of $C_{\Gamma'}[T_p \R^d]$ where the segment vertices are constrained to lie on the line $T_pL$, modulo positive scaling and translations which preserve $T_pL$.  
Now there is a related bundle $S^{d-1} \rtimes \overline{F}[\Gamma_i]$ with the same fiber $\overline{F}[\Gamma_i]$ over $S^{d-1}$.  Its total space is given by allowing the oriented tangent line $T_pL$ to vary over all possible directions in $S^{d-1}$.  The space $\SS(\Gamma_i, \Gamma')$ maps to this bundle and in fact, as a bundle over $X[\Gamma_i/\Gamma']$, it is the pullback in the square below, where the lower horizontal map is given by the unit derivative of $L$ at the point $p$:
\[
\xymatrix{
\SS(\Gamma_i, \Gamma') \ar[r] \ar[d] & S^{d-1} \rtimes \overline{F}[\Gamma_i] \ar[d]\\
X[\Gamma_i/\Gamma'] \ar[r] & S^{d-1}
}
\]  
Since the composite $\SS(\Gamma_i, \Gamma') \to (S^{d-1})^{E(\Gamma_i)} \to (S^{d-1})^{E(\Gamma')}$ factors through $S^{d-1} \rtimes \overline{F}[\Gamma'] \to (S^{d-1})^{E(\Gamma')}$, it suffices to show that the image of this latter map has positive codimension.  The dimension of the domain is $(d-1)+ r+ds-2 = d+r+ds-3$, and
\[
(d-1)|E(\Gamma')| - (d+r+ds-3)
\geq (d-1) \frac{r+3s}{2} - (d+r+ds-3)
= \frac{1}{2}(d-3)(r+s-2)
\]
where the inequality is assured by the collapses in Section \ref{Collapses}.
Since $\SS(\Gamma_i, \Gamma')$ is a hidden face (not involving infinity), $r+s>2$.  Thus the right-hand side above is positive, provided $d>3$.

\emph{Case IV}:
Finally, we consider the case of a collision at $\infty$ along the link $L$.  So $\Gamma' \subset \Gamma_i$ has $r>0$ segment vertices and $s$ free vertices, and $\Sigma \Gamma_i$ is biconnected.  Much like in the case of a face at infinity with no segment vertices, such a face $\SS(\Gamma_i, \Gamma')$ is a product bundle $X[\Gamma_i \setminus \Gamma'] \x \overline{F}[\Sigma \Gamma']$.  Here the fiber $\overline{F}[\Sigma \Gamma']$ is a quotient of the subspace of $C_{\Sigma \Gamma'}[T_\infty S^d]$ where $\infty\in \Sigma \Gamma'$ is fixed, and where the segment vertices are constrained to lie on the subset $T_\infty L$; the quotient is by positive scaling.  Thus the dimension of $C_{\Sigma \Gamma'}[T_\infty S^d]$ is $r+ds-1$.
Now the map $\SS(\Gamma_i, \Gamma') \to (S^{d-1})^{E(\Gamma')}$ factors through $\overline{F}[\Sigma \Gamma'] \to (S^{d-1})^{E(\Gamma')}$.  We check that the image of the latter map has positive codimension.  In fact, since the domain has dimension $r+ds -1$, this codimension is at least
\[
(d-1)|E(\Gamma')| - (r+ds-1) \geq
(d-1) \frac{r+3s}{2} - (r+ds-1)=
\frac{1}{2}(d-3)(r+s) + 1 >0.
\]
This completes the proof of statement (1).  

For statement (2), we have already shown that the codimension is greater than 1 in every case except Case III.  In this case, $d\geq 5$ ensures that the codimension is at least 2, provided that the total number $r+s$ of vertices colliding is at least 4.  But the possibility that  $r+s=3$ is ruled out because in $\Gamma'$, every segment vertex has an incident edge (by Section \ref{Collapses}), every free vertex has valence at least 3 (by Section \ref{HiddenFaceInvolution}), and multiple edges are not allowed (by Section \ref{MinimalCocycles}).
\end{proof}

The proof of statement (1) above, Case (III), is the only place in our whole construction where we really need to require that $d>3$.  
For $d=4$, the proof of statement (2) fails.  In fact, it fails for the face corresponding to the collision of all the vertices in the second graph in expression (\ref{Type2Graphs}) for $d=4$, as noted in \cite[Remark A.12]{CCRL-AGT}.

\section{The Main Result}
\label{Results}
Having done most of the required work in the previous Section, we have almost established our main result.  We will  state it and give what remains of its proof in Section \ref{ProofMainThm}.  
In Section \ref{Classical}, we discuss the impediment to extending our construction to the classical case $d=3$, namely the anomalous faces.
We provide some examples and consequences in Section \ref{MainThmExamples}.  

\subsection{Proof of the main result}
\label{ProofMainThm}

We first review some ingredients needed to state and prove our main result.  
Fix an antipodally symmetric unit volume form $\omega$ on $S^{d-1}$.  Let $\omega^{(N)}$ be the corresponding representative of a generator of $H^{N(d-1)}((S^{d-1})^N)$.  It can be obtained via pullbacks by the $N$ projections and the wedge product.  Its image in singular cohomology is the $N$-fold cross product of the singular cochain corresponding to $\omega$.  
Recall that $q: (S^{d-1})^N \to (S^{d-1})^N /G$ is the quotient, where $G$ is a finite group that depends on the parity of $d$ but in either case acts by orientation-preserving diffeomorphisms.
Recall that $\mathcal{R} \subset (S^{d-1})^N/G$ is the image of the degenerate locus $\d X_\gamma$ under $\Phi_\gamma$.

By Corollary \ref{GeneratorOfSpheresCor}, there is a class $[\omega^{(N)}_G]$ in $H^{N(d-1)}((S^{d-1})^N/G, \mathcal{R};\,\Z)$ satisfying 
\begin{equation}
\label{OmegaNGCharacterization1}
q^*([\omega^{(N)}_G]) = r [\omega^{(N)}]
\end{equation}
where $r$ is a certain positive integer (also given in the statement of Theorem \ref{MainTheorem} below).
An alternative characterization of $[\omega^{(N)}_G]$ is that it maps to the generator $+1$ of the $\Z$ summand below, where the isomorphism comes from Proposition \ref{GeneratorOfSpheres} and where $+1$ is ultimately determined by the orientation on $S^{d-1}$:
\begin{equation}
\label{OmegaNGCharacterization2}
[\omega^{(N)}_G] \in
H^{N(d-1)}((S^{d-1})^N/G, \mathcal{R}; \, \Z) \to
H^{N(d-1)}((S^{d-1})^N/G; \, \Z) \cong \Z \oplus \mathrm{Torsion} \ni (+1,0)
\end{equation}
(This property illustrated by \eqref{OmegaNGCharacterization2} seems more useful than \eqref{OmegaNGCharacterization1} for generalizations beyond classes that can be realized over $\R$.)
The class $[\omega^{(N)}_G]$ is uniquely determined for $d\geq 5$ but not for $d=4$.
In either case, we can represent the class $q^*([\omega^{(N)}_G])$ by an antipodally symmetric de Rham cochain that is $G$-invariant and vanishes in a neighborhood of $\mathcal{R}$.  
Via the de Rham isomorphism, this form maps to a cochain in $H^{N(d-1)}((S^{d-1})^N/G, \mathcal{R}; \, \Z)$ (since its volume is $r \in \Z$).    
For each $N$ and each $d$, fix such a cochain representative $\omega^{(N)}_G$.
By Proposition \ref{MapToSpheresP}, we have a map of pairs $\Phi_\gamma:(X_\gamma, \ \d X_\gamma) \to ((S^{d-1})^N/G, \ \mathcal{R})$.

To state the theorem, we need to know the dimension of the fiber $F_\gamma$ of the bundle $X_\gamma \to \L^d_m$.  By construction it is that of $F[\Gamma_i]$ for $\Gamma_i$ with the maximum number of edges $N$ among graphs in $\gamma$.  Using the definition of defect $k$ and order $n$, one finds that $\dim F_\gamma = (3-d)n -k + N(d-1)$.

Proposition \ref{FundClassProp} showed that $(F_\gamma, \d F_\gamma)$ is a nicely glued manifold with corners, so by Lemma \ref{GluedMfdFundClass} it has a fundamental class $[F_\gamma, \d F_\gamma] \in  H_{\dim F_\gamma}(F_\gamma, \d F_\gamma;\, \Z)$, which induces an isomorphism $H^{\dim F_\gamma}(F_\gamma, \d F_\gamma;\, \Z) \to \Z$.
Thus there is a pushforward map  from the Serre spectral sequence, as in Section \ref{GluedMfdIntegration} and Lemma \ref{GluedIntegrationAndSSS}:
\[
\pi^!: H^{N(d-1)}(X_\gamma, \d X_\gamma;\, \Z) \to H^{(d-3)n+k}(\L^d_m;\, \Z)
\]

\begin{theorem}
\label{MainTheorem}
Let $\gamma = \sum_i c_i \Gamma_i \in Z^*(\LD; \,\Z)$ be an integer-valued, minimal graph cocycle. 
Then 
\[
\pi^!( \Phi_\gamma^*(\omega^{(N)}_G))=r \sum_i c_i I_{\Gamma_i}
\]
where the right-hand side is the configuration space integral associated to $\gamma$, and where $r=N!\, 2^{N-1}$ if $d$ is odd and $r=N!\, 2^{2N-2}$ if $d$ is even.  Moreover, for defect $k=0$, the cohomology of the graph complex injects into $H^{*}(\L^d_m;\, \Z)$.
\end{theorem}

Thus we realize integer multiples of the configuration space integral cohomology classes over $\Z$.
The bulk of the work was constructing $X_\gamma$ so that the pushforward is defined.  We related the pushforward to integration in Lemma \ref{GluedIntegrationAndSSS}, so it essentially remains to check that this Lemma  applies here.

\begin{proof}
For a cocycle $\gamma = \sum_i c_i \Gamma_i$, use the same notation as in Section \ref{FundClass}, so that the glued bundle $F_\gamma \to X_\gamma \to \L^d_m$ is a fiberwise quotient of the bundle $\tF_\gamma \to \tX_\gamma \to \L^d_m$ by the relations $g$, $h$, $r$, and $c$.  Let $\psi$ denote the quotient map $\tX_\gamma \twoheadrightarrow X_\gamma$.  The space $\tF_\gamma$ is a manifold with corners, and we write $\d \tF_\gamma$ for its boundary.  Not all of $\d \tF_\gamma$ is mapped into the degenerate locus $\d F_\gamma$, but $\d F_\gamma \subset \psi(\d \tF_\gamma)$.

The sum of configuration space integrals is then precisely 
\[
\sum_i c_i I_{\Gamma_i} = \int_{\tF_\gamma} \widetilde{\Phi}_\gamma^*(\omega^{(N)})
\]
where we write $\widetilde{\Phi}_\gamma$ for the smooth map $\tF_\gamma \to (S^{d-1})^N$ to distinguish it from the induced map on quotients $F_\gamma \to (S^{d-1})^N/G$.  If $q: (S^{d-1})^N\to (S^{d-1})^N/G$ is the quotient map, then $q \circ \widetilde{\Phi}_\gamma = \Phi_\gamma \circ \psi$.

We consider differential forms, de Rham cohomology, and integration of forms on $F_\gamma$ and $X_\gamma$ as described for glued manifolds in Section \ref{GluedMfdIntegration}.  The key was to define all these notions via the quotient map $\psi$.   We can also represent the cohomology of the quotient $(S^{d-1})^N / G$ relative to its subspace $\mathcal{R}$ in these terms.  
By Proposition \ref{GeneratorOfSpheres}, $q^*(\omega^{(N)}_G) = r \omega^{(N)}$ for $r$ as in the Theorem statement.  
By equation \eqref{StratifoldIntegration},
\[
\int_{F_\gamma} \Phi_\gamma^*(\omega^{(N)}_G) :=
\int_{\tF_\gamma} \psi^* \Phi_\gamma^*(\omega^{(N)}_G) =
\int_{\tF_\gamma} \widetilde{\Phi}_\gamma^*  q^*(\omega^{(N)}_G) =
\int_{\tF_\gamma} \widetilde{\Phi}_\gamma^* (r \omega^{(N)}) =
r \sum_i c_i I_{\Gamma_i}
\]
So it suffices to show that  $\pi^!(\Phi_\gamma^*[\omega^{(N)}_G]) = \int_{F_\gamma} \Phi_\gamma^*(\omega^{(N)}_G)$.  

By Proposition \ref{FundClassProp}, the pair $(F_\gamma, \d F_\gamma)$ of the fiber and its degenerate locus are a nicely glued manifold with corners.  
We can thus apply Lemma \ref{GluedIntegrationAndSSS} to $X=X_\gamma$ and $\tX = \tX_\gamma$.  
In our application of Lemma \ref{GluedIntegrationAndSSS}, the target pair of spaces $(Y,Z)$ is $((S^{d-1})^N/G, \mathcal{R})$.  The space $Y$ is not quite a manifold (because $G$ does not act freely), so while we can make sense of differential forms on it, we do not immediately have a de Rham map.  However, we can use $(S^{d-1})^N$ as a proxy in this regard.

More precisely, consider the commutative diagram below, which corresponds to diagram \eqref{IntVsPushfwd} in Lemma \ref{GluedIntegrationAndSSS}.
The surjections $q^*$ in cohomology are inclusions on the chain level, and the symbols $\int$ next to solid arrows denote standard de Rham isomorphisms.   The dashed arrow is discussed below.  As before, $H^*_{dR}(-)$ denotes de Rham cohomology, and $H^*(-;\, \R)$ denotes singular cohomology with $\R$ coefficients.
\begin{equation}
\label{MainThmDiagram}
\xymatrix@C2.5pc@R1.5pc{
H^{N(d-1)} _{dR}(X_\gamma, \d X_\gamma) \ar@{->}[d]_-{\int_{F_\gamma}} & 
\ar[l]_-{\Phi_\gamma^*} H^{N(d-1)}_{dR}((S^{d-1})^N/G, \mathcal{R})  \ar@{->>}[r]^-{q^*}_-{\cong(d\geq 5)} 
\ar@{-->}[ddd]_-{\int}^-{\cong (d \geq 5)}
& H^{N(d-1)}_{dR}((S^{d-1})^N)   \ar[ddd]_-{\int}^-\cong \\
H^{(d-3)n+k}_{dR}(\L^d_m) \ar[d]_-{\int}^-\cong & &\\
H^{(d-3)n+k}(\L^d_m; \,\R) & &\\
H^{N(d-1)}(X_\gamma, \d X_\gamma; \,\R) \ar[u]^-{\mathcal{I}} & 
\ar[l]_-{\Phi_\gamma^*} H^{N(d-1)}((S^{d-1})^N/G, \mathcal{R}; \,\R) \ar@{->>}[r]^-{q^*}_-{\cong(d\geq 5)}   & 
H^{N(d-1)}((S^{d-1})^N; \,\R)
}
\end{equation}
If $d \geq 5$, the maps $q^*$ are isomorphisms, so composing the isomorphisms clockwise around the right-hand  side of the diagram provides the vertical dashed arrow.
For $d=4$, the argument is similar.  We just recall that we represented $[\omega^{(N)}_G]$ by a cochain on $(S^{d-1})^N$ in both de Rham and singular cohomology, as discussed before the Theorem statement, and these were related by the de Rham map.  Thus the vertical dashed arrow is defined on the submodule generated by this cochain.
In conclusion, we establish $\pi^!(\Phi_\gamma^*[\omega^{(N)}_G]) = \int_{F_\gamma} \Phi_\gamma^*(\omega^{(N)}_G)$ for $d\geq 4$, completing the proof that the pushforward recovers integer multiples of configuration space integral classes.

Finally, the statement in the theorem about nontriviality in defect 0 follows from the fact that $\sum_i I_{\Gamma_i} = \mathcal{I}(\Phi_\gamma^*\omega^{(N)}))$ and the injectivity in cohomology of the configuration space integral classes in defect 0 \cite{CCRL-AGT}. 
\end{proof}

\begin{remark}
We consider the dependence of the resulting cohomology class 
$\pi^!(\Phi_\gamma^*(\omega^{(N)}))$
on the choices made in constructing the glued bundle $F_\gamma \to X_\gamma \to \L^d_m$:  

(1)  The choice of labelings of the $\Gamma_i$ representing $\gamma$ do not affect the diffeomorphism types of the pieces, but the homeomorphism type of $F_\gamma$ may depend on the choice of principal face gluings.  
This latter dependence does not however affect the resulting cohomology class in our construction, at least for $d$ odd: in any case it agrees with the integral of differential forms, which does not depend on the labelings of the graphs \cite[Theorem 4.2]{CCRL-AGT}.

(2)  For $d=4$, as in Proposition \ref{ClassInRelCoh}, there is also a choice of preimage of $[\omega^{(N)}]$ in the cohomology of $S^{N(d-1)}$ relative to the degenerate locus, and different choices may produce different cohomology classes via the fiber integration map.  
For $d \geq 5$, the choice in cohomology is uniquely determined by the orientation on $S^{d-1}$.
\end{remark}

\begin{remark}
\label{CoarseQuotient}
One could ask if the classes constructed in Theorem \ref{MainTheorem} can be realized using the cohomology of the coarser quotient $\tX_\gamma/\d \tX_\gamma$ rather than that of $X_\gamma/ \d X_\gamma$.  This was essentially the approach taken in our earlier work \cite{Rbo}.  In fact, they do lie in the image of the pushforward $H^{N(d-1)}(\tX_\gamma, \d \tX_\gamma;\,\Z) \to H^{(d-3)n+k}(\L^d_m;\, \Z)$.  One can prove this by considering an analogue of \eqref{IntegrationBySSS} for $(\tX_\gamma, \d \tX_\gamma)$, where $p+q = N(d-1)$ and $q=(d-3)n+k$:
\[
H^{p+q}(\tX_\gamma, \d \tX_\gamma;\, \Z)   \twoheadrightarrow 
{E}_\infty^{p, q} \overset{\cong}{\longrightarrow} 
{E}_2^{p, q} \cong  H^{p}(\L^d_m; \, H^q(\tF_\gamma, \d \tF_\gamma;\,\Z)) \twoheadrightarrow 
H^{p}(\L^d_m;\, \Z) 
\]
The inclusion ${E}_\infty^{p, q} \to {E}_2^{p, q}$ is an isomorphism because $(\tX_\gamma, \d \tX_\gamma)$ has the cohomology of a trivial bundle (as shown in an unpublished part of the author's 2010 PhD thesis, by including it into a trivial bundle and using a Serre spectral sequence argument).  The last map is surjective because $\tF_\gamma$ is a manifold with corners, so its top cohomology $H^q$ is free of rank $\pi_0(\tF_\gamma)$.

However,  $\tX_\gamma/ \d \tX_\gamma$ admits no analogue of $\Phi_\gamma$, so
this surjectivity just means that $\pi^!\Phi_\gamma^*(\omega^{(N)}_G)$ is {cohomologous} to the pushforward of some cochain representing a class in $H^*(\tX_\gamma, \d \tX_\gamma)$.  This caveat limits the utility of using the coarser quotient.
To illustrate this point, consider the pairwise linking number as the fiberwise integral over the trivial bundle $X=\L_2^3 \x (S^1 \x S^1)\overset{\pi}{\longrightarrow} S^1 \x S^1$ of the form $\Phi^*\omega$, where $\Phi: X \to S^2$ and $\omega$ represents a generator of $H^2(S^2)$.  Of course, for any open set $U$ in $S^1 \x S^1$, there is a cochain $\eta$ cohomologous to $\Phi^*\omega$ and supported in $\L_2^3 \x U$.  (In this analogy, $U$ corresponds to an open set disjoint from all the boundary faces of the $\tX_\gamma$.)  The pushforward $\pi^!(\eta)$ (or $\int_{S^1 \x S^1}\eta$) is the linking number.  However, a representative of this sort is unlikely to be informative for calculations.
\end{remark}

\subsection{Remarks on the case of classical knots and links}
\label{Classical}
For $d=3$, invariants of links are degree-0 cohomology classes in the space of links, and these classes come from the defect-0 part  $\LD^{0,n}$ of the graph complex, i.e.~from cocycles of (uni)trivalent graphs.
Our construction does not apply directly to $d=3$ because in the proof of Lemma \ref{PositiveCodimension}, the image under the spherical map $\Phi_\gamma$ of a face $\SS(\Gamma, \Gamma')$ of some $X[\Gamma]$ in Case III may be a codimension-zero subspace.  

We now characterize the $\Gamma'$ for which $\SS(\Gamma, \Gamma')$ is such a problematic face, called an \emph{anomalous face}. 
First, $\Gamma'$ is biconnected, and $|V(\Gamma')|\geq 3$.
Each vertex in $\Gamma'$ is trivalent, as guaranteed by Section \ref{Collapses}, so the graph $U(\Gamma')$ is a union of components of the graph $U(\Gamma)$ obtained by forgetting the segments.
The segment vertices of $\Gamma'$ must lie on only one segment, since the vertices in $\Gamma'$ have collided in $\SS(\Gamma, \Gamma')$.  So anomalous faces arise only for invariants that detect knotting, rather than linking.  For example, they do not arise for Milnor's homotopy invariants, finite-type invariants of pure braids, or Milnor's concordance invariants.
By adding certain collapses to our construction, we can assume that $U(\Gamma')$ is a single component of $U(\Gamma)$.  These collapses deal with disconnected $U(\Gamma')$ by forgetting certain relative rates of approach as in Section \ref{Collapses}; see 
 \cite[Section 5.2, Case 1]{DThurstonABThesis} and \cite[Section 4.2]{VolicBT} for the treatment in de Rham theory.
Finally, $\Gamma'$ cannot be crossed by another component $\Gamma''$ of $\Gamma$, for then vertices in $\Gamma''$ would be in the collision.  (That is, we cannot have segment vertices $i,j,k$ with $i,k \in V(\Gamma')$, $j \in V(\Gamma'')$ and $i < j < k$.)  One can restrict to primitive cocycles as in  \cite[Section 6]{DThurstonABThesis} and \cite{VolicBT}, and then the anomalous faces are precisely those involving the collision of all the points in a diagram $\Gamma$ on a knot.  

Poirier handles these faces and obtains a result similar to our Main Theorem, but for $d=3$ \cite[Proposition 1.11]{PoirierAGT}.  
Our construction does not readily handle the anomalous faces in general, though it can handle them in some special cases.
For example, for the type-2 cocycle in Example \ref{Type2Example}, the face of the tripod diagram $Y$ where all 4 points collide is a face $\SS \subset \d X_\gamma$ for which the proof of Lemma \ref{PositiveCodimension} does not guarantee that $\Phi_\gamma(\SS)$ has positive codimension.  However, on this face $\SS$, the three segment vertices are collinear, so the three vectors in  $\Phi_\gamma(\SS)$ are coplanar. Hence $\Phi_\gamma(\SS)$ indeed has positive codimension.

Our results and those of Poirier suggest that configuration space integrals may be related to counting formulae for finite-type invariants like the arrow diagram formulae of Goussarov, Polyak, and Viro \cite{GPV, PolyakViro-IMRN}.
Polyak and Viro establish such a formula by configuration spaces for the type-2 invariant \cite{PolyakViro-Casson}.
The rough idea to do find such formulae for an arbitrary finite-type invariant appears in the tinkertoy diagrams of D.~Thurston \cite{DThurstonABThesis}, though a complete exposition in the general case has not appeared, to our knowledge.
Even the values of finite-type invariants over $\Z/2$ (where one can ignore signs and directions of arrows) can be useful.  For example, the mod-2 values of the first four finite-type invariants separate the four simplest nontrivial knots, as shown in the Table in \cite{Stanford-mod2}.

\subsection{Basic examples}
\label{MainThmExamples}

\begin{example}[The order-2 cocycle]
\label{Type2Knots}
Consider the cocycle of order 2 for $d$ odd from Example \ref{Type2Example}.  This (signed) sum of two graphs corresponds to two (oriented) configuration spaces $X$ and $Y$, where $Y$ has the orientation opposite from the canonical one coming from the vertex labels.  There are three pairs of principal faces glued to each other, and each gluing identifies a face of $X$ with a face of $Y$.  
We thus obtain a $\Z$-valued cohomology class in the space of long knots in $\R^d$, $d$ odd.
Our construction in this example is similar to that of Polyak and Viro \cite{PolyakViro-Casson}.  The main difference is that they make a glued space involving 6 copies of each of $X$ and $Y$, whereas we take only one copy of each we quotient by some symmetries of the product of spheres.  
Although their construction is for long knots in $\R^3$, rather than $\R^d$, $d>3$, our construction for this example is otherwise the same (see the discussion of the anomalous face for the tripod in Section \ref{Classical}).  
We may also consider the cocycle (\ref{Type2GraphsEven}) for $d$ even.
\end{example}

\begin{example}[The triple linking number for long links]
For long links of 3 components in $\R^d$, $d$ odd, there is a cocycle of defect 0 and order 2:
\begin{align*}
+\quad \raisebox{-2pc}{\includegraphics[scale=0.3]{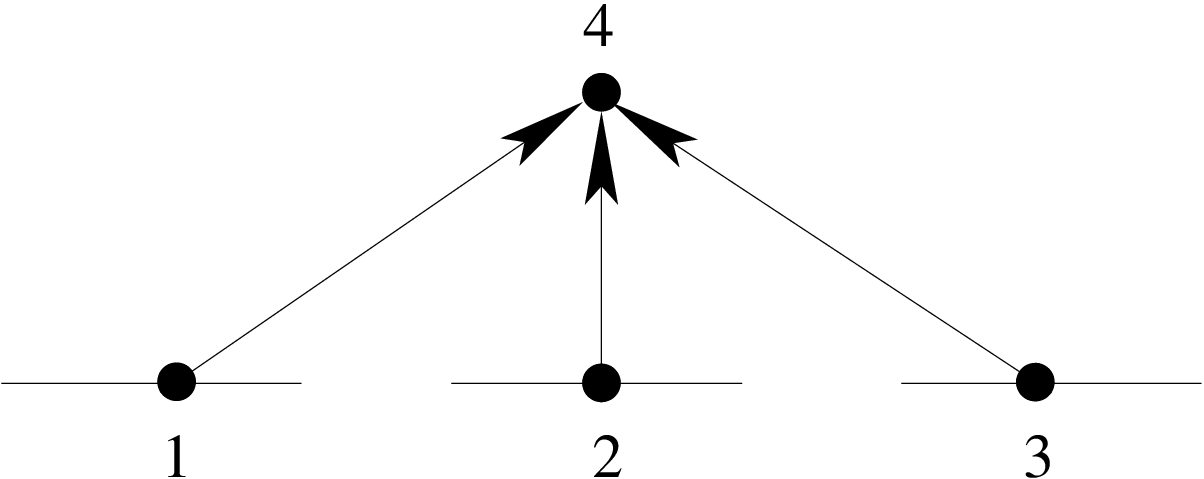}} \qquad - \qquad
\raisebox{-2pc}{\includegraphics[scale=0.3]{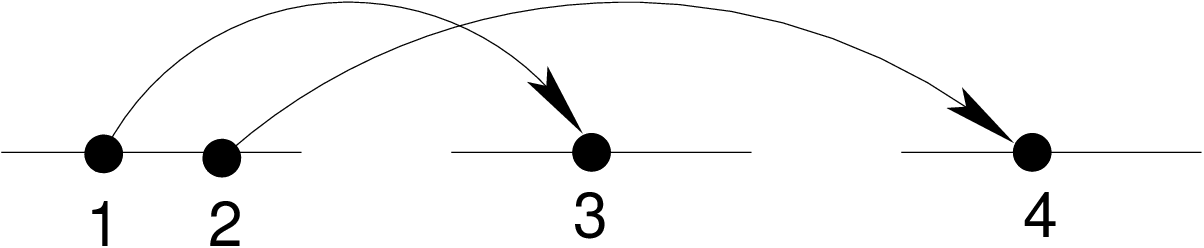}}  \\
-\quad \raisebox{-2pc}{\includegraphics[scale=0.3]{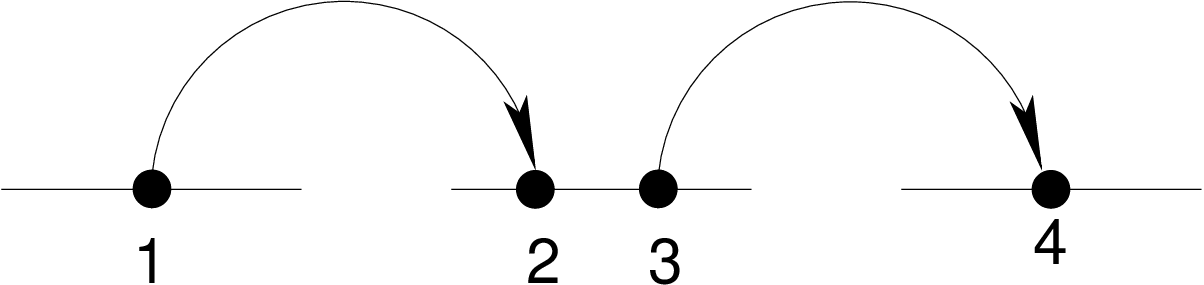}} \qquad - \qquad
\raisebox{-2pc}{\includegraphics[scale=0.3]{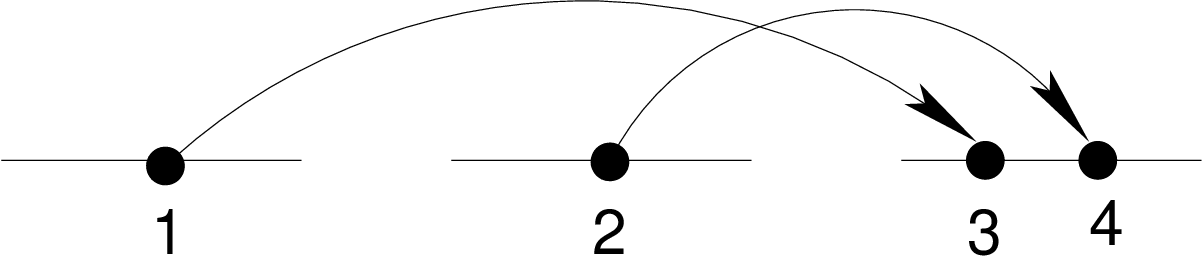}} 
\end{align*}
This cocycle gives rise via integration (and hence via our construction) to the Milnor triple linking number for long links.  Our present construction in the case of this cocycle yields the same glued space as in our previous work \cite{HoMilnor3ple}, after correcting that previous work for some sign errors.\footnote{These errors were due to an incorrect identification between the Lie orientation and the ``integration orientation,'' i.e., the vertex-labelings and edge-orientations.  See \cite{MilnorTrivalentTrees} for more details on this identification.}
Namely, we glue together four spaces $T, L, M, R$ corresponding to the four terms above.  The space $T$ has three principal faces, with each one glued to the principal face of $L,M$, or $R$.  There are no hidden faces with involution, so the remaining codimension-1 faces are just part of the degenerate locus.
\end{example}

We could write down further examples, coming from higher-order cocycles in defect 0.  For example, there is a cocycle of order 3 and defect 0 \cite[Section 5.1, Figure 3]{CCRL-AGT}.  Clearing denominators gives an integral cocycle, with six terms, so there are copies of six types of spaces glued together.  

Recall that in general, defect-0 cocycles are precisely cocycles of (uni)trivalent diagrams.  These correspond to finite-type invariants of knots and links, by the ``Fundamental Theorem of Finite-Type Invariants.''  
In Bar-Natan's paper \cite{BarNatanTopology}, defect-0 graph cocycles appear in the guise of functionals on the quotient of trivalent diagrams by the STU relation.
We may combine our construction and the nontriviality in defect 0 from Theorem \ref{MainTheorem} to yield the following result.  The analogue for closed links also holds, since our methods are easily adapted to that setting.

\begin{corollary}
For each $\Z$-valued finite-type invariant of $m$-component long links, $m\geq1$, we have a nontrivial integer-valued cohomology class in 
$\L^d_m$ 
for any $d>3$ odd.
\qed
\end{corollary}

\begin{example}
Sakai found a cocycle of defect 1 in the graph complex for long knots in {odd}-dimensional Euclidean space, and he used configuration space integrals to produce real-valued cohomology classes in $\Emb(\R, \R^d)$ for $d$ odd.  This cocycle has nine terms (with coefficients $\pm 1, \pm 2$), so we refer the reader to \cite[Figure 10]{Sakai-AGT} for the precise expression.  Our construction produces $\Z$-valued cohomology classes in embedding spaces out of this cocycle, just as for defect-0 cocycles.  
\end{example}

\begin{example}
\label{LongoniExample}
Longoni found the following defect-1, order-3 cocycle in the space of closed knots in even-dimensional Euclidean space \cite{Longoni}:
\[
\Gamma \quad  = 
\quad 2\Gamma_1 \quad + \quad \Gamma_2 \qquad := 
\qquad 2 \qquad \raisebox{-3.1pc}{\includegraphics[scale=0.2]{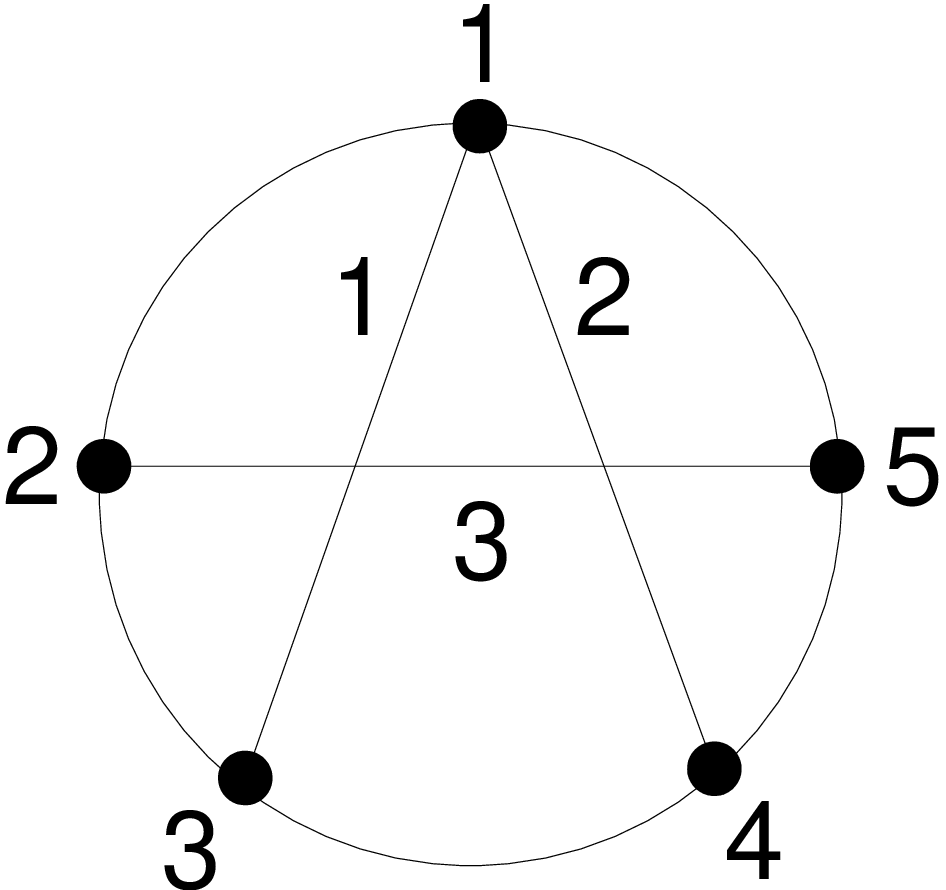}} \qquad + 
\qquad \raisebox{-4.1pc}{\includegraphics[scale=0.2]{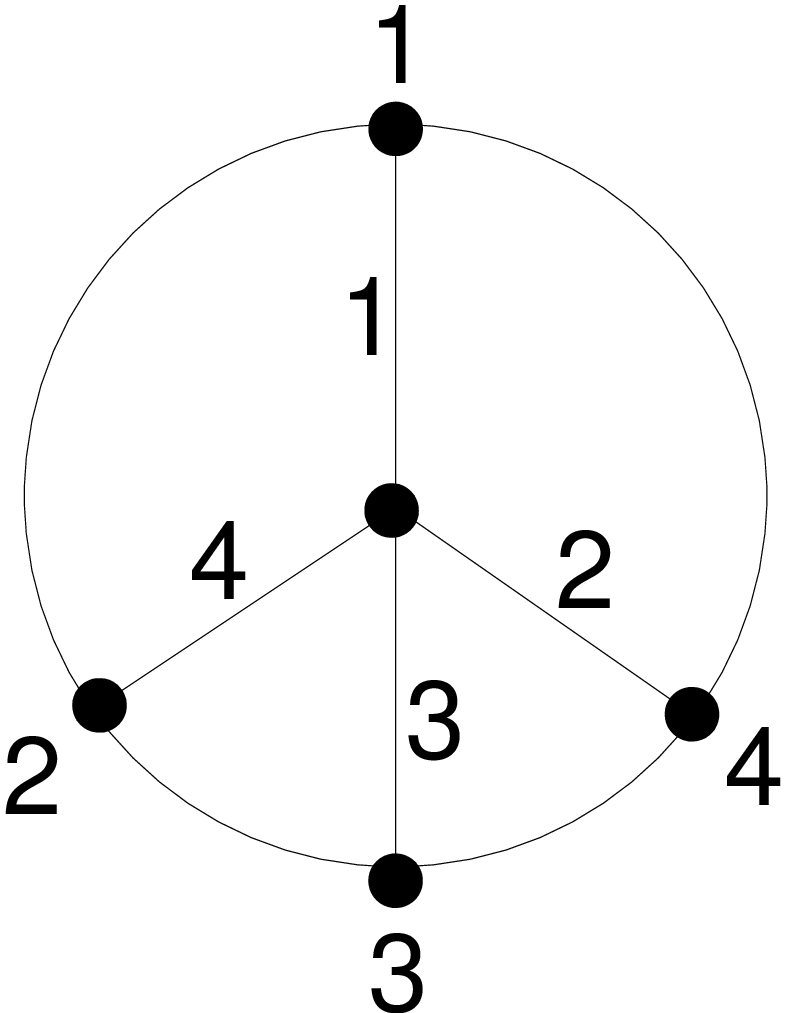}} 
\]
By readily adapting our methods from long links to closed links, we get from this cocycle a $\Z$-valued cohomology class in $\Emb(S^1, \R^d)$, $d$ even.
\end{example}

\section{Towards nontrivial torsion classes}
\label{NontrivialTorsion}
We now sketch the possibility of finding nontrivial torsion classes using our construction.  
For this entire Section, we omit the relation \eqref{GraphsWithOrRevAutoZero} on $\LD$.  That is, we do \emph{not} set graphs with orientation-reversing automorphisms to zero, for such graphs are precisely one possible source of nontrivial torsion classes.  Our main construction is thus to be carried out with the modification that we omit the foldings $r$ described in Section \ref{RemainingPrincipalFaces}.  

The key ingredients would be work of Pelatt and Sinha on constructing homology classes and work of Turchin showing torsion cohomology classes in the graph complex.  
These torsion classes are known to exist only at the spectral sequence level, so it would be interesting to show that they survive as nontrivial classes in the cohomology of the space of knots.  
We summarize the salient points of both authors' work in Sections \ref{TurchinSection} and \ref{PelattSinhaSection}, and we consider a prototype for detecting nontrivial torsion in Section \ref{ToyModel}.

\subsection{Turchin's calculations in the cohomology of graph complexes}
\label{TurchinSection}
In \cite[Appendix B]{VictorBialg}, Turchin determined the cohomology of the graph cochain complex for long knots, over $\Z$, up to order $n=5$ and for all possible defects $k$ for these orders.  
We write $\KD^{k,n}$ for the cochain groups $\LD^{k,n}$ in this special case of $m=1$.
Among these groups are the following ones containing torsion:
\begin{align*}
\text{For $d$ even: } & & H^*(\KD^{2,4}) &\cong \Z/10  & & \text{For $d$ odd: } & H^*(\KD^{2,4}) \cong \Z/2 \\
& & H^*(\KD^{1,4}) &\cong \Z/2 \oplus \Z/2   & &  & H^*(\KD^{3,5}) \cong \Z/2 \\
& & H^*(\KD^{2,5}) &\cong \Z/2  & & & &
\end{align*}
Turchin represents cocycles in $\KD^{k,n}$ by chord diagrams (with $n$ chords and $2n-k$ segment vertices).  One can prove that this complex of chord diagrams agrees with our complex in cohomology;  the map between them is given by remembering only the chord diagram terms.  
Thus our construction produces torsion classes in $H^*(\Emb(\R, \R^d);\, \Z)$ out of these graph cohomology classes.

\begin{question}
\label{Nontriviality}
Do the torsion elements above survive as nontrivial classes in $H^*(\Emb(\R,\R^d))$ under our construction?
\end{question}

A positive answer to this question is suggested by the injectivity of our construction in defect 0 in our Main Theorem and the collapse at $E^2$ of the Vassiliev spectral sequence over $\Q$ for knots in $\R^d, d\geq 4$, proven by Lambrechts, Turchin, and Voli\'{c} \cite{LTV}.  That is, such an answer would generalize their result for all non-torsion classes to include also several torsion classes.  Furthermore, one may ask whether the Vassiliev spectral sequence collapses at $E^2$ over $\Z$ or $\Z/p$, and we hope the methods sketched below may provide an answer. 

At present, a proof of this conjecture is hindered by the lack of explicit formulae for a full graph cocycle $\gamma$ representing a class as above, or even just the chord diagram terms of $\gamma$.  
The next two subsections suggest how we might detect nontriviality, if we had either explicit formulae for graph cocycles or a broad framework for constructing dual homology classes.

\subsection{Pelatt and Sinha's homology class in the space of knots}
\label{PelattSinhaSection}
Pelatt and Sinha constructed nontrivial homology classes over $\Z$ in spaces of long knots in $\R^d$ for $d$ even \cite{Pelatt}.  
Their methods apply in either parity.  
In fact, their work only uses $d$ even in that it studies Longoni's graph cocycle $\Gamma = \Gamma_1 + \Gamma_2$ mentioned in Example \ref{LongoniExample}.  Pelatt and Sinha constructed a dual homology class in the space of long knots from the bracket expression 
\[
\beta_1 + \beta_2:=[[x_1,x_4],x_3][x_2,x_5] + [x_1,x_4][[x_2,x_5],x_3],
\]
which by work of Sinha \cite{SinhaTop} corresponds to a cycle in the homology Vassiliev spectral sequence.  They define two families $\mathcal{M}_1$ and $\mathcal{M}_2$ of resolutions of singular knots, with singularities prescribed by each $\beta_i$.  A factor in $\beta_i$ with $k$ terms $x_j$ in nested brackets corresponds to a $k$-fold self-intersection.  
Thus each $\mathcal{M}_i$ comes from a singular knot with one triple-point and one double-point. 
Systematically gluing together the boundaries of the $\mathcal{M}_i$, they get a homology class $[\mathcal{M}] \in H^*(\Emb(\R, \R^d))$.  The configuration space integral $I_\Gamma$ of $\Gamma$ is then paired with $[\mathcal{M}]$, and a nonzero result establishes the nontriviality of $[\mathcal{M}]$, as well as that of $I_\Gamma$.  The fact that $\langle I_\Gamma , [\mathcal{M}] \rangle \neq 0$ is proven by checking that the non-chord-diagram term $I_{\Gamma_2}$ vanishes on all of $\mathcal{M}$, while $I_{\Gamma_1}$ incurs a nonzero contribution only from the piece $\mathcal{M}_1$ corresponding to $\beta_1$, which prescribes exactly the singularity data detected by $\Gamma_1$.  The calculation of the pairing, though written via integrals in \cite{Pelatt}, relies only on a signed count of preimages and is thus amenable to detecting torsion classes.

\subsection{A toy model example for nontrivial 2-torsion.}
\label{ToyModel}
We now present an example of a torsion cocycle, which is of lower order than the Longoni cocycle.
We cannot as of yet realize it in the space of knots, as we explain below, but it serves to further illustrate how we may detect nontrivial torsion.
Consider the following graph $\Gamma^1_2$ of order 2 and defect 1.  As suggested by Turchin's calculations, it is a cocycle in a related but different graph complex, namely a complex for the space $\overline{\Emb}(\R, \R^d)$ of long embeddings modulo immersions, i.e.~$\overline{\Emb}(\R, \R^d):= \mathrm{hofib}(\Emb(\R, \R^d) \to \mathrm{Imm}(\R, \R^d))$.  
\[
\Gamma_2^1\quad := \quad \raisebox{-1pc}{\includegraphics[scale=0.3]{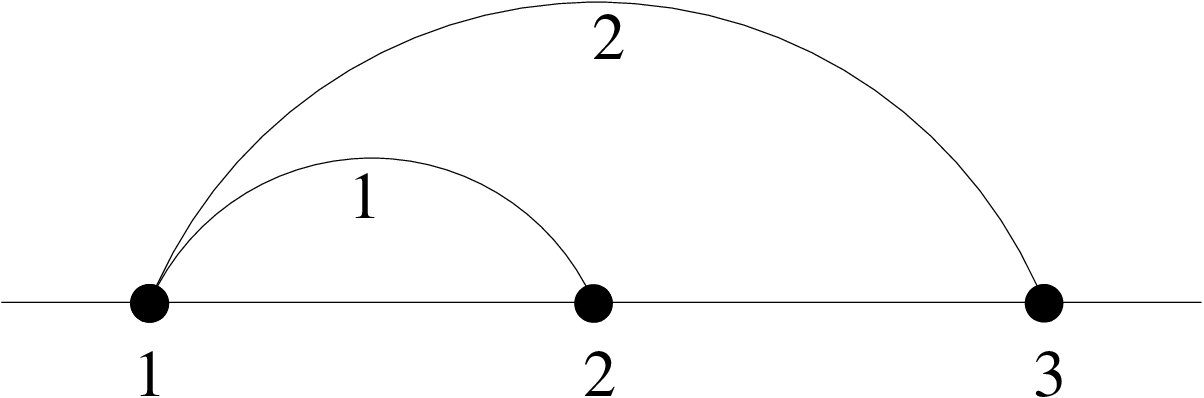}}
\]
The differential applied to the sum of the three order-2, defect-0 chord diagrams (with appropriate signs) yields $2 \Gamma_2^1$, so $\Gamma_2^1$ generates a subgroup isomorphic to $\Z/2$ in the cohomology of the graph complex for $\Embbar(\R, \R^d)$.
The graph $\Gamma_2^1$ is however not closed in the graph complex $\KD^{*,*}$ for $\Emb(\R, \R^d)$.  

Our construction does not immediately apply to $\overline{\Emb}(\R, \R^d)$.  Indeed, if one views an element in this homotopy fiber as an embedding together with a path through immersions to the standard unknot, the evaluation map in the pullback square \eqref{GammaBTsquare} is not well defined.  
Nonetheless, if we imagine that we could produce a cohomology class $\pi^! \Phi_{\Gamma^1_2}^*(\omega^{(2)}_G)$ out of $\Gamma_2^1$, 
we would pair it with a homology class corresponding to the bracket expression $\beta:=[[x_3,x_1],x_2]$. 
In fact, one can check that $\beta$ is a cycle, under the differential $d_1$ given in \cite[Section 2]{Pelatt}.

We would construct a homology class $[\mathcal{M}_\beta]$ out of a family of embeddings $\mathcal{M}_\beta$, by starting with an immersion $f$ with one triple-point, where the three strands involved have linearly independent tangent vectors.  Resolving the triple-point by pushing off one strand in all possible directions gives an $S^{d-2}$ family of knots with one-double point, except at four points in $S^{d-2}$, where the knot has two double-points.  
The four points can be partitioned into pairs according to the singularity data: one pair has the chord diagram X from \eqref{Type2GraphsEven}, while the other has a ``two-striped rainbow,'' i.e.~chords joining points 1 and 4 and points 2 and 3.  (In the setting of long knots, there are three chord diagrams with two chords, but we can consider only these two,by making an appropriate choice of which of the three strands to move when resolving the triple-point.  Namely, we choose the ``middle'' strand in terms of the order in the domain $\R$ of $f$.)   
The resolutions of the extra double point can thus be glued pairwise.
The remaining double-point in all of the diagrams can be resolved via an $S^{d-3}$ family of embeddings.  The overall result is a family of knots $\mathcal{M}_\beta$ parametrized by a manifold with no boundary, which thus represents a homology class $[\mathcal{M}_\beta]$.  Its dimension is $(d-2) + (d-3)=2d-5$.

We would finally check that $\langle \pi^! \Phi_{\Gamma^1_2}^*(\omega^{(2)}_G), [\mathcal{M}_\beta] \rangle = \pm 1$.  
Indeed, $\pi^! \Phi_{\Gamma^1_2}^*(\omega^{(2)}_G)$ counts preimages under the Gauss map $\Phi_{\Gamma_2^1}$, as in \eqref{BigPhi}, and $\mathcal{M}_\beta$ is constructed to have exactly one such preimage.

If this construction 
were possible, it would yield a class $[\mathcal{M}_\beta] \in H^{2d-5}(\overline{\Emb}(\R, \R^d))$ for even $d\geq 4$.  Since $\overline{\Emb}(\R, \R^d) \simeq \Emb(\R, \R^d) \x \Omega^2 S^{d-1}$ \cite[Proposition 5.17]{SinhaOperads} and since $\Gamma_2^1$ is not a cocycle in the complex for $\Emb(\R, \R^d)$ itself, one might ask if the linear dual to $[\mathcal{M}_\beta]$ comes from $\pi_{2d-3}(S^{d-1})$ via the Hurewicz map.  That is, after relabeling, is it related to $\pi_{2j-1}(S^j)$ for $j\geq 3$ odd?  Notably, $\pi_{2j-1}(S^j)$ has 2-torsion (in fact a $\Z/2$ summand) for every odd $j\geq 3$ at least up to $j=19$ \cite{TodaCompositionMethods}.
\bigskip

As a final remark on this direction, possible approaches to Question \ref{Nontriviality} include finding explicit formulae for torsion classes in low degrees, as well as the general setting of the Vassiliev spectral sequence over $\Z$ or $\Z/p$.
Our construction provides cohomology classes in spaces of links for all graph cohomology classes.
A potentially fruitful approach for finding homology classes to pair these with (either in low degrees or in general) may be to exploit the pairing of graphs and trees in the context of the homology and cohomology of configuration spaces \cite{Sinha-Graphs-Trees}.

\section{Variants of the main construction over $\Z/p$}
\label{ModPSection}
Our construction can be modified to produce classes over $\Z/p$ which are not necessarily mod-$p$ reductions of the classes over $\Z$ from our Main Theorem.
If we let $p$ be prime and consider the graph cochain complex over $\Z/p$, we get $\Z/p$-vector spaces generated by graphs, with a differential given by edge contractions with the same signs as over $\Z$, but reduced modulo $p$.    A trivial variant of our construction arises from taking mod-$p$ reductions of integer-valued cocycles; a possible upshot is that computation over these $\Z/p$-vector spaces may be easier than over $\Z$.
There are however more nuanced variants.  First of all, as mentioned in Section \ref{NontrivialTorsion}, we should drop the relation \eqref{GraphsWithOrRevAutoZero}, which kills some classes that are nontrivial 2-torsion over $\Z$.  
Second, while mod-$p$ reduction gives a map on cocycles  $Z^*(\LD;\, \Z) \to Z^*(\LD;\, \Z/p)$, it is not surjective, as shown by Example \ref{DogolazkyExample} below for $p=2$.\footnote{The mod-$p$ reduction map is however injective on the subset of minimal cocycles, as guaranteed by condition (2) in the definition of minimality (see Definitions \ref{SimplifiedAndMinimalDefs}).}
We discuss the case $p=2$ in Section \ref{Mod2Subsection} and the case of odd $p$ in Section \ref{ModPSubsection}.

\subsection{The case of $p=2$}
\label{Mod2Subsection}  
Suppose $\gamma= \sum c_i \Gamma_i \in Z^*(\LD;\,\Z/2)$ is a mod-2 cocycle.  That means that every isomorphism class of principal face in the associated union of configuration space bundles $\coprod_i c_i \tX[\Gamma_i]$ appears an even number of times.  Topologically, we can glue these faces in pairs, without regard to orientation.  The remaining faces can be folded, collapsed, or relegated to the degenerate locus, as in Sections \ref{HiddenFaceInvolution}, \ref{Collapses}, and \ref{FundClass}.  The resulting pair $(X_\gamma, \d X_\gamma)$ will have a fundamental class over $\Z/2$, as in Propositions \ref{GluedMfdFundClass} and \ref{FundClassProp}.  

It will also have a well defined map $\Phi_\gamma$ to a quotient $Y:=(S^{d-1})^N/G'$ by the action of a group $G'$ generated by some permutations, antipodal maps, and reflections.  (We name it $G'$ to avoid confusion with $G=G_o,G_e$ given in Definition \ref{GDefinitions}, though certainly we may consider $G'=G$.)  
The group $G'$ is thus required to satisfy the analogue of Proposition \ref{MapToSpheresP}.  Ideally, it should be the smallest group for which this analogue holds, as explained below, and we will not specify $G'$ any further.   The  analogue of Proposition \ref{GeneratorOfSpheres} with $G$ replaced by $G'$ holds, using the same proof.  A  generator $\omega^{(N)}_{G'}$ of the $\Z$ summand of $H^{N(d-1)}((S^{d-1})^N/G'; \, \Z)$ reduces modulo 2 to a nontrivial class.
Let  $\omega^{(N)}_{G',2}$ be the lift of this mod-2 class to a relative class in $H^{N(d-1)}((S^{d-1})^N/G', \mathcal{R}; \, \Z/2)$.  The following result is thus a straightforward adaptation of the Main Theorem.

\begin{proposition}
Let $G'$ and $\omega^{(N)}_{G',2}$ be as defined above.
Let $\gamma = \sum_i c_i \Gamma_i \in Z^*(\LD^{k,n}; \,\Z/2)$.
\begin{enumerate}
\item
There is a class $\pi^! (\Phi_\gamma^*(\omega^{(N)}_{G',2})) \in H^{n(d-3)+k}(\L^d_m;\, \Z/2)$ obtained via the pullback $\Phi_\gamma^*$ followed by pushforward $\pi^!$ from the Serre spectral sequence.  
\item
If $\gamma$ is the mod-2 reduction of a cocycle over $\Z$ and $G$ is the group defined in Definition \ref{GDefinitions}, then $\pi^! (\Phi_\gamma^*(\omega^{(N)}_{G',2}))$ is the mod-2 reduction of $r$ times the configuration space integral class associated to $\gamma$, where $r=N!\, 2^{N-1}$ if $d$ is odd and $r=N!\, 2^{2N-2}$ if $d$ is even.
\qed
\end{enumerate}
\end{proposition}

We next describe a potential pitfall of using too large a $G'$ while working modulo $p$ in the case $p=2$.  The fact that one can ignore orientations modulo 2 may lead one to take $G'$ to be $\Sigma_N \wr \Z/2$, the group generated by all permutations and antipodal maps.  That is, one could try to set $Y=\mathrm{Sym}^N(\R \mathrm{P}^{d-1})$, since in this case $H^{N(d-1)}(Y;\, \Z/2)$ has nontrivial cohomology.\footnote{The $N$-fold cross product of a generator of $H^{d-1}(S^{d-1};\, \Z/2)$ corresponds to a generator of $H^{N(d-1)}(Y;\, \Z/2)$.  This can be deduced from a theorem of Nakaoka \cite[Corollary 3]{Nakaoka1959} on symmetric powers of spheres, using the $N$-fold product of the identity Steenrod operation, and then passing from spheres to projective spaces by the Serre spectral sequence.}
The problem is that the map $X_\gamma \to Y$ may potentially factor through a map $X_\gamma \to Z$, where $Z = (S^{d-1})^N/H$ and $H<G'$.  The map $H^*(Y;\,\Z/2) \to H^*(Z;\,\Z/2)$ is given by multiplication by $|G'/H|$, so the map $X_\gamma \to Y$ will be identically zero in $\Z/2$ cohomology if the index of $H$ is even.  For example, if we take $d$ odd and attempt the same construction as in the Main Theorem, but over $\Z/2$ and with $G'=\Sigma_N \wr \Z/2$, every resulting class will be zero, since  $\Phi_\gamma$ factors through the quotient by $G_o$, which has index 2 in $\Sigma_N \wr \Z/2$.

This consideration raises the question of whether the factors $r_{\mathrm{odd}}=|G_o|=N!\,2^{N-1}$ and $r_{\mathrm{even}}=|G_e|=N!\,2^{2N-2}$  in Theorem \ref{MainTheorem} are sharp.  That is, for a given cocycle, is there a modification of the construction where one uses a smaller group $G'$?  
Notice that the existence of a nontrivial class over $\Z/2$ by the above method with $|G'| \equiv 0 \mod 2$  would not immediately produce a contradiction.  Indeed, there are maps 
\[
X_\gamma \overset{\Phi_\gamma}{\longrightarrow} (S^{d-1})^N/G \overset{q}{\longleftarrow} (S^{d-1})^N
\]
but no obvious factoring of $\Phi_\gamma$ through $q$.  (The construction in de Rham theory is roughly such a factoring, but not over $\Z$.)  Being able to carry out the construction with a given group $G'$ can rule out nontriviality over $\Z/2$ of the construction with any larger group.  Equivalently, finding a nontrivial class over $\Z/2$ using a given group $G'$ would establish the sharpness of $G'$.  Although we do not need orientations to get a fundamental class over $\Z/2$, it is clear that the labelings of the graphs should be useful in finding a sharp bound for the group $G'$.

We conclude this Subsection with one explicit potential source of classes over $\Z/2$ that would not come from classes over $\Z$:

\begin{example}
\label{DogolazkyExample}
In defect 0, there is a graph cocycle of order 5 over $\Z/2$ which is not the mod-2 reduction of a cocycle over $\Z$.  This cocycle was found by Dogolazky \cite{Dogolazky} and further studied by Stanford \cite{Stanford-mod2}.
The chord diagram terms in this cocycle are
\[
\raisebox{-3.5pc}{\includegraphics[scale=0.3]{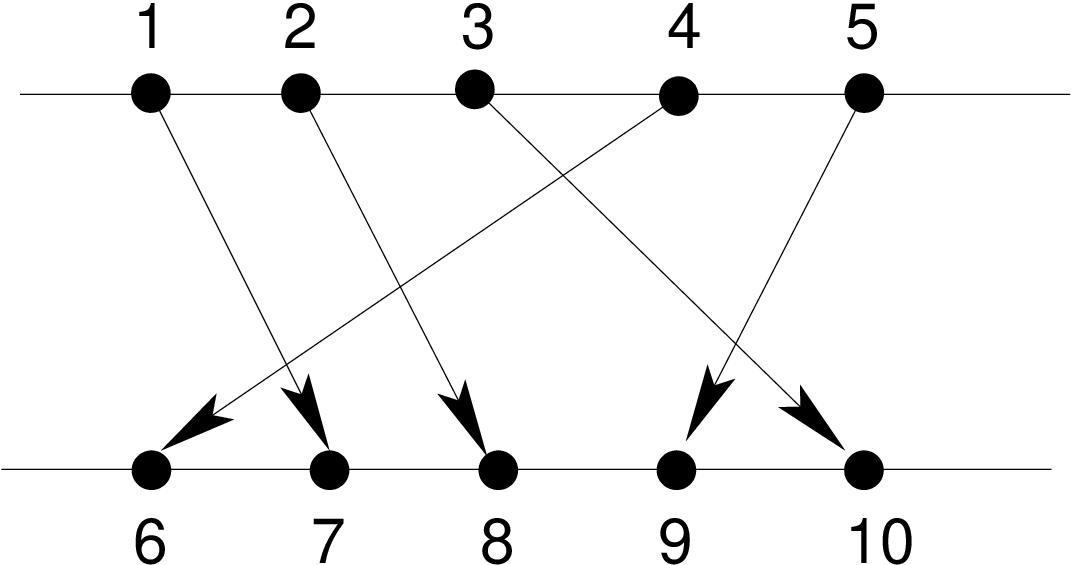}} \qquad + \qquad
\raisebox{-3.5pc}{\includegraphics[scale=0.3]{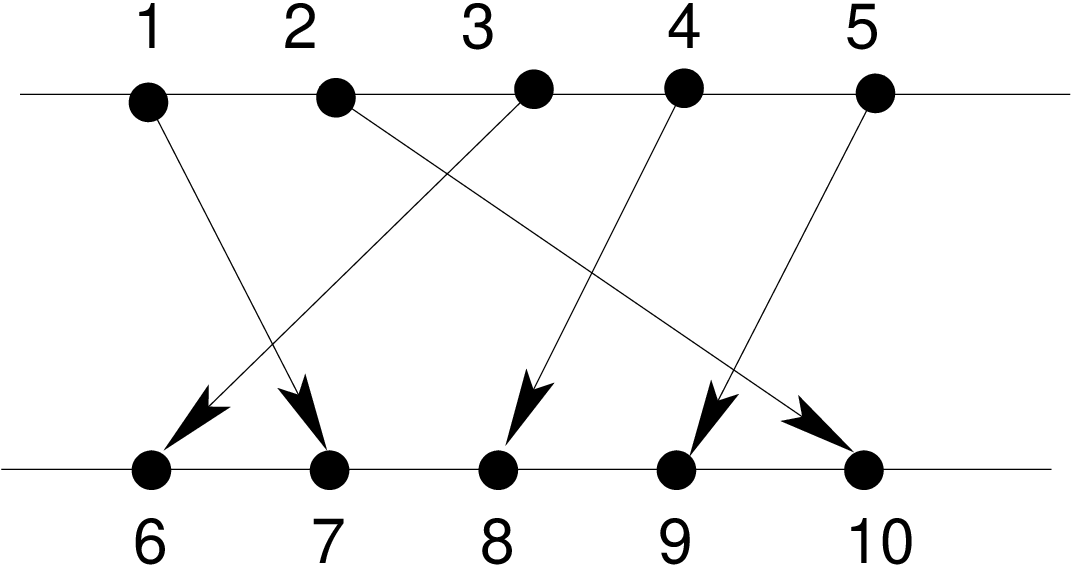}}
\]
To find the explicit expression for the corresponding graph cocycle, one would use the STU relation repeatedly to determine the coefficient for every (uni)trivalent graph on two segments with ten vertices.  
Stanford showed that the cocycle described above ``integrates'' to zero, by considering topological relations that an invariant giving rise to it would have to satisfy.  It would be interesting to see if our construction (applied with a suitably small group of symmetries $G'$) produces a nontrivial cohomology class in $H^*(\Emb(\R \sqcup \R, \R^d);\,\Z/2)$ for odd $d \geq 3$.
\qed
\end{example}

\subsection{The case of odd primes $p$}
\label{ModPSubsection}
For odd $p$, we can also modify our construction to give classes over $\Z/p$ that need not be mod-$p$ reductions of classes over $\Z$.  In this case, we would identify a principal face with \emph{all} the other principal faces in its unoriented isomorphism class, as in the construction of Kuperberg and Thurston \cite{KuperbergThurston}.  
For $p\neq 2$, signs matter, so we certainly need to keep track of orientations of graphs.
Though we glue multiple principal faces at a time, we treat the signs exactly as we did over $\Z$ in Section \ref{GluingSubsection} and in particular Lemma \ref{GluingInteriorsOfPrincipalFaces}.
The folding, collapsing, and relegation of the degenerate locus of the appropriate faces proceed exactly as in Sections \ref{HiddenFaceInvolution}, \ref{Collapses}, and \ref{FundClass}.  
For a mod-$p$ cocycle $\gamma$, the resulting space $X_\gamma$ has a fundamental cycle relative to the degenerate locus $\d X_\gamma$ because the boundary contribution at each glued face will be a multiple of $p$.
We map $X_\gamma$ to a quotient $Y=(S^{d-1})^N/G'$, where $G'$ is a group that is as small as possible, as in the case of $p=2$.  Again, a group that is too large will result in only trivial classes.
As over $\Z/2$, we can find a generator $\omega^{(N)}_{G'}$ of $H^{N(d-1)}((S^{d-1})^N/G'; \, \Z)$ that reduces to a nontrivial class over $\Z/p$.  Let $\omega^{(N)}_{G',p}$ be the lift of this class to a class in $H^{N(d-1)}((S^{d-1})^N /G, \mathcal{R}; \, \Z/p)$.

\begin{proposition}
Let $G'$ and $\omega^{(N)}_{G',p}$ be as defined above.
Let $\gamma = \sum_i c_i \Gamma_i \in Z^*(\LD^{k,n}; \,\Z/p)$.  Then there is a class $\pi^! (\Phi_\gamma^*(\omega^{(N)}_{G',2})) \in H^{n(d-3)+k}(\L^d_m;\, \Z/p)$ obtained via the pullback $\Phi_\gamma^*$ followed by pushforward $\pi^!$ from the Serre spectral sequence.  
\end{proposition}

It may be interesting to apply the above construction in connection with mod-$p$ {cycles} appearing in work of Turchin \cite{Victor-DyerLashof} and, for the case $d=3$, work of Budney and F.~Cohen \cite{Budney-Cohen}.

    \bibliographystyle{plain}
    \bibliography{refs}

\end{document}